\documentclass[notitlepage,11pt,reqno,letterpaper]{amsart}
\usepackage{amssymb,amsmath,amsthm,color,tocvsec2, mathtools,mathrsfs,cancel,listings}

\usepackage[breaklinks=true]{hyperref}%To get clickable links to papers
\usepackage[letterpaper]{geometry}

\newcommand{\edit}[1]{{\color{black}#1}} %% for edits
\newcommand{\commen}[1]{}

\newcommand{\wt}{\operatorname{wt}}
\newcommand{\conv}{\operatorname{conv}}
\newcommand{\supp}{\operatorname{supp}}
\newcommand{\ch}{\operatorname{ch}}

\newcommand{\Z}{\ensuremath{\mathbb{Z}}}
\newcommand{\bff}{\ensuremath{\mathbf{f}^{(\lambda)}}}

\renewcommand{\subset}{\subseteq}
\renewcommand{\leq}{\leqslant}
\renewcommand{\geq}{\geqslant}

\newtheorem{utheorem}{\textrm{\textbf{Theorem}}}

\newtheorem{theorem}{Theorem}[section]
\newtheorem{cor}[theorem]{Corollary}
\newtheorem{lemma}[theorem]{Lemma}
\newtheorem{prop}[theorem]{Proposition}

\theoremstyle{definition}
\newtheorem{defn}[theorem]{Definition}
\newtheorem{remark}[theorem]{Remark}
\newtheorem{example}[theorem]{Example}
\newtheorem{question}[theorem]{Question}

\numberwithin{equation}{section}
\numberwithin{figure}{section}

\begin{document}

\vspace*{-7mm}

\title[Weight-formula for highest weight modules, higher order
parabolic category $\mathcal{O}$]{A weight-formula for all highest weight
modules, and\\ a higher order parabolic category $\mathcal{O}$}

\author{Apoorva Khare}
\address[A.~Khare]{Department of Mathematics, Indian Institute of
Science, Bangalore, India}
\email{\tt khare@iisc.ac.in}

\author{G.\ Krishna Teja}
\address[G.~Krishna Teja]{Statistics and Mathematics Unit, Indian
Statistical Institute, Bangalore, India}
\email{\tt tejag@alum.iisc.ac.in}

\date{\today}



\begin{abstract}
Let $\mathfrak{g}$ be a complex Kac--Moody algebra, with Cartan
subalgebra $\mathfrak{h}$. Also fix a weight $\lambda \in
\mathfrak{h}^*$.
For $M(\lambda) \twoheadrightarrow V$ an arbitrary highest weight
$\mathfrak{g}$-module, we provide a cancellation-free, non-recursive
formula for the weights of $V$. This is novel even in finite type, and is
obtained from $\lambda$ and a collection $\mathcal{H} = \mathcal{H}_V$ of
independent sets in the Dynkin diagram of~$\mathfrak{g}$ that are
associated to $V$.

Our proofs use and reveal a finite family (for each $\lambda$) of
``higher order Verma modules'' $\mathbb{M}(\lambda, \mathcal{H})$ --
these are all of the universal modules for weight-considerations. They
(i)~generalize and subsume parabolic Verma modules $M(\lambda,J)$, and
(ii)~have pairwise distinct weight-sets, which exhaust the weight-sets of
all modules $M(\lambda) \twoheadrightarrow V$. As an application, 
we explain the sense in which the modules $M(\lambda)$ of Verma and
$M(\lambda, J_V)$ of Lepowsky are respectively the zeroth and first order
upper-approximations of every $V$, and continue to higher order
upper-approximations $\mathbb{M}_k(\lambda, \mathcal{H}_V)$. We
also determine the $k$th order integrability and lower-approximation of
$V$, for all $k \geq 0$.

We then introduce the category $\mathcal{O}^{\mathcal{H}} \subset
\mathcal{O}$, which is a higher order parabolic analogue that contains
the higher order Verma modules $\mathbb{M}(\lambda, \mathcal{H})$. We
show that $\mathcal{O}^{\mathcal{H}}$ has enough projectives, and also
initiate the study of BGG reciprocity, by proving it for all
$\mathcal{O}^{\mathcal{H}}$ over $\mathfrak{g} = \mathfrak{sl}_2^{\oplus
n}$. Finally, we provide a BGG resolution for {\color{black}our} universal modules
$\mathbb{M}(\lambda, \mathcal{H})$ in certain cases {\color{black}including $A_3$ and any decomposable rank 3 $\mathfrak{g}$}; this yields a
Weyl-type character formula for them, which involves the action of a
parabolic Weyl semigroup.
\end{abstract}
\maketitle

%\vspace*{-5mm}

\settocdepth{section}
\tableofcontents

%{{{1 Section 1 - Introduction
\section{Introduction}

Throughout this paper, and unless otherwise specified, we work over
$\mathbb{C}$, with
$\mathfrak{g}$ denoting an arbitrary Kac--Moody
algebra,\footnote{More precisely, we work with any Lie algebra
$\widetilde{\mathfrak{g}} \twoheadrightarrow \mathfrak{g}
\twoheadrightarrow \overline{\mathfrak{g}}$ for a given generalized
Cartan matrix. Thus $\mathfrak{g}$ lies in between
$\widetilde{\mathfrak{g}}$ generated purely by the Chevalley--Serre
relations, and the quotient $\overline{\mathfrak{g}}$ of
$\widetilde{\mathfrak{g}}$ by the maximal ideal trivially intersecting
$\mathfrak{h}$. Our results hold over all such intermediate Lie algebras
$\mathfrak{g}$.}
$U \mathfrak{g}$ its universal enveloping algebra,
$\mathfrak{h} \subset \mathfrak{g}$ a fixed Cartan subalgebra, and
$\lambda \in \mathfrak{h}^*$ an arbitrary (highest) weight.
As further notation: denote by $\Delta$ the root system,
$\Pi = \{ \alpha_i : i \in I \}$ a base of simple roots indexed by nodes
$I$,
$\{ e_i, f_i, \alpha_i^\vee : i \in I \}$ a set of Chevalley generators,
and $W$ the Weyl group generated by the simple reflections $\{ s_i : i
\in I \}$.
We will identify subsets $J \subset I$ with the corresponding Dynkin
sub-diagrams of the diagram on $I$ for $\mathfrak{g}$.

The reader immediately interested in the main results can skip directly
to Section \ref{S2}.

\subsection{Characters}

The study of semisimple, affine, and Kac--Moody Lie algebras
$\mathfrak{g}$ and their representations is a prominent theme in
mathematics, from early work by Cartan, Killing, and Weyl to the
Langlands program in recent times, with numerous other connections and
applications. A central question involves understanding the structure of
(simple) highest weight $\mathfrak{g}$-modules. In this work we focus on
their characters and associated information.

We begin when $\mathfrak{g}$ is simple and finite-dimensional. If the
highest weight $\lambda \in \mathfrak{h}^*$ is dominant integral --
denoted $\lambda \in P^+$ -- the character of the corresponding simple
module $L(\lambda)$ is given by the celebrated formula of Weyl
\cite{Weyl} (and variants by Freudenthal and others). In standard
notation:
\begin{equation}\label{EWeyl}
\lambda \in P^+ \quad \implies \quad \ch L(\lambda) = \sum_{w \in W}
\frac{(-1)^{\ell(w)} e^{w \bullet \lambda}}{\prod_{\alpha \in \Delta^+}
(1 -e^{-\alpha})},
\end{equation}
with $\bullet$ the dot-action.
In contrast, when $\lambda$ is ``generic'', the module itself is
a Verma module $M(\lambda)$ \cite{Ver}, with a transparent character
formula (related to the Kostant partition function \cite{Kos1}):
\begin{equation}\label{EVerma}
\lambda \in \mathfrak{h}^* \quad \implies \quad \ch M(\lambda) = 
\frac{e^\lambda}{\prod_{\alpha \in \Delta^+} (1 - e^{-\alpha})}.
\end{equation}

For arbitrary highest weights, one uses Kazhdan--Lusztig theory
\cite{BeBe,BrKa,KL1,Soergel} to write down the character. For instance,
if $\lambda$ is dominant integral and $w \in W$, then we have the simple
character
\begin{equation}\label{EKL}
\ch L(ww_\circ \bullet \lambda) = \sum_{x \leqslant w} (-1)^{\ell(w) -
\ell(x)} P_{x,w}(1) \ch M(xw_\circ \bullet \lambda),
\end{equation}
where $P_{x,w}$ denotes the relevant Kazhdan--Lusztig polynomial.
Notice that computing weight multiplicities -- or even the easier
question of which weights occur -- using these formulas for $L(\lambda)$
is hard for two reasons:
(a)~the presence of signs, leading to cancellations, and
(b)~furthermore for non-integrable modules, the recursive nature of
Kazhdan--Lusztig polynomials.

If $\mathfrak{g}$ is of infinite type, less is known. For symmetrizable
$\mathfrak{g}$, one uses the Weyl--Kac character formula, but character
formulas are not known for \textit{all} highest weights in affine type --
and indeed, simple modules with highest weight $\lambda$ at critical
level behave very differently from those with $\lambda$ at non-critical
level (see e.g.\ \cite{FG}).
For non-symmetrizable $\mathfrak{g}$, even the first step above is
challenging, i.e.\ it remains open if the maximal integrable module (for
$\lambda$ dominant integral) is simple.

Clearly, understanding arbitrary highest weight modules (i.e., quotients
of Verma modules) is harder -- even for $\mathfrak{g}$ of finite type,
hence for arbitrary Kac--Moody $\mathfrak{g}$.

\subsection{Weights}

We now turn to the theme of the present work. Closely associated to the
``quantitative'' Weyl character formula is a ``qualitative'' picture,
which was known from the outset -- the easier question of determining the
weights (i.e.\ ignoring multiplicities). As is folklore: the set of
weights of a simple finite-dimensional highest weight module $L(\lambda)$
is $W$-invariant with convex hull the polytope with vertices
$W(\lambda)$, and the weights are recovered by intersecting with the
$\lambda$-translate of the root lattice. A similar statement holds for
integrable $L(\lambda)$ over Kac--Moody $\mathfrak{g}$.\footnote{See also
the recent work \cite{BJK}, where the authors extend this to all Demazure
modules for classical simple $\mathfrak{g}$.}

The uniformity of this description turns out to hold more generally.
Recently in \cite{DK2,DK,Kh1}, Dhillon and Khare proved several positive
formulas for the weights of $L(\lambda)$ for arbitrary (including
non-integrable) simple modules over all Kac--Moody $\mathfrak{g}$. In
contrast to the above story for characters, these weight-formulas hold
uniformly, for all highest weights and across all types (for
$\mathfrak{g}$). One of these formulas exactly generalizes the above
result in terms of convex hulls (always in $\mathfrak{h}^*$): now one
works with a $W_J$-invariant polyhedral shape rather than a $W$-invariant
one, corresponding to the partial integrability $J \subset I$ of (the not
necessarily fully-integrable module) $L(\lambda)$.

We now present one of these formulas; this serves to motivate our main
result, as well as to introduce some of the necessary notation. As this
result -- and our paper -- makes extensive use of parabolic Verma modules
\cite{GL,lepo1}, we begin by setting notation for them.

\begin{defn}\label{Dpvm}
Given a subset $S \subset \mathbb{R}$ and subsets $X,Y$ of a real vector
space, $SX$ will denote the set of finite $S$-linear combinations of
elements of $X$, with the empty sum denoting $0$. Moreover,
\[
X \pm Y := \{ x \pm y : x \in X, \ y \in Y \}, \qquad X \setminus Y := \{
x \in X : x \not\in Y \}
\]
will denote the Minkowski sum and difference, and set difference,
respectively.

For $J \subset I$, define $\Pi_J := \{ \alpha_j : j \in J \}$ and
$\Delta_J = \Delta_J^+ \sqcup \Delta_J^-$ to be $\Delta \cap \Z \Pi_J$.
Now let the Levi subalgebra
$\displaystyle \mathfrak{l}_J := \mathfrak{h} + \bigoplus_{\alpha \in
\Delta_J} \mathfrak{g}_\alpha$,
and let $\mathfrak{g}_J := \mathfrak{g}(A_{J \times J})$, where $A
= A_{I \times I}$ is the generalized Cartan matrix for $\mathfrak{g}$.
Also fix a (non-canonical) realization of $\mathfrak{g}_J$ as a
subalgebra of $\mathfrak{g}$. Next, for an $\mathfrak{h}$-module $V$,
denote by $\wt V := \{ \mu \in \mathfrak{h}^* : V_\mu \neq 0 \}$ its
\textit{set of weights}, where $V_\mu := \{ v \in V : h \cdot v = \mu(h)
v\ \forall h \in \mathfrak{h} \}$. E.g.\ for $\mathfrak{a} =
\mathfrak{g}, \mathfrak{g}_J, \mathfrak{l}_J$ (over ${\rm ad}\,
\mathfrak{h}$) and $\alpha \in \Delta$, $\mathfrak{a}_\alpha$ denotes the
$\alpha$-root space.

Given $\lambda \in \mathfrak{h}^*$, define its {\em integrability}
$J_\lambda$ to be:
\begin{equation}
J_\lambda := \{ i \in I \ | \ \langle \lambda, \alpha_i^\vee \rangle
\in \Z_{\geq 0} \},
\end{equation}
where $\langle \cdot , \cdot \rangle$ denotes the evaluation map $:
\mathfrak{h}^* \times \mathfrak{h} \to \mathbb{C}$. For $J \subset
J_\lambda$, define the \textit{parabolic Verma module} $M(\lambda,J) :=
{\rm Ind}^{\mathfrak{g}}_{\mathfrak{p}_J} L_J^{\max}(\lambda)$, where
$\mathfrak{p}_J := \mathfrak{l}_J + \mathfrak{n}^+$ is a parabolic
subalgebra of $\mathfrak{g}$, and $L_J^{\max}(\lambda)$ is the maximal
integrable highest weight $\mathfrak{l}_J$-module, which is given a
$\mathfrak{p}_J$-module structure via
$(\mathfrak{p}_J)_\alpha \cdot L_J^{\max}(\lambda) = 0$ for $\alpha
\not\in \Delta_J$. Note that $M(\lambda, J) \cong U \mathfrak{g}
\otimes_{U \mathfrak{p}_J} L_J^{\max}(\lambda)$.
\end{defn}\medskip

Both the ``quantitative'' and ``qualitative'' pictures are well known for
parabolic Verma modules. For the former, we mention a variant (see e.g.\
\cite{DK}) that extends the Atiyah--Bott version of the Weyl--Kac
character formula \cite{atbo} and hence subsumes Equations~\eqref{EWeyl}
and~\eqref{EVerma}. Namely, for an arbitrary parabolic Verma module over
Kac--Moody $\mathfrak{g}$, one has
\begin{equation}\label{EAtiyahBott}
\ch M(\lambda,J) = \sum_{w \in W_J} \frac{(-1)^{\ell(w)} e^{w \bullet
\lambda}}{\prod_{\alpha \in \Delta^+} (1 - e^{-\alpha})^{\dim
\mathfrak{g}_\alpha}}, \qquad \forall \lambda \in \mathfrak{h}^*, \ J
\subset J_\lambda,
\end{equation}
where $W_J$ is the corresponding parabolic Weyl subgroup.

For the latter picture, one has Minkowski difference
formulas obtained from parabolic induction:
\begin{align}\label{Epvm}
\begin{aligned}
\wt M(\lambda, J) = &\ \wt L_J^{\max}(\lambda) - \Z_{\geq 0} (\Delta^+
\setminus \Delta_J^+)\\
= &\ ((\lambda - \Z_{\geq 0} \Delta^+) \cap \conv W_J(\lambda) )
- \Z_{\geq 0} (\Delta^+ \setminus \Delta_J^+).
\end{aligned}
\end{align}

Resuming the above discussion, we write a positive weight-formula for
simple modules:

\begin{theorem}[Khare~\cite{Kh1}, Dhillon--Khare~\cite{DK}]\label{Tdk}
Let $\mathfrak{g}$ and $\lambda \in \mathfrak{h}^*$ be arbitrary. Then,
\begin{equation}\label{EdkL}
\wt L(\lambda) = \wt M(\lambda, J_\lambda).
\end{equation}
\end{theorem}

Given the uniformity of this weight-formula for all $L(\lambda)$,
$\lambda \in \mathfrak{h}^*$ (via \eqref{Epvm}), our original goal in
this work was a more challenging result: a positive formula for the
weights of an arbitrary highest weight module $V$. We make a few remarks
here, addressing the treatment and proofs below.

First, a weight-formula for $V$ was unknown beyond simple and parabolic
Verma modules, even in finite type. We provide a uniform, positive
formula for all highest weight modules over all $\mathfrak{g}$.
Perhaps one ``miracle'' here is that such a formula exists in the first
place, and it uses simply the Dynkin diagram and the images in $V$ of
some lines in $M(\lambda)$ killed by $\mathfrak{n}^+$! See
Theorem~\ref{T1}.

Second is the ``next'' question, of characters. Our quest for a formula
for $\wt V$ also yields rewards on this side. Even more: we discovered a
novel (to our knowledge) family of ``higher order Verma modules''
$\mathbb{M}(\lambda, \mathcal{H})$ which ``cover'' {\color{black}all highest weight
modules $V$ -- one for each weight-set $\wt V$! --} and subsume parabolic Verma modules
$M(\lambda,J)$. 
(Over $\mathfrak{sl}_2^{\oplus n}$, these modules $\mathbb{M}(\lambda,
\mathcal{H})$ comprise \textit{all} highest weight modules.)
A question of future interest is to study the family $\mathbb{M}(\lambda,
\mathcal{H})$, starting with their characters and BGG-type resolutions
(we begin this study by obtaining such a resolution in all {\color{black} decomposable rank 3 cases of $\mathfrak{g}$  and in type $A_3$}, in
Section~\ref{Sfinal}), as well as their geometric counterparts via
localization.

{\color{black}Curious byproducts of our weight-based approach are : 1)~A} family of Minkowski
difference formulas for all parabolic Verma modules.
 These
subsume~\eqref{Epvm} as a special case, and interestingly, do not hold in
general on the {\color{black}module level} via parabolic induction
(except~\eqref{Epvm}).
{\color{black}2)~A family of simples whose weight-sets patch up to the whole $\wt V$ for each highest weight $\mathfrak{g}$-module $V$, strengthening the ``local Jordan--H\"{o}lder series'' \cite[Lemma 9.6]{Kac book} of $V$; yielding formulas for $\wt V$.}

\subsection{Holes{\color{black}, and motivations behind them}}
We lead up to our main results by introducing -- by example -- another
key ingredient.

As mentioned above, Dhillon and Khare \cite{DK,Kh1} showed that for all
$\lambda \in \mathfrak{h}^*$, the convex hull of weights of a simple
highest weight $\mathfrak{g}$-module $L(\lambda)$ recovers $\wt
L(\lambda)$, by intersecting a suitable $W_{J_\lambda}$-invariant shape
with the root lattice-translate $\lambda + \Z \Delta$. However, the same
does not hold for all highest weight modules $V$. 
For instance if $\mathfrak{g} = \mathfrak{sl}_4$ and $\langle
\lambda, \alpha_i^\vee \rangle = 1, 0$ for $i=1,3$, then $\wt
\left( M(\lambda) / U \mathfrak{g} \cdot f_1^2 f_3 \cdot
M(\lambda)_\lambda \right) \subsetneq \wt M(\lambda)$, but their convex
hulls are equal. The key fact underlying this is an even simpler example
over the non-simple Lie algebra $\mathfrak{g} = \mathfrak{sl}_2 \oplus
\mathfrak{sl}_2$. Namely, the module
\begin{equation}\label{EM00}
V_{00} := M(0,0) / M(-2,-2)
\end{equation}
has a ``hole'' inside the hull: its weights are precisely the lattice
points along the boundary of $\conv (\wt V_{00})$,
i.e., $-\Z_{\geq 0} \alpha_1 \cup -\Z_{\geq 0} \alpha_2$,
and all interior lattice points $(-2 \Z_{>0})^2$ lie in $\conv (\wt
V_{00})$ but not in $\wt V_{00}$. This simple example is at the heart of
progress in multiple directions, below.

More generally, holes can occur in a $\mathfrak{g}$-module $M(\lambda)
\twoheadrightarrow V$ as follows. Suppose $\{ \alpha_h : h \in H \}$ is a
set of pairwise ``orthogonal'' roots such that $\langle \lambda,
\alpha_h^\vee \rangle \in \Z_{\geq 0}$ for all $h$ -- i.e., $H$ is an
independent subset of $J_\lambda$. If $f_h \in \mathfrak{g}_{-\alpha_h}$
denotes a Chevalley generator, then applying the lowering
operator-product
\[
\prod_{h \in H} f_h^{\langle \lambda, \alpha_h^\vee \rangle + 1}
\]
to the highest weight line $V_\lambda$ can \textit{sometimes} yield
zero.\footnote{We will index simple roots in such ``holes'' by $h \in H
\subset I$. This index should not be confused with elements of
$\mathfrak{h}$.} Whenever this happens, letting $\mathfrak{l}_H :=
\mathfrak{sl}_2^{\oplus H} + \mathfrak{h}$ denote the corresponding Levi,
the set $\wt U(\mathfrak{l}_H) V_\lambda$ comprises the $\lambda$-shifted
root-lattice points along a ``thickening'' of the boundary of its convex
hull -- i.e., there is a hole in the interior of the convex hull.
Clearly, whether or not this happens depends on 
(a)~the highest weight module $V$, and
(b)~the independent subset $H \subset J_\lambda$ of nodes.
The latter also shows that there are at most finitely many holes in $V$,
each corresponding to a {\color{black}maximal} one-dimensional weight space of $M(\lambda)$ that
consists of maximal vectors (i.e., vectors annihilated by all raising
operators $e_i$).
{\color{black} Conversely, all the above products that do not kill $V_{\lambda}$-line (which are the non-relations in $V$ that) determine $\wt V$.
In addition to these two perspectives or importance of holes (broadly, weights supported over independent sets - lost and surviving) in our study, we have from the previous work \cite{MDWF} of the second author : Raising operator actions (or integrability theory) on weight spaces lead to -- i.e., $\mathfrak{n}^+ V_{\mu}$ for each $\mu\in \wt V$ contain -- non-hole 
1-dim. weights.}
%}}}

%{{{1 Section 2 - Main results
\section{Main results}\label{S2}

In all results below, $V$ denotes a general nonzero highest weight
$\mathfrak{g}$-module over an arbitrary Kac--Moody algebra
$\mathfrak{g}$, with a general (fixed) highest weight $\lambda \in
\mathfrak{h}^*$.
The reader can go through the entire work assuming $\mathfrak{g} =
\mathfrak{sl}_{n+1}$, or $\mathfrak{g}$ semisimple, without losing (most
of) the novelty in the work.

With the motivation and background given above, we begin by formalizing
the above notion of holes. In what follows, recall $J_\lambda := \{ i \in
I \ | \ \langle \lambda, \alpha_i^\vee \rangle \in \Z_{\geq 0} \}$.

\begin{defn}\label{Dholes}
For a module $M(\lambda) \twoheadrightarrow V$ over Kac--Moody
$\mathfrak{g}$, the {\bf set of holes} in $V$ or $\wt V$ is:
\begin{equation}
\mathcal{H}_V := \left\{ H \subset J_\lambda \ \bigg| \text{ the Dynkin
subdiagram on } H \text{ has no edge and } \left( 
\prod_{h \in H} f_h^{\langle \lambda, \alpha_h^\vee \rangle + 1} \right)
V_\lambda = 0 \right\}.
\end{equation}
\end{defn}

We take the empty product to be $1 \in U \mathfrak{g}$ here, so that
$\emptyset \in \mathcal{H}_V$ if and only if $V=0$.

\begin{remark}\label{RT1}
Explicitly, the set of holes $\mathcal{H}_V$ in $V$ consists of
precisely those independent sets $H \subset J_\lambda$, for which
defining $\lambda_H = \lambda - \sum_{h \in H} (\langle \lambda,
\alpha_h^\vee \rangle + 1 ) \alpha_h = \prod_{h \in H} s_h \bullet
\lambda$, one has:
(i)~the weight space $M(\lambda)_{\lambda_H}$ is one-dimensional (via
Kostant's function, this is equivalent to $H$ having no edge),
(ii)~this weight space consists of maximal vectors (via
$\mathfrak{sl}_2$-theory), and
(iii)~the same weight space in $V$ is $V_{\lambda_H} = 0$.
\end{remark}

Now our first result is a positive formula for $\wt V$. Recall,
$M(\lambda,J)$ is a parabolic Verma module.

\begin{utheorem}\label{T1}
Fix a complex Kac--Moody Lie algebra $\mathfrak{g}$, and $\mathfrak{h}$,
$\Delta, \Pi = \{ \alpha_i : i \in I \}$ as above.
Let $\lambda \in \mathfrak{h}^*$ and let $M(\lambda) \twoheadrightarrow
V$ be nonzero. Then
\begin{equation}\label{Eweights}
\wt V = \bigcup_{J \subset J_\lambda\, :\, J \cap H
\neq \emptyset\; \forall H \in \mathcal{H}_V} \wt M(\lambda,J)
\end{equation}
if $\mathcal{H}_V$ is nonempty. Otherwise, $\wt V = \wt M(\lambda)$.
\end{utheorem}

An immediate consequence of Theorem~\ref{T1} is the following ``geometric
combinatorial'' formula:

\begin{cor}
For all $\lambda \in \mathfrak{h}^*$ and nonzero modules
$M(\lambda) \twoheadrightarrow V$ with $\mathcal{H}_V$ nonempty,
\begin{equation}
\wt V = \bigcup_{J \subset J_\lambda\, :\, J \cap H \neq \emptyset\;
\forall H \in \mathcal{H}_V} \wt L_J^{\max}(\lambda) - \Z_{\geq 0}
(\Delta^+ \setminus \Delta^+_J).
\end{equation}
\end{cor}

\begin{remark}\label{Rintuit}
We provide some intuition behind the formula~\eqref{Eweights}. Suppose
$\wt M(\lambda, J) \subset \wt V$ for some $J \subset J_\lambda$. Then
one has an inclusion of holes (i.e., of the lost one-dimensional
weight spaces killed by $\mathfrak{n}^+$ -- or \textit{highest weight
lines}): $\mathcal{H}_V \subset \mathcal{H}_{M(\lambda, J)}$. By the
universality -- or the $U \left( \mathfrak{n}^-_{\Delta^- \setminus
\Delta_J^-} \right)$-freeness -- of $M(\lambda, J)$, the line $\prod_{h
\in H} f_h^{\langle \lambda, \alpha_h^\vee \rangle + 1} \cdot
M(\lambda)_\lambda$ being quotiented out of $M(\lambda, J)$ implies some
$h \in J$. Therefore $J \cap H \neq \emptyset\ \forall H \in
\mathcal{H}_V$ -- which explains the necessity of this condition in the
union in~\eqref{Eweights}. Theorem~\ref{T1} says that firstly, this
condition also guarantees $\wt M(\lambda, J) \subset \wt V$; and
secondly, such considerations recover \textit{all} weights of $V$.
\end{remark}

\begin{remark}[Alternate formulas]\label{Remark alternate wt-forms.}
For computational purposes, one can work with the subset
$\mathcal{H}_V^{\min}$ consisting of the minimal sets in $\mathcal{H}_V$
under inclusion -- i.e., replace respectively in~\eqref{Eweights}:
\[
\mathcal{H}_V \leadsto
\mathcal{H}_V^{\min},\qquad
J_\lambda \leadsto \cup_{H \in \mathcal{H}_V^{\min}} H.
\]
Indeed, we work with $\mathcal{H}_V^{\min}$ in Sections~\ref{SO}
and~\ref{Sfinal}; for now, observe that if $V = M(\lambda,J_0)$, then
working with minimal holes yields exactly one term on the right-hand side
in~\eqref{Eweights}: $J = J_0$.
Section~\ref{S5final} also provides two ``$k$th order'' weight-formulas
for $\wt V$ (for all $\mathfrak{g}, \lambda, V$, and $k \geq 1$), which
extend Theorem~\ref{T1} above and Theorem~\ref{T2} below. Furthermore,
there is yet another formulation of Equation~\eqref{Eweights} in terms of
the ``higher order parabolic category'' $\mathcal{O}^{\mathcal{H}_V}$ --
see~\eqref{Ealtwts}.
\end{remark}

We make two observations before proceeding further. First, the formula in
Theorem~\ref{T1} is clearly uniform across all types for $\mathfrak{g}$
and all highest weights $\lambda$. Second, in each such case, the formula
is visibly positive as well as non-recursive. 
These are in contrast to
the situation for characters, in which case one does not even have
conjectural formulas in all cases, or weight multiplicities even for all
integrable simple highest weight modules $L(\lambda)$ if $\mathfrak{g}$
is non-symmetrizable. 
{\color{black} Following Theorem \ref{T1} and to strengthen it, we solve in our third result Theorem \ref{Theorem wt-form. by composition series simples}, the below problem.}
\begin{question}\label{Question Theorem A converse}
   {\color{black} Given a weight $\mu\in \wt V$, for which sets $J$ in \eqref{Eweights} we have $\mu\in\wt M(\lambda, J)$?}
\end{question}\smallskip

Our second result shows the existence of a \textit{finite} collection of
``uniform'' highest weight modules $M(0) \twoheadrightarrow
\mathbb{M}_t$, such that for every pair -- $\lambda \in \mathfrak{h}^*$
and a module $M(\lambda) \twoheadrightarrow V$ -- there exists $t$ such
that
\[
\wt V = \wt L_{J_\lambda}^{\max}(\lambda) + \wt \mathbb{M}_t.
\]
In particular, $\wt V$ combines $\wt L_{J_\lambda}^{\max}(\lambda)$
(which is a fundamental object -- in fact, a parabolic Verma
$\mathfrak{l}_{J_\lambda}$-module) with the weights of some
$\mathbb{M}_t$ from a finite collection \textit{that works for all
$\lambda \in \mathfrak{h}^*$ and all $V$}. (That said, which
$\mathbb{M}_t$ to use does depend on $(\lambda,V)$, as we explain.)
%can take only a finite set of values, which are
%closely related across highest weights $\lambda$.
%We stress here that the modules $\mathbb{M}_t$ ``work'' for all weights
%$\lambda \in \mathfrak{h}^*$, yielding a uniform finite bound for all
%pairs $(\lambda, V)$.

To define these modules $\mathbb{M}_t$ and state our {\color{black}second weight-formula}, additional
notation is required.

\begin{defn}\hfill
\begin{enumerate}
\item Given a subset of simple roots (indexed by) $J \subset I$, define
${\rm Indep}(J)$ to comprise the collection of independent subsets $H
\subset J$, i.e.\ whose induced Dynkin subgraph in the Dynkin diagram of
$\mathfrak{g}$ has no edges. Note that $\emptyset$ and $\{ j \}$
are in ${\rm Indep}(J)$ for all $J \subset I$ and $j \in J$.

\item Given $J \subset I$ and a subset $\mathcal{H}$ of ${\rm Indep}(I)$,
define the highest weight $\mathfrak{g}$-module
\begin{equation}\label{EMH}
\mathbb{M}(\mathcal{H}) := \frac{M(0)}{\displaystyle \sum_{H \in
\mathcal{H}} U\mathfrak{g} \left( \prod_{h \in H} f_h \right) M(0)_0},
\end{equation}
where $M(0)_0$ is the highest weight line in the Verma module $M(0)$. If
$H = \emptyset$, define the empty product $\prod_{h \in H} f_h$ to be
$1$. (Thus, if $H = \emptyset \in \mathcal{H}$ then
$\mathbb{M}(\mathcal{H}) = 0$.)

\item Recall, a subset $\mathcal{H}$ of a poset $(\mathcal{P}, \preceq)$
is \textit{upper-closed} if
$H \in \mathcal{H}$, $H \preceq H' \text{ in } \mathcal{P}$ imply
$H' \in \mathcal{H}$.
\end{enumerate}
\end{defn}

The modules $\mathbb{M}(\mathcal{H})$ are studied in the next section;
they comprise the sought-for finite set $\{ \mathbb{M}_t \}_t$:

\begin{utheorem}\label{T2}
Fix $\lambda \in \mathfrak{h}^*$. Let $M(\lambda) \twoheadrightarrow V$
be nonzero, and $\mathcal{H}_V$ be as in Theorem~\ref{T1}. Then
$\mathcal{H}_V$ is a proper, upper-closed subset of ${\rm
Indep}(J_\lambda)$, and
\begin{equation}\label{ET2}
\wt V = \wt L_{J_\lambda}^{\max}(\lambda) + \wt \mathbb{M}(\mathcal{H}_V).
\end{equation}

\noindent Moreover, the finite collection of upper-closed subsets of
${\rm Indep}(J_\lambda)$ is in bijection with the set
$\{ \wt V : M(\lambda) \twoheadrightarrow V \}$, via $\Psi_\lambda :
\mathcal{H} \mapsto \wt L_{J_\lambda}^{\max}(\lambda) + \wt
\mathbb{M}(\mathcal{H})$. In particular, for $M(0)
\twoheadrightarrow V$ we have:
\begin{equation}\label{Etop}
\wt V = \wt \mathbb{M}(\mathcal{H}_V).
\end{equation}
\end{utheorem}

In fact, we extend~\eqref{Etop} to all modules $M(\lambda)
\twoheadrightarrow V$ (for all $\lambda \in \mathfrak{h}^*$) in the next
section. Thus, \eqref{Etop} and its extension reveal a second ``miracle''
about the sets $\wt V$ (cf.\ a few lines below Theorem~\ref{Tdk}): the
``obvious'' (see Remark~\ref{Rtop}) holes $\mathcal{H}_V$ in the
weight-set obtained from the top -- i.e.\ the line $V_\lambda$ -- are the
\textit{only} ones, for every highest weight module over every Kac--Moody
algebra.

\begin{remark}
Akin to the explicit weight-formula in Theorem~\ref{T1}, the finite
number of weight-sets for each $\lambda \in \mathfrak{h}^*$ (in
Theorem~\ref{T2}) is also in stark contrast to the situation for
characters. For instance, let $\mathfrak{g}$ be of infinite type, and
consider any sequence of increasing words in the Weyl group:
\[
w_1 := s_{i_1}, \quad w_2 := s_{i_2} s_{i_1}, \quad \dots\ ; \qquad
\ell(w_n) = n\ \forall n \geq 1.
\]
Then the modules $M(0) / M(w_n \bullet 0)$ have pairwise distinct
characters. Theorem~\ref{T2} nevertheless shows that they -- and all
other modules $M(0) \twoheadrightarrow V$ -- collectively yield only
finitely many weight-sets, $\wt \mathbb{M}(\mathcal{H})$.
(In particular, the weight-sets of $M(0) / M(w_n \bullet 0)$ eventually
stabilize.)
\end{remark}

\begin{remark}
It is natural to ask if~\eqref{ET2} specializes to the Minkowski
difference formula~\eqref{Epvm} for $V$ a parabolic Verma module. This is
not true on the nose; rather, in Proposition~\ref{Pminkowski} we exhibit
a \textit{family} of Minkowski difference formulas for each parabolic
Verma module, of which~\eqref{Epvm} is one extreme. An interesting
feature is: these formulas (except~\eqref{Epvm}) hold for weights, but do
\textit{not} lift in general to the level of parabolic induction.
Following this family of formulas in Proposition~\ref{Pminkowski}, we
explain how~\eqref{ET2} generalizes the Minkowski difference formula at
the other extreme to~\eqref{Epvm}.
\end{remark}
{\color{black} Our third weight-formula is more conceptual and emphizes the need to work with holes $\mathcal{H}_V$ for studying highest weight $\mathfrak{g}$-modules $V$.
While the above weight formulas and those in Subsection~\ref{S5final} are formulated using the lost-weights/holes in $V$, the one below uses crucially all the surviving non-holes.
Recall, the standard result \cite[Lemma 9.6]{Kac book} says each $V_{\mu}$ weight space is spread across the subquotients $\frac{V_i}{V_{i+1}}$ in a filtration $V_{k+1}\subset V_k \subset \cdots \subset V_1\subset V=V_0$ of $V$ with $\Big(\frac{V_i}{V_{i+1}}\Big)_{\mu}\neq 0\implies \frac{V_i}{V_{i+1}} = L(\lambda_i)$ for some $\lambda_i\preceq \lambda$.
Here ` $\preceq$ ' is the usual partial order on $\mathfrak{h}^*$. 
However it is not known in the literature (to the best of our knowledge), an example of such a simple $\frac{V_i}{V_{i+1}}=L(\lambda_i)$ containing any fixed weight $\mu\in \wt V$.
Solutions to this problem in Theorem \ref{Theorem wt-form. by composition series simples}(b) also solve Question \ref{Question Theorem A converse}.} 
\begin{utheorem}\label{Theorem wt-form. by composition series simples}{\color{black}
Let $\mathfrak{g}$ and $V\twoheadleftarrow M(\lambda)$ be as in Theorem \ref{T2}, and we define involutions $s_H:=\prod\limits_{h\in H}s_h$ for each independent set $H\subseteq I$. \quad (a) We have the following weight-formula  \begin{equation}\label{Eqn wt-form. by non-hole composition series simples}
 \wt V  \ = \qquad \quad  \bigcup\limits_{\mathclap{H \in \mathrm{Indep}(J_{\lambda})\setminus \mathcal{H}_V }} \qquad \wt L\big(s_H \bullet \lambda\big)
\end{equation} Every composition series for $V$ over semisimple $\mathfrak{g}$, or its local version over Kac--Moody $\mathfrak{g}$, involves all the above simples.\smallskip\\
(b) For any fixed $\mu= \lambda-\sum\limits_{i\in I }c_i\alpha_i\in \wt V$, strengthening \eqref{Eqn wt-form. by non-hole composition series simples},  we now obtain a simple module in this formula that contains $\mu$.
Let $J\subseteq J_{\lambda}$ be any set with all the properties in \eqref{Eweights} and such that the dominant integral weight in the orbit $W_J\mu$ is non-degenerate (see \cite[Proposition 11.2 and (11.2.1)]{Kac book}) w.r.t. the $J^c$-projection $\lambda-\sum\limits_{i\in J^c}c_i\alpha_i$; such a set $J$ exists by Theorem \ref{T1}.
\begin{itemize}
\item If $\supp(\lambda-\mu)\subseteq J$, we immediately have $\mu\in \wt L^{\max}_{J}(\lambda)\subseteq \wt L(\lambda)$.
\item Assume that $J^c\cap \supp(\lambda-\mu)\neq \emptyset$.
Applying \cite[Algorithm 7.3]{MDWF} to the pair \big($J^c$ and sequence $(c_i)_{i\in J^c}$\big) yields an enumeration for $J^c\cap J_{\lambda}$   which terminates in an independent set (say) $\emptyset\neq H\subseteq J^c$.
Then $\mu\in \wt L\big( s_H\bullet \lambda\big)$.
\end{itemize}}
\end{utheorem}
{\color{black}
\begin{remark}
    Theorem \ref{Theorem wt-form. by composition series simples} brings down the complexity of computing the weight-set $\wt V$ of any $V\twoheadleftarrow M(\lambda)$, to that of writing $\wt L(\Lambda)$ for dominant integral $\Lambda\in P^+$; more precisely of $\wt L_{J_{\lambda}}^{\max}(\lambda)$.
    Unlike weight formula \eqref{Eweights} which runs only over the minimal holes $\mathcal{H}^{\min}$ in $\mathcal{H}=\mathcal{H}_V$, as explained in Remark \ref{Remark alternate wt-forms.}, formula 
    \eqref{Eqn wt-form. by non-hole composition series simples} involves all the holes in $\mathcal{H}=\mathrm{Indep}(J_{\lambda})\setminus \mathcal{H}_V$.
\end{remark}
}

\noindent
Our next result is an application {\color{black}of Theorem \ref{T2}, and is a bird's eye view point of Theorem \ref{Theorem wt-form. by composition series simples}.}

\begin{utheorem}\label{T3}
Fix a nonzero $\mathfrak{g}$-module $M(\lambda) \twoheadrightarrow V$ and
set $S := \wt V$. There exist unique maximum and minimum modules
$V^{\max}(S), V^{\min}(S)$ with highest weight $\lambda$, which satisfy
the property:

A module $M(\lambda) \twoheadrightarrow V'$ has $\wt V' = S$, if and only
if $V^{\max}(S) \twoheadrightarrow V' \twoheadrightarrow V^{\min}(S)$.

\noindent In particular, $\wt V^{\max}(S) = \wt V^{\min}(S) = S$.
\end{utheorem}

Theorem~\ref{T3} has several consequences that are explored in
Section~\ref{ST3}:
\begin{enumerate}
\item We define the $k$th order upper- and lower-approximations
$\mathbb{M}_k(\lambda, \mathcal{H}_V)$ and $\mathbb{L}_k(\lambda,
\mathcal{H}_V)$ of every highest weight module $M(\lambda)
\twoheadrightarrow V$. See Definition~\ref{Dapprox}.

\item We isolate the common universal property of these two
approximations, i.e.\ the \textit{$k$th order integrability of $V$}. See
Proposition~\ref{PMkLk} and Definition~\ref{Dhigher}.

\item This leads to a ``repeated-stratification'' of the poset (under
quotienting) of all highest weight $\mathfrak{g}$-modules -- not just of
their sets of weights. See Remark~\ref{Rstrata}.
\end{enumerate}

\begin{remark}[Working over quotient Kac--Moody
algebras]\label{Rworking1}
Our results until Section~\ref{SO} are valid independent of which
Kac--Moody quotient algebra
$\widetilde{\mathfrak{g}} \twoheadrightarrow \mathfrak{g}
\twoheadrightarrow \overline{\mathfrak{g}}$ (all associated to a given
generalized Cartan matrix) is used. This is because (see e.g.\ \cite{DK})
$\wt M(\lambda,J)$, hence $\wt L(\lambda)\ \forall \lambda \in
\mathfrak{h}^*$, does not change across all such $\mathfrak{g}$. Note,
this extends the folklore result that $\wt L(\lambda)$ does not change
across all $\mathfrak{g}$, for $\lambda \in P^+$. Theorem~\ref{T1}
extends these facts to say that $\wt \mathbb{M}(\lambda,\mathcal{H})$
does not change across all such $\mathfrak{g}$, for every $\mathcal{H}$.
(Section~\ref{SO} works in finite type, and for results in
Section~\ref{Sfinal}, see Remark~\ref{Rworking2}.)
Similarly, the computation of weights of highest weight modules over
$\mathfrak{l}_J \subset \mathfrak{g}$ and over $\mathfrak{g}(A_{J \times
J})$ yield the same results modulo identifying $\mathfrak{g}(A_{J \times
J}) \cong \mathfrak{g}_J \subset \mathfrak{l}_J$.
\end{remark}
{\color{black}All of our above weight-formulas are summarized in Subsection \ref{Subsection conclusion} in Section \ref{Sfinal}}  

Having discussed weights in detail, in the final two sections we initiate
the study of two facets on the representation side. First, we introduce
and study ``higher order versions'' of the parabolic category
$\mathcal{O}^{\mathfrak{p}_J}$, corresponding to hole-sets $\mathcal{H}$
with higher order{\color{black}/size} holes:

\begin{defn}\label{DOH}
Given a complex semisimple Lie algebra $\mathfrak{g}$ with simple roots
$\Pi = \{ \alpha_i : i \in I \}$, and a subset $\mathcal{H} \subset {\rm
Indep}(I)$, the \textit{higher order parabolic category
$\mathcal{O}^{\mathcal{H}}$} is the full subcategory of objects in
$\mathcal{O}$ on which the following lowering operator-products all act
locally nilpotently:
\[
{\bf f}_H = {\bf f}^{(0)}_H := \prod_{h \in H} f_h, \qquad H \in
\mathcal{H}.
\]
\end{defn}

\noindent Zeroth and first order special cases are $\mathcal{O}$ and
$\mathcal{O}^{\mathfrak{p}_J}$, respectively. 

Now for the next main result. Notice, Theorem~\ref{Tdk} says that the
weights of a simple module $L(\lambda)$ agree with those of its universal
highest weight cover $M(\lambda, J_\lambda)$ in the parabolic category
$\mathcal{O}^{\mathfrak{p}_{J_\lambda}}$. 
We extend this to every highest weight module $V$, inside the higher
order parabolic category $\mathcal{O}^{\mathcal{H}_V}$.
We also show each $\mathcal{O}^\mathcal{H}$ has enough projectives, and
in a special case, a variant of BGG reciprocity.

\begin{utheorem}\label{T4}
Fix a complex semisimple Lie algebra $\mathfrak{g}$ and a subset
$\mathcal{H} \subset {\rm Indep}(I)$.
\begin{enumerate}
\item With notation from Definition~\ref{DOH}, the category
$\mathcal{O}^{\mathcal{H}}$ is an abelian subcategory of $\mathcal{O}$,
which has enough projectives and enough injectives.

\item The weights of every highest weight module in $\mathcal{O}$, say
$M(\lambda) \twoheadrightarrow V$, agree with those of the universal
highest weight cover of $L(\lambda)$ in $\mathcal{O}^{\mathcal{H}_V}$.

\item If $\mathfrak{g} = \mathfrak{sl}_2^{\oplus n}$, then every
projective module in $\mathcal{O}^{\mathcal{H}}$ has a ``standard
filtration'', and a variant of BGG reciprocity holds. See
Theorem~\ref{TBGGrec} for the details.

\item Suppose $\mathfrak{g} = \mathfrak{sl}_2^{\oplus n}$. Let
$S_\mathcal{H} \subset \mathfrak{h}^*$ comprise the weights $\lambda$
such that $L(\lambda) \in \mathcal{O}^\mathcal{H}$, and let
$P^\mathcal{H}(\lambda)$ denote the projective cover of $L(\lambda)$ in
$\mathcal{O}^\mathcal{H}$. If
\begin{equation}\label{ECartanmatrix}
C = (c_{\lambda \mu})_{\lambda, \mu \in S_\mathcal{H}}, \qquad
c_{\lambda \mu} := [P^\mathcal{H}(\lambda) : L(\mu)]
\end{equation}
is the corresponding {\em ``Cartan matrix''} of Jordan--H\"older
multiplicities, then $C$ is symmetric.
\end{enumerate}
\end{utheorem}

\begin{remark}
We make two points here. First, over $\mathfrak{g} =
\mathfrak{sl}_2^{\oplus 2}$ and over every $\mathfrak{g}$ of rank $\geq
3$, we show that the category $\mathcal{O}^{\mathcal{H}}$ is \textit{not}
always a highest weight category \cite{CPS}. The reason is that
filtrations for different projectives can feature multiple standard
objects with the same highest weight. See Section~\ref{SBGG}.
Next, recall the original goal of the Bernstein--Gelfand--Gelfand paper
\cite{BGG} was to construct a category $\mathcal{O}$ of
$\mathfrak{g}$-modules with enough projectives, whose ``Cartan matrix''
is symmetric. This was parallel to two situations (in positive
characteristic), over finite groups \cite{CR} and over semisimple Lie
algebras \cite{H71}. In all three cases, the Cartan matrix for a suitable
category is symmetric because there is an intermediate class of ``Verma''
modules $M_i$ with ``reciprocity'', i.e.\ $C = D^T D$ for $D = [M_i :
L_j]$. Moreover, $C = D^T D$ in all parabolic/\textit{first order}
categories $\mathcal{O}^{\mathfrak{p}_J}$ \cite{Rocha}. In higher order,
as noted above, not every $\mathcal{O}^\mathcal{H}$ satisfies BGG
reciprocity ``on the nose'' -- i.e.\ we do not get $C = D^T D$ in
$\mathcal{O}^\mathcal{H}$. Nevertheless, the Cartan matrix $C$ as in
\eqref{ECartanmatrix} is symmetric.
\end{remark}

The second study we initiate is that of the \textit{characters} of
``higher order Verma modules''. These include parabolic Vermas and
the family $\mathbb{M}(\mathcal{H})$ above, for which we provide
BGG-type resolutions and Weyl character type formulas -- over Kac--Moody
$\mathfrak{g}$ -- in {\color{black} settings (1)--(4) elaborated below}:\qquad 
{\color{black} The first two cases are when there are no edges between minimal holes -}

(1)~for $\lambda$ with arbitrary integrable roots $J_\lambda$, and
pairwise orthogonal minimal holes in $\mathcal{H}^{\min}$;

(2)~for $\lambda$ with pairwise orthogonal integrable roots $J_\lambda$,
and arbitrary minimal holes in $\mathcal{H}^{\min}$.

\noindent {\color{black}Case (2)} includes (a BGG resolution of) every
highest weight module over $\mathfrak{sl}_2^{\oplus n}$.
{\color{black}The following two theorems are stated for $\lambda = 0$, and we show the analogous results for all $\lambda \in \mathfrak{h}^*$, in
Section~\ref{Sfinal}.}

\begin{utheorem}\label{T5}
In the two settings {\color{black}(1) and (2)} above, there exists a \underline{parabolic Weyl semigroup
$(W_{\mathcal{H}}, \ell_{\mathcal{H}})$}, and a BGG resolution of the
module $\mathbb{M}(\mathcal{H})$ of the form
\begin{equation}\label{EBGG Theorem F}
0 \longrightarrow
M_k \overset{d_k}{\longrightarrow}
M_{k-1} \overset{d_{k-1}}{\longrightarrow}
\cdots \overset{d_2}{\longrightarrow}
M_1 \overset{d_1}{\longrightarrow}
M_0 \overset{d_0}{\longrightarrow}
\mathbb{M}(\mathcal{H}) \to 0,
\end{equation}
with $k = |\mathcal{H}^{\min}|$ and
$M_t \cong \bigoplus_{w \in W_{\mathcal{H}} \; : \;
\ell_{\mathcal{H}}(w) = t} M(w \bullet 0) \ \forall t$.
This implies the Weyl--Kac character formula
\begin{equation}\label{EWKformula}
\ch \mathbb{M}(\mathcal{H}) = \sum_{w \in W_{\mathcal{H}} }
\frac{(-1)^{\ell_{\mathcal{H}}(w)} e^{w \bullet 0}}{\prod_{\alpha
\in \Delta^+} (1 - e^{-\alpha})^{\dim \mathfrak{g}_\alpha}},
\end{equation}
formulated in the spirit of the classical character formulas above.

In particular, if the holes in $\mathcal{H}^{\min}$ are pairwise
orthogonal, then the character of $\mathbb{M}(\mathcal{H})$ is
``$W_{\mathcal{H}}$-invariant'':
\begin{equation}\label{EWchar}
w(\ch \mathbb{M}(\mathcal{H})) = (-1)^{\ell(w) - \ell_\mathcal{H}(w)} \ch
\mathbb{M}(\mathcal{H}), \quad \forall w \in W_\mathcal{H}.
\end{equation}
\end{utheorem}
{\color{black} To go to the next two settings, note over any rank 3 decomposable $\mathfrak{g}= \mathfrak{g}_{\{1\}}\oplus \mathfrak{g}_{\{2,3\}}$ with nodes $2,3\in I=\{1,2,3\}$ having at least one edge in-between -- for e.g. $\mathfrak{g}$ of types $A_1\times A_2$, $A_1\times B_2,\  A_1\times \widehat{A_1},\ldots$ -- we have two non-isomorphic second order Vermas with highest weight $0$: 
\[
\frac{M(0)}{\Big\langle f_1f_2 M(0)_0 , \ f_1f_3 M(0)_0 \Big\rangle}\quad\text{ and }\quad \frac{M(0)}{\Big\langle f_1f_3 M(0)_0 , \ f_2 M(0)_0 \Big\rangle}.
\]
Above, node $1$ is in every size 2 hole.
Next, given two size 2 holes $H_1\neq H_2\in \mathcal{H}^{\min}$, all the edges in the Dynkin subgraph on $H_1\cup H_2$ occur between a fixed pair of nodes one each from $H_1$ and $H_2$.  \\
More generally, we consider -

(3)~Cases of any $\lambda$, holes in $\mathcal{H}^{\min}$ having sizes at most 2, and with a distinguished node $1\in I$ contained in all the size 2 minimal holes.
 
 (4)~The earliest case with edges from a hole incident upon both the nodes of another size 2 hole.
 E.g. over $\mathfrak{sl}_4(\mathbb{C})$, $\lambda=0$ and $\mathcal{H}=\mathcal{H}^{\min}=\big\{  \{1,3 \}\ , \ \{2\} \big\}$; with $2\in I$ the non-leaf in Dynkin graph.
 }
\begin{utheorem}\label{Theorem character of second order Verma via parabolic Vermas}
{\color{black} 
Fix the second order Verma $\mathfrak{g}$-module $\mathbb{M}(\mathcal{H})$ over any symmetrizable Kac--Moody algebra $\mathfrak{g}$, with its minimal holes $\mathcal{H}^{\min}$ as in setting (3) above. 
Let $J$ be the union of all the singleton holes and $H$ be the union of all the size two holes in $\mathcal{H}^{\min}$.
Then the analogue of the Weyl group (and parabolic Weyl subgroups, and of Weyl semigroups in Theorem \ref{T5}) here is $W_{\mathcal{H}} :=W_J\sqcup \ s_1\cdot\big( W_{J\sqcup H}\setminus W_J\big)$.
And $\ell_{\mathcal{H}}$ is the usual length function $\ell$ \big(on $W_{J\sqcup H}$\big) with ignoring the occurrences of $s_1$ in every $w\in W_{\mathcal{H}}$; this can be read off from \eqref{Character numerator in Setting (3)} below.
For this pair $(W_{\mathcal{H}}, \ell_{\mathcal{H}})$, the Weyl--Kac type character formula analogous to \eqref{EWKformula} holds true for $\mathbb{M}(\mathcal{H})$, with the usual denominator and the following character numerator
	\begin{equation}\label{Character numerator in Setting (3)}
\sum\limits_{w\in W_J} (-1)^{\ell(w)}e^{w\bullet \lambda} +\sum\limits_{w\in  W_{H \sqcup J}\setminus W_J}(-1)^{\ell(w)} e^{s_1w \bullet 0}.
	\end{equation} 
    The BBG type resolution \eqref{Eqn character formula upper half} (in Section \ref{Sfinal}) analogous to \eqref{EBGG Theorem F} also holds true for $\mathbb{M}(\mathcal{H})$.}
    \end{utheorem}
    {\color{black}The final setting reveals more interesting and insightful case of $(W_{\mathcal{H}}, \ell_{\mathcal{H}})$ -- which seems to have not appeared in the literature --  enumerating the character numerators of higher order Vermas.}
\begin{prop}\label{Prop A3 second order holes character}
{\color{black}	In Setting (4) over $\mathfrak{sl}_4(\mathbb{C})$, the second order Verma $\mathbb{M}\big( \big\{ \{1,3\} , \{2\} \big\}\big)$ admits a Weyl--Kac character type formula, with the (usual denominator and) numerator:}
	\begin{center}
		${\color{black}\bf   e^0 - e^{s_2\bullet \lambda} - e^{s_1s_3\bullet 0} + e^{s_1s_3\bullet s_2\bullet 0}+ e^{s_2 \bullet s_1s_3\bullet 0}}$\\ \smallskip
		$\displaystyle {\color{black}\bf +e^{s_3\bullet s_2\bullet s_1\bullet 0}+ e^{s_1\bullet s_2\bullet s_3 \bullet 0 }}$ \\ \smallskip
		$\displaystyle {\color{black}\bf -e^{s_2\bullet s_3\bullet s_2\bullet s_1\bullet 0}   -e^{s_1\bullet s_3\bullet s_2\bullet s_1\bullet 0}-   e^{s_2\bullet s_1\bullet s_2\bullet s_3 \bullet 0 } -e^{s_3\bullet s_1\bullet s_2\bullet s_3\bullet 0} }$\\ \smallskip
		$\displaystyle {\color{black}\bf + e^{s_2\bullet s_1\bullet s_3\bullet s_2\bullet s_1\bullet 0} +  e^{s_2\bullet s_3\bullet s_1\bullet s_2\bullet s_3\bullet 0} }$\\ \smallskip
		$\displaystyle {\color{black} - e^{s_2\bullet s_1s_3\bullet s_2\bullet 0} + e^{s_1s_3\bullet s_2\bullet s_1 s_3\bullet 0}  }$\\ \smallskip
		$\displaystyle {\color{black}\bf - e^{w_0\bullet 0} }$
	\end{center}
\end{prop}
 {\color{black}$(W_{\mathcal{H}}, \ell_{\mathcal{H}})$ in this case can be derived from this numerator expansion.}
	% and $\ell_{\mathcal{H}}(w):=\ell(w)\ \forall\ w\in W_I$ and $\ell_{\mathcal{H}}(s_1w)=\ell(w)$ $\forall$ $w\in W_{I\sqcup H}\setminus W_I$,  	yields the desired form of the character formulas running over $\big(W_{\mathcal{H}}, \ell_{\mathcal{H}}(.)\big)$. 

{\color{black}\begin{question}
 Does $(W_{\mathcal{H}},\ell_{\mathcal{H}})$ in Proposition \ref{Prop A3 second order holes character} yield the BGG resolution in \eqref{EBGG Theorem F} for $\mathbb{M}(\mathcal{H})$?   
\end{question}
}
We conclude this section on a philosophical note. The recent papers
\cite{DK2,DK,Kh1,Kh2,Teja} obtained information about
(i)~the weights of simple modules $L(\lambda)$ (for all $\lambda \in
\mathfrak{h}^*$), and
(ii)~the convex hull of $\wt V$ and its face lattice for all highest
weight modules,
using the ``first order information'' associated to every module $V$ --
namely, its \textit{integrability}, defined as:
\begin{equation}\label{Eint}
J_V := \{ h \in J_\lambda \ | \ f_h^{\langle \lambda, \alpha_h^\vee
\rangle + 1} V_\lambda = 0 \}.
\end{equation}
This first order information corresponds to precisely the ``singleton
holes'' in $\mathcal{H}_V$. Moreover, our results on $\wt V$ specialize
to their analogues in the above works when $\mathcal{H}_V$ is the
upper-closure of its singleton elements.
However, a general module $M(\lambda) \twoheadrightarrow V$ can involve
``higher order holes''. That is: the above papers operated using
integrability, i.e., $\mathfrak{sl}_2$-theory. In contrast, we use higher
order integrability -- see Definition~\ref{Dhigher} -- and
``$\mathfrak{sl}_2 \oplus \cdots \oplus \mathfrak{sl}_2$'' theory (in
that each hole $H \in \mathcal{H}_V$ corresponds to a
line in the Verma submodule $U\mathfrak{g}_H \cdot M(\lambda)_\lambda$
killed by $\mathfrak{n}^+$, and $\mathfrak{g}_H \simeq
\mathfrak{sl}_2^{\oplus H}$). For another, ``higher level'' use of this
theory, see the character formulas in Section~\ref{Sfinal}.
%}}}

\subsection{Organization}
In each of the next sections, we prove one of the theorems above.
In Section~\ref{S3} we show Theorem~\ref{T1} as well as~\eqref{Etop}, and
introduce the higher order Verma modules $\mathbb{M}(\lambda,
\mathcal{H})$.

Section~\ref{ST2} proves Theorems~\ref{T2} {\color{black}and \ref{Theorem wt-form. by composition series simples},} and a set of Minkowski
differences for each parabolic Verma module.

Section~\ref{ST3} shows Theorem~\ref{T3} and provides formulas for the
$k$th order upper- and lower-approximations $\mathbb{M}_k(\lambda,
\mathcal{H}_V)$ and $\mathbb{L}_k(\lambda, \mathcal{H}_V)$ of each
highest weight module $V$ -- these include the Verma module $M(\lambda)$
and simple module $L(\lambda)$ when $k=0$, and the parabolic Verma cover
$M(\lambda, J_V)$ when $k=1$.
We also identify in Section~\ref{Sspeculate} the ``higher order
integrability'' that is preserved by the interval of modules
$[\mathbb{L}_k(\lambda, \mathcal{H}_V), \mathbb{M}_k(\lambda,
\mathcal{H}_V)]$ for each $k \geq 0$ and each $V$.
We end with two ``$k$th order'' weight-formulas in Section~\ref{S5final},
which specialize to Theorems~\ref{T1} and~\ref{T2} for $k = 1, \infty$
respectively.

In Section~\ref{SO}, which is over $\mathfrak{g}$ of finite type, we
study the higher order parabolic categories $\mathcal{O}^{\mathcal{H}}$.
We identify the simples and their standard covers in
$\mathcal{O}^{\mathcal{H}}$, and show $\mathcal{O}^{\mathcal{H}}$ has
enough projectives. We also prove BGG reciprocity in all
$\mathcal{O}^{\mathcal{H}}$ over $\mathfrak{g} = \mathfrak{sl}_2^{\oplus
n}$ -- leading to a potential formula in general -- and explain why
$\mathcal{O}^{\mathcal{H}}$ is not always a highest weight category, over
$\mathfrak{sl}_2^{\oplus 2}$ and each higher rank $\mathfrak{g}$.

In Section~\ref{Sfinal}, we provide a BGG resolution and Weyl character
formula for the higher order Verma modules $\mathbb{M}(\lambda,
\mathcal{H})$ in {\color{black} our four settings (1)--(4) above}. That is, we prove the extension
to general $\lambda \in \mathfrak{h}^*$ of Theorem{\color{black}s~\ref{T5} and \ref{Theorem character of second order Verma via parabolic Vermas} and Proposition \ref{Prop A3 second order holes character}. We end by
discussing a speculative BGG resolution for higher order Vermas in general case.}

\subsection{Acknowledgments}
We thank Akaki Tikaradze and Jugal K.\ Verma for useful conversations
regarding Koszul resolutions; Gurbir Dhillon{\color{black}, Sankaran Viswanath and Amritanshu Prasad}  for many valuable
discussions{\color{black}; and K. Hariram for his help in Sage Math coding.}
%We also thank Arun Ram for helpful remarks.
 This work is partially supported by
Ramanujan Fellowship SB/S2/RJN-121/2017,
SwarnaJayanti Fellowship grants SB/SJF/2019-20/14 and DST/SJF/MS/2019/3,
and the FIST program 2021 [TPN--700661]
from SERB and DST (Govt.~of India). {\color{black}The second author also acknowledges his NBHM Postdoc. fellowship (Ref. 0204/13/2023/R{\&}D-II/8206)}.

%{{{1 Section 3 - Theorem A: The weight-formula, and higher order Verma modules
\section{Theorem \ref{T1}: {\color{black}Weight-formula - by parabolic and higher order Verma
modules}}\label{S3}

\subsection{Higher order Verma modules: examples}

We first show Theorem~\ref{T1}. Begin with Theorem~\ref{T2}, which says
the weight-sets of all $M(\lambda) \twoheadrightarrow V$ are of the form
$\wt L_{J_\lambda}^{\max}(\lambda) + \wt \mathbb{M}(\mathcal{H}_V)$.
Associated to this formula, there is a universal module of highest weight
$\lambda$ that we now introduce.

\begin{defn}\label{DMlambdaH}
Given $\lambda \in \mathfrak{h}^*$ and a subset $\mathcal{H} \subset {\rm
Indep}(J_\lambda)$, define the module
\begin{equation}\label{Euhwm}
\mathbb{M}(\lambda, \mathcal{H}) := \frac{M(\lambda)}{\displaystyle
\sum_{H \in \mathcal{H}} U\mathfrak{g} \left( \prod_{h \in H}
f_h^{\langle \lambda, \alpha_h^\vee \rangle + 1} \right)
M(\lambda)_\lambda}.
\end{equation}
\end{defn}

We term these objects \textit{higher order Verma modules}.
These are fundamental objects which include all $M(\lambda)$ and
$M(\lambda, J)$, see below -- and they are indispensable to understanding
all highest weight modules, for several reasons:
(a)~In Theorem~\ref{T1}, we show that $\wt V = \mathbb{M}(\lambda,
\mathcal{H}_V)$.
(b)~In Theorem~\ref{T2}, we show $\wt \mathbb{M}(\lambda, \mathcal{H}) =
\wt L_{J_\lambda}^{\max}(\lambda) + \wt \mathbb{M}(\mathcal{H}_V)$, and
these exhaust all weight-sets $\wt V$.
(c)~In Theorem~\ref{T3}, given $V$, $\mathbb{M}(\lambda, \mathcal{H}_V)$
turns out to be the maximum module $V^{\max}(\wt V)$ (shown below).
Despite these attractive properties -- and notwithstanding
Example~\ref{Exeveryhw} over $\mathfrak{sl}_2^{\oplus n}$ -- to the
best of our knowledge the nontrivial among these modules have not been
studied in the literature. The only ones that have been studied are the
``easy'' case -- parabolic Verma modules $M(\lambda,J)$ -- and the
original inspiration for these modules: $M(0,0) / M(-2,-2)$ over
$\mathfrak{sl}_2^{\oplus 2}$ (in e.g.\ previous work \cite{Kh1,DK}).
Thus we seek to understand these modules -- their characters,
resolutions, etc. -- before general highest weight modules and others in
$\mathcal{O}$. We begin their study in this paper.

\begin{remark}\label{Ruhwm}
Some clarifying observations:
(1)~$\mathbb{M}(\lambda, \mathcal{H})$ is unchanged if one replaces
$\mathcal{H}$ by its upper-closure in ${\rm Indep}(J_\lambda)$, or by any
set in between.
(2)~In ``reverse'', replacing $\mathcal{H}$ by its ``minimal'' elements
$\mathcal{H}^{\min}$ does not change $\mathbb{M}(\lambda, \mathcal{H})$.
Thus the modules $\mathbb{M}(\lambda, \mathcal{H})$ subsume the usual
parabolic Verma modules $M(\lambda,J)$, in the special case that
$\mathcal{H}$ is the upper-closure of the singleton sets in it.
(3)~The modules $\mathbb{M}(\lambda, \mathcal{H})$ specialize to
$\mathbb{M}(\mathcal{H})$ in~\eqref{EMH} via $\lambda \leadsto 0$.
(4)~The extreme cases are:
(i)~$\mathbb{M}(\lambda, \mathcal{H}) = 0 \Longleftrightarrow H =
\emptyset \in \mathcal{H}$,
(ii)~$\mathbb{M}(\lambda, \mathcal{H}) = M(\lambda) \Longleftrightarrow
\mathcal{H} = \emptyset$.
(iii)~As the simple module $L(\lambda)$ has maximum possible
integrability $J_\lambda$, similarly it has maximum possible
$\mathcal{H}$-set, ${\rm Indep}(J_\lambda) \setminus \{ \emptyset \}$.
\end{remark}

\begin{remark}[Weak Minkowski decomposition]
Like their first order versions $M(\lambda, J)$, the modules
$\mathbb{M}(\lambda, \mathcal{H})$ have a freeness property over $U
\mathfrak{a}_\mathcal{H}$, for a Lie subalgebra $\mathfrak{a}_\mathcal{H}
\subset \mathfrak{n}^-$. Namely, define
\[
\mathfrak{a}_\mathcal{H} := \bigoplus_{\alpha \not\in
\Delta_{\widetilde{H} }} \mathfrak{n}^-_\alpha, \qquad
\text{where} \qquad \widetilde{H} := \bigcup_{H \in
\mathcal{H}^{\min}} H \ \subseteq \ I.
\]
Also define $\bff_H := \prod_{h \in H} f_h^{\langle \lambda,
\alpha_h^\vee \rangle + 1}$ for $H \in {\rm Indep}(J_\lambda)$. Then by
the PBW theorem,
\[
\mathbb{M}(\lambda, \mathcal{H}) \ \cong_{\mathfrak{a}_\mathcal{H}} \
\frac{U \mathfrak{a}_\mathcal{H} \otimes_\mathbb{C} U
\mathfrak{n}^-_{\Delta_{\widetilde{H}} }}{\displaystyle U
\mathfrak{a}_\mathcal{H} \otimes_\mathbb{C} \sum_{H \in
\mathcal{H}^{\min}} U \mathfrak{n}^-_{\Delta_{\widetilde{H} }}
\cdot \bff_H} \ \cong_{\mathfrak{a}_\mathcal{H}} \ U
\mathfrak{a}_\mathcal{H} \otimes_\mathbb{C} \frac{U
\mathfrak{n}^-_{\Delta_{\widetilde{H}} }}{\displaystyle
\sum_{H \in \mathcal{H}^{\min}} U
\mathfrak{n}^-_{\Delta_{\widetilde{H}} } \cdot \bff_H}.
\]
Hence by Theorem~\ref{T1}, the weights of every highest weight
$\mathfrak{g}$-module $M(\lambda) \twoheadrightarrow V$ satisfy:
\begin{equation}
\wt V = \wt_J V - \mathbb{Z}_{\geq 0} (\Delta^+ \setminus \Delta^+_J),
\qquad \forall J \supseteq \bigcup_{H \in \mathcal{H}_V^{\min}} H.
\end{equation}
Note that the special case $J = J_\lambda$ {\color{black} -- among a large collections of subsets $J$, including $J=$ the support/union of all the holes in $\mathcal{H}^{\min}_V$ in the terminology of this paper --} was one of the main results in
previous work {\color{black}\cite[Theorems A and C(CV)]{MDWF}}.
\end{remark}

Next, here are some examples of highest weight modules, including
$\mathbb{M}(\lambda, \mathcal{H})$ (beyond the ``obvious'' cases
$M(\lambda,J)$ in Remark~\ref{Ruhwm}(2)), in order to build more
intuition.

\begin{example}\label{Exrank2}
As above: the first nontrivial example is in~\eqref{EM00}.
(This was originally used in \cite{Kh1} by the first author to observe
that the convex hull of $\wt V$ does not always yield $\wt V$.) In this
example, $\mathfrak{g} = \mathfrak{sl}_2 \oplus \mathfrak{sl}_2$ and
$\lambda = (0,0)$. Setting $\mathcal{H} = \{ \{ 1, 2 \} \} = \{ J_\lambda
\}$ yields $\mathbb{M}(\lambda, \mathcal{H}) = M(0,0) / M(-2,-2)$. This
is the ``simplest'' module whose weights are not determined by their
convex hull. It is also the prototypical module for all such cases; see
Theorem~\ref{Tlepowsky} and the subsequent lines. By Theorem~\ref{T1},
\[
\wt M(0,0) / M(-2,-2) = \wt M((0,0), \{ 1 \}) \cup \wt M((0,0), \{ 2 \})
= - \Z_{\geq 0} \alpha_2 \cup - \Z_{\geq 0} \alpha_1.
\]
\end{example}

\begin{example}[All highest weight modules over
$\mathfrak{sl}_2^{\oplus n}$]\label{Exeveryhw}
Let $\mathfrak{g} = \mathfrak{sl}_2^{\oplus n}$. For $n=1$, every highest
weight $\mathfrak{g}$-module is either Verma or finite-dimensional --
i.e.\ a parabolic Verma module. What about higher $n$?
We claim, every module $M(\lambda) \twoheadrightarrow V$ equals
$\mathbb{M}(\lambda, \mathcal{H})$ for some $\mathcal{H} \subset {\rm
Indep}(J_\lambda) = 2^{J_\lambda}$ -- adding to the fundamental nature of
these modules. The claim follows by noting that if $0 \to N \to
M(\lambda) \to V \to 0$, then $N$ is generated by weight \textit{spaces}
(since $\mathfrak{g} = \mathfrak{sl}_2^{\oplus n}$), and hence by maximal
weight vectors. By $\mathfrak{sl}_2^{\oplus n}$-theory, these are
precisely $\prod_{h \in H} f_h^{\langle \lambda, \alpha_h^\vee \rangle +
1} \cdot M(\lambda)_\lambda$ for $H \subset J_\lambda$.
\end{example}

\begin{example}[Some rank-$4$ examples]\label{Exrank4}
Let $\mathfrak{g}=\mathfrak{sl}_5$, $I=\{1,2,3,4\}$, and
$\lambda=\varpi_1-\varpi_4$, with $\varpi_i$ the fundamental weights.
Note that 
$J_{\lambda}=\{1,2,3\}$ and
${\rm Indep}(J_\lambda) = \{\{1\},\, \{2\},\, \{3\},\, \{1,3\} \}$.
Consider 
\[
\mathbb{M}\left(\lambda, \{\{2\},\{1,3 \} \} \right) =
\frac{M(\lambda)}{U(\mathfrak{g}) f_2 \cdot M(\lambda)_{\lambda} \, + \,
U(\mathfrak{g}) f_1^2f_3 \cdot M(\lambda)_{\lambda}}.
\]
Notice, $\mathcal{H} = \{\{2\},\,\{1,3\}\}$ is upper-closed in ${\rm
Indep}(J_{\lambda})$, and $\mathbb{M}(\lambda,\mathcal{H})$ is a quotient
of the parabolic Verma module $M(\lambda,\{2\})$. Now Theorem~\ref{T1}
yields:
\[
\wt 
\mathbb{M}\left(\lambda, \{\{2\},\{1,3 \} \} \right) \ = \ \wt
M(\lambda,\{1,2\} )\cup \wt M(\lambda,\{2,3\}) \cup \wt M(\lambda, \{ 1,
2, 3 \}),
\]
where we can omit the final set as it lies in the first two terms; see
Remark~\ref{Ralt}.

Similarly, consider the following two modules:
\[
V_1 = \frac{M(\lambda)}{U(\mathfrak{g})f_2^2f_3 \cdot
M(\lambda)_{\lambda}}, \qquad V_2 =
\frac{M(\lambda)}{U(\mathfrak{g})f_1^2 \cdot M(\lambda)_{\lambda} +
U(\mathfrak{g})f_2^2f_3 \cdot M(\lambda)_{\lambda}}.
 \]
Then $J_{V_1}=\emptyset$, $J_{V_2}=\{1\}$, and the sets
$\mathcal{H}_{V_1} = \emptyset$, $\mathcal{H}_{V_2} =
\{\{1\},\,\{1,3\}\}$ are upper-closed in ${\rm Indep}(J_{\lambda})$. As
the nodes $2,3$ are adjacent in the Dynkin diagram, $\{2,3\}$ does not
contribute to a hole in $\mathcal{H}_{V_1}$ and $\mathcal{H}_{V_2}$.
Hence by Theorem \ref{T1}, $\wt V_1\ = \ \wt M(\lambda)$ and $\wt V_2 \ =
\ \wt M(\lambda,\{1\})$.\qed
\end{example}

Following these examples, we add to the intuition behind Theorem~\ref{T1}
in Remark~\ref{Rintuit}. First note an elementary lemma, which follows by
considering the Kostant partition function.

\begin{lemma}\label{Lkostant}
Fix Kac--Moody $\mathfrak{g}$ and $\lambda \in \mathfrak{h}^*$. The
weight space $M(\lambda)_\mu$ of the Verma module is one-dimensional if
and only if $\mu = \lambda - \sum_{h \in H} n_h \alpha_h$, where $H
\subset I$ is independent and all $n_h \in \Z_{>0}$.
\end{lemma}

\begin{remark}\label{Rtop}
Returning to Theorem~\ref{T1}, suppose $0 \to N \to M(\lambda) \to V \to
0$. It is not clear if weights are lost upon quotienting $M(\lambda)$ by
maximal vectors in $N$ corresponding to non-independent nodes -- e.g.\
they are not lost in Example~\ref{Exrank4} with $f_2^2 f_3
M(\lambda)_\lambda$ (modulo proving Theorem~\ref{T1}). However,
$1$-dimensional weight spaces $M(\lambda)_\mu \subset N$ are clearly
sensitive for $\wt V$, since then $V_\mu = 0$ and weights are lost, by
$\mathfrak{sl}_2^{\oplus n}$-theory.
In proving Theorem~\ref{T1}, we show that $\wt V = \wt
\mathbb{M}(\lambda, \mathcal{H}_V)$ (see~\eqref{Euhwm}). This shows the
converse to the above application of $\mathfrak{sl}_2^{\oplus n}$-theory
(and extends~\eqref{Etop}):
weights are lost when passing from $M(\lambda)$ to $V$ (i.e.\ now, when
passing to $\mathbb{M}(\lambda, \mathcal{H}_V)$) \textit{only} if one
proceeds as in Example~\ref{Exeveryhw}. That is, only if one quotients
out $1$-dimensional weight spaces in $M(\lambda)$ spanned by
\textit{maximal} vectors for $\mathfrak{sl}_2^{\oplus H}$, for some
independent set/hole -- by Lemma~\ref{Lkostant} -- $H \in \mathcal{H}_V$.
\end{remark}

We continue with a ``computational'' remark and some examples.

\begin{remark}\label{Ralt}
One can work with fewer sets $J \subset J_\lambda$ in~\eqref{Eweights},
as $M(\lambda,J) \twoheadrightarrow M(\lambda, J')$ for $J \subset J'
\subset J_\lambda$.
Also, $J$ only needs to intersect the minimal holes $H \in
\mathcal{H}_V^{\min}$, so one can just use the subsets $J \subset \cup_{H
\in \mathcal{H}_V^{\min}} H$. Next, if the holes in
$\mathcal{H}_V^{\min}$ are pairwise disjoint, one uses exactly $\prod_{H
\in \mathcal{H}_V^{\min}} |H|$-many transversal sets $J$
in~\eqref{Eweights} -- as in the previous two examples -- by selecting
one node from each $H$ in $J$.
Of course, this does not always hold -- e.g.\ for $\mathfrak{g} =
\mathfrak{sl}_2^{\oplus 3}$ and $\mathcal{H}_V^{\min} = \{ \{ 1, 2 \}, \{
2, 3 \}, \{ 1, 3 \} \}$, the three sets $J = \{ 1, 2 \}, \{ 2, 3 \}, \{
1, 3 \}$ suffice.
(This example also features below, see~\eqref{E122331}.)
\end{remark}

\begin{example}
Suppose $V = L(\lambda)$ is simple. Then for each integrable direction $j
\in J_\lambda$, $\{ j \} \in \mathcal{H}_{L(\lambda)}$ since
$f_j^{\langle \lambda, \alpha_j^\vee \rangle + 1} \cdot
L(\lambda)_\lambda = 0$ (as its preimage generates a proper submodule of
$M(\lambda)$). Moreover, every hole $H \subset J_\lambda$. Thus there is
a unique set $J$ in the union in~\eqref{Eweights} for $V = L(\lambda)$,
namely, $J = J_\lambda$. Hence Theorem~\ref{Tdk} is an immediate
consequence of Theorem~\ref{T1}: $\wt L(\lambda) = \wt M(\lambda,
J_\lambda)$.
\end{example}

\begin{example}[Multiplicity-free character formula for
$\mathbb{M}(\lambda, \mathcal{H})$]
Let $\mathfrak{g}=\mathfrak{sl}_2^{\oplus n}$, $\lambda \in
\mathfrak{h}^*$, and let $\emptyset \neq \mathcal{H} \subset {\rm
Indep}(J_\lambda) = 2^{J_\lambda}$. We compute $\ch
\mathbb{M}(\lambda,\mathcal{H})$, or simply the weights of
$\mathbb{M}(\lambda, \mathcal{H})$, since all highest weight
$\mathfrak{sl}_2^{\oplus n}$-modules have one-dimensional weight spaces.
Enumerate the subsets of $J_\lambda$ that intersect all holes in
$\mathcal{H}$ as $\{J_1,\ldots, J_l \}$. (We may consider only the
minimal such subsets.) Then
\[ 
\wt \mathbb{M}(\lambda,\mathcal{H})\ =\ \bigcup\limits_{i=1}^l (\wt
L_{J_i}^{\max}(\lambda) - \mathbb{Z}_{\geq
0}(\Delta^+\setminus\Delta_{J_i}^+))\ =\ \bigcup\limits_{i=1}^l (\wt
L_{J_i}^{\max}(\lambda) - \mathbb{Z}_{\geq 0}\Pi_{J_i^c})
\] 
by Theorem~\ref{T1}. Over $\mathfrak{sl}_2^{\oplus n}$, it is not hard to
show that
\[
\big(\wt L_{J_i}^{\max}(\lambda) - \mathbb{Z}_{\geq 0}\Pi_{J_i^c}
\big)\cap \big(\wt L_{J_j}^{\max}(\lambda) - \mathbb{Z}_{\geq
0}\Pi_{J_j^c} \big)= \big(\wt L_{J_i\cup J_j}^{\max}(\lambda) -
\mathbb{Z}_{\geq 0}\Pi_{\left(J_i\cup J_j\right)^c} \big).
\]
By the inclusion-exclusion principle, and since all weight spaces of
$\mathbb{M}(\lambda, \mathcal{H})$ are one-dimensional,
\begin{equation}
\ch \mathbb{M}(\lambda,\mathcal{H})\ =\ \sum_{\emptyset \neq S\subseteq
\{ 1, \dots, l \}} (-1)^{|S|-1} \ch M(\lambda, \cup_{i\in S} J_i), \qquad
\mathcal{H} \neq \emptyset.
\end{equation}
This provides an alternating formula for $\ch \mathbb{M}(\lambda,
\mathcal{H})$ in terms of the sets $J_i$ that are transversing the holes
in $\mathcal{H}$ (equivalently, in $\mathcal{H}^{\min}$) -- over
$\mathfrak{g} = \mathfrak{sl}_2^{\oplus n}$.
This picture is ``orthogonal'' to the alternating
formula~\eqref{EcharMlH} that we obtain below -- that formula is
alternating in terms of the holes in $\mathcal{H}^{\min}$, and follows
from the BGG-type resolution \eqref{EBGGH2} for the same module
$\mathbb{M}(\lambda, \mathcal{H})$ over $\mathfrak{g} =
\mathfrak{sl}_2^{\oplus n}$. 
\end{example}

\subsection{Proof of Theorem~\ref{T1}: reverse inclusion}

We now turn to the proof of Theorem~\ref{T1}. It is clear that for any
highest weight module $V$, the universal module
$\mathbb{M}(\lambda, \mathcal{H}_V) \twoheadrightarrow V$, implying an
inclusion of their weight-sets. We will show this inclusion is in fact an
\textit{equality} (extending~\eqref{Etop}), so that
$\mathbb{M}(\lambda, \mathcal{H}_V)$ is the ``universal highest weight
cover'' of $V$. First, an observation useful below:

\begin{lemma}\label{Luhwm}
Fix $\lambda$ and $\mathcal{H} \subset {\rm Indep}(J_\lambda)$, and let
$M := \mathbb{M}(\lambda, \mathcal{H})$.
Then $\mathcal{H}_M$ equals the upper-closure of $\mathcal{H}$. In
particular, for all nonzero modules $M(\lambda) \twoheadrightarrow V$, one
has $\mathcal{H}_V = \mathcal{H}_N$ for $N = \mathbb{M}(\lambda,
\mathcal{H}_V)$.
\end{lemma}

Returning to our present goal, in the course of proving Theorem~\ref{T1}
we study three sets of weights:
\begin{equation}\label{E3sets}
\wt \mathbb{M}(\lambda, \mathcal{H}_V), \qquad \wt V, \qquad \text{and}
\qquad S(\lambda, \mathcal{H}_V),
\end{equation}
where for convenience we define for any subset $\mathcal{H} \subset {\rm
Indep}(J_\lambda)$:
\begin{equation}\label{ESH}
S(\lambda, \mathcal{H}) := \begin{cases}
\wt M(\lambda), & \text{if }\mathcal{H}=\emptyset,\\
\emptyset, & \text{if }\emptyset \in \mathcal{H},\\
\displaystyle \bigcup_{J \subset J_\lambda\, :\, J \cap H \neq
\emptyset\; \forall H \in \mathcal{H}} \wt M(\lambda,J), \qquad \qquad
& \text{if }\mathcal{H}\neq \emptyset\text{ and } \emptyset \notin
\mathcal{H}.
\end{cases}
\end{equation}

Theorem~\ref{T1} asserts that the final two sets of weights
in~\eqref{E3sets} coincide, but here we also claim the added equality
$\wt V = \wt \mathbb{M}(\lambda, \mathcal{H}_V)$ for all nonzero
modules $M(\lambda) \twoheadrightarrow V$ (which extends~\eqref{Etop}
and generalizes~\eqref{EdkL}, and) which is repeatedly used in later
sections. Since $\mathbb{M}(\lambda, \mathcal{H}_V) \twoheadrightarrow
V$, the first set in~\eqref{E3sets} contains the second. The next step
is:

\begin{prop}\label{Pinclusion}
If $\lambda \in \mathfrak{h}^*$ and $M(\lambda) \twoheadrightarrow V$ is
nonzero, then $\wt V \supseteq S(\lambda, \mathcal{H}_V)$.
\end{prop}

The proof appeals to a result from previous work, which is the special
case $\mathcal{H}_V = \emptyset$ of Theorem~\ref{T1}:

\begin{theorem}[{\cite[Theorem B(B0)]{MDWF}}]\label{Tnoholes}
Let $\lambda\in \mathfrak{h}^*$, and $M(\lambda) \twoheadrightarrow V$.
Suppose there are no holes in $\wt V$, i.e., $\mathcal{H}_V=\emptyset$.
Then $\wt V=\wt M(\lambda)$.
\end{theorem}

For completeness -- and as this is a special case of
Theorem~\ref{T1} -- we provide a proof in Appendix~\ref{Anoholes}.

\begin{proof}[Proof of Proposition~\ref{Pinclusion}]
Notice $\emptyset\notin \mathcal{H}_V$ since $V \neq 0$. If
$\mathcal{H}_V=\emptyset$, the result follows from
Theorem~\ref{Tnoholes}. Thus, assume henceforth that $\emptyset\notin
\mathcal{H}_V$ and $\emptyset\neq \mathcal{H}_V$. Pick $J \subseteq
J_\lambda$ with the property that $J\cap H\neq\emptyset$ $\forall$
$H\in\mathcal{H}_V$; thus $J^c \not\supseteq H\ \forall H\in
\mathcal{H}_V$. Hence by Theorem \ref{Tnoholes} over
$\mathfrak{g}_{J^c}$,
\[
\lambda-\mathbb{Z}_{\geq 0}\Pi_{J^c} = \wt
U(\mathfrak{g}_{J^c})V_{\lambda} \subseteq \wt V.
\]
Next for any $\xi\in \mathbb{Z}_{\geq 0}\Pi_{J^c}$, every nonzero vector
$x\in V_{\lambda-\xi}$ is a maximal vector for the action of the Levi
$\mathfrak{l}_J$, so it generates the highest weight
$\mathfrak{l}_{J}$-module $U(\mathfrak{l}_J)x$ of highest weight
$\lambda-\xi$. Hence $\wt L_J^{\max}(\lambda-\xi)\subseteq \wt
\left(U(\mathfrak{l}_J)x\right) \subseteq \wt V$. Now by the integrable
slice decomposition (see Lemma~\ref{Lslice}),
\[
\wt M(\lambda,J)=\bigsqcup_{\xi\in \mathbb{Z}_{\geq 0}\Pi_{J^c}}\wt
L_J^{\max}(\lambda-\xi)\subseteq \wt V.
\]
As $J$ is arbitrary, this yields $S(\lambda,\mathcal{H}_V)\subseteq \wt
V$.
\end{proof}

\subsection{Proof of Theorem~\ref{T1}: forward inclusion}

Returning to the discussion preceding Proposition~\ref{Pinclusion}, the
proof of Theorem~\ref{T1} is completed by showing the other inclusion:

\begin{theorem}\label{Tinclusion}
For all $\mathfrak{g}$, weights $\lambda$, and nonzero modules
$M(\lambda) \twoheadrightarrow V$, $\wt V \subset S(\lambda,
\mathcal{H}_V)$.
\end{theorem}

\noindent Recall that $\mathcal{H}_V$ and $S(\lambda, \mathcal{H}_V)$
were defined in Definition~\ref{Dholes} and Equation~\eqref{ESH},
respectively.

Theorem~\ref{Tinclusion} not only implies Theorem~\ref{T1} (given
Proposition~\ref{Pinclusion}), but together with Lemma~\ref{Luhwm} it
also implies the remaining sought-for inclusion in~\eqref{E3sets}:
\[
\wt M \subset S(\lambda, \mathcal{H}_M) = S(\lambda, \mathcal{H}_V),
\qquad \text{for } M = \mathbb{M}(\lambda, \mathcal{H}_V).
\]

In the rest of this section, we show Theorem~\ref{Tinclusion}. The proof
uses a fundamental result from~\cite{DK}:

\begin{lemma}[Integrable slice decomposition, Khare \cite{Kh1},
Dhillon--Khare \cite{DK}]\label{Lslice}
If $J \subset J_\lambda$, then
\begin{equation}\label{Eslice}
\wt M(\lambda,J)\ = \ \wt L_J^{\max}(\lambda)-\mathbb{Z}_{\geq
0}(\Delta^+\setminus \Delta^+_J)\ = \ \bigsqcup_{\xi\in \mathbb{Z}_{\geq
0}\Pi_{I\setminus J}}\wt L_J^{\max}(\lambda-\xi).
\end{equation}
\end{lemma} 

We will also need the following application of $\mathfrak{sl}_3$- (or
rank-$2$) theory, aka the Serre relations:

\begin{lemma}\label{LSerre}
Fix Kac--Moody $\mathfrak{g}$, integers $k, M, N \geq 0$, and pairwise
distinct nodes $h_1, \dots, h_k, h, i \in I$, such that $h$ is adjacent
to $i$ but not to $h_1, \dots, h_k$.
Say $Y \in U\mathfrak{n}^-$ is a linear combination of words in the
Chevalley generators $f_i, f_{h_1}, \dots, f_{h_k}$, with $N$ occurrences
of $f_i$ in each word. Then
\[
f_h^{M - N \langle \alpha_i, \alpha_h^\vee \rangle} \cdot Y = X \cdot
f_h^M,
\]
for $X \in U\mathfrak{n}^-$ again a linear combination of words with $N$
occurrences of $f_i$ in each word.
\end{lemma}

\begin{proof}
It suffices to work with $Y$ a single word in the given alphabet. The
next calculation (found in textbooks) holds in any associative algebra,
and is specialized here to $U(\mathfrak{n}^-)$:
\begin{equation}\label{ESerre2}
f_h^l f_i = \sum_{j=0}^{- \langle \alpha_i, \alpha_h^\vee \rangle}
\binom{l}{j} (\text{ad}\, f_h)^j(f_i) f_h^{l-j} = X_1 f_h^{l + \langle
\alpha_i, \alpha_h^\vee \rangle}, \qquad \text{for all } l \geq - \langle
\alpha_i, \alpha_h^\vee \rangle.
\end{equation}
Here $X_1 \in U(\mathfrak{n}^-)$ is a linear combination of words in $\{
f_l : l \in I \}$, each containing just one $f_i$ -- and the sum stops
where it does due to the Serre relations. Now suppose
$Y = Y_1 f_i Y_2 f_i \cdots f_i Y_{N+1},$
with each $Y_j$ a word in the $f_{h_t}$. Successively
applying~\eqref{ESerre2} with $l = M - (N-j) \langle \alpha_i,
\alpha_h^\vee \rangle$ for $j = 0, 1, \dots$,
\begin{align*}
f_h^{M - N \langle \alpha_i, \alpha_h^\vee \rangle} \cdot Y = &\ Y_1
\cdot X_1 f_h^{M - (N-1) \langle \alpha_i, \alpha_h^\vee \rangle} \cdot
Y_2 f_i \cdots f_i Y_{N+1}\\
= &\ Y_1 X_1 Y_2 \cdot X_2 f_h^{M - (N-2) \langle \alpha_i, \alpha_h^\vee
\rangle} \cdot Y_3 f_i \cdots f_i Y_{N+1}\\
= &\ \cdots\\
= &\ Y_1 X_1 \cdots Y_N \cdot X_N f_h^M \cdot Y_{N+1}.
\end{align*}
Setting $X := Y_1 X_1 \cdots Y_N X_N Y_{N+1}$, we are done, since
each $X_t$ contains exactly one $f_i$.
\end{proof}

\noindent With Lemmas~\ref{Lslice} and~\ref{LSerre} in hand, we complete
the proof of the remaining half of Theorem \ref{T1}.

\begin{proof}[Proof of Theorem \ref{Tinclusion}]
When $\mathcal{H}_V=\emptyset$, Theorem~\ref{Tnoholes} implies $\wt V
=\wt M(\lambda)= S(\lambda,\mathcal{H}_V)$. We now assume
$\mathcal{H}_V\neq\emptyset$, and also $\emptyset\notin \mathcal{H}_V$
(else the result is trivial). Notice, for each $H\in {\rm
Indep}(J_\lambda)$ the vector $\bigg(\prod\limits_{h\in H}
f_h^{\langle\lambda,\alpha_h^\vee\rangle+1}\bigg) m_{\lambda}$ is a
maximal vector in the $\left(\prod_{h\in H}s_h\right)\bullet
\lambda$-weight space of $M(\lambda)$.

We now turn to the proof, with $\mathcal{H}_V\neq\emptyset$,
$\emptyset\notin \mathcal{H}_V$. The idea is to work with \textit{all}
triples $(\lambda',V', \mu')$, where $\lambda' \in \mathfrak{h}^*$,
$M(\lambda') \twoheadrightarrow V'$ (with holes $\mathcal{H}_{V'}$), and
$\mu' \in \wt V'$ are arbitrary. We make the following\medskip

\noindent \textbf{Claim.}
\textit{For every triple $(\lambda',V',\mu')$ as above,
$\mu'\in S(\lambda',\mathcal{H}_{V'})$}.\medskip

This claim -- which implies the theorem -- is now shown by induction on
${\rm ht}(\lambda'-\mu')\geq 0$. In the base case, $\mu' = \lambda'$, and
so the claim holds trivially.\medskip

\noindent \textbf{Induction step:}
Fix an arbitrary $(\lambda,V,\mu)$ as above, with ${\rm
ht}(\lambda-\mu)>0$ (and assume the result is true for all triples
$(\lambda', V', \mu')$ with smaller ${\rm ht}(\lambda' - \mu')$). We
introduce notation for a set in the union in $S(\lambda, \mathcal{H}_V)$
in~\eqref{ESH}:
\begin{equation}\label{EJV}
\mathfrak{J}(V) := \{ J \subset J_\lambda \ | \ J \cap H \neq \emptyset \
\forall H \in \mathcal{H}_V \}.
\end{equation}

Thus the goal is to find $J \in \mathfrak{J}(V)$ such that $\mu \in \wt
M(\lambda, J)$. This would imply $\mu \in S(\lambda, \mathcal{H}_V)$ and
hence show the induction step.

We now break up the remainder of the induction step into (sub-)steps, in
the interest of clarity. The reader may find it helpful to first read the
plan for all steps, before reading the details.\medskip

\noindent \textit{\underline{Step 1:}
First choose an arbitrary element $K \in \mathfrak{J}(V)$ -- for
instance, $K = J_\lambda$. We produce a PBW monomial $F_1 \in
U(\mathfrak{n}^-_K)$ such that $\mu \leq \nu \leq \lambda$, where $\nu :=
\lambda + \wt(F_1) \in \wt V$.}\smallskip

To do so, write $\mathfrak{n^-}$ as the direct sum of the two Lie
subalgebras $\mathfrak{n}' := \bigoplus\limits_{\alpha\in
\Delta^-\setminus \Delta^-_K}\mathfrak{n}^-_{\alpha}$ and
$\mathfrak{n}_K^-$. Via the PBW theorem, fix a basis for
$U(\mathfrak{n}^-)$ consisting of monomials in negative root vectors such
that in each monomial, elements from $\mathfrak{n}'$ always occur to the
left of those from $\mathfrak{n}^-_K$. Now fix a nonzero highest weight
vector $v_\lambda \in V_\lambda$, and pick a nonzero weight vector 
\[
z:= F_2 \cdot F_1 \cdot v_{\lambda} \in V_\mu, \quad \text{ for PBW
monomials }F_1\in U(\mathfrak{n}^-_K)\text{ and }F_2\in U (\mathfrak{n}')
\]
with $\lambda+\wt(F_1)+\wt(F_2)=\mu$. Note that
$\wt(F_1)\in-\mathbb{Z}_{\geq 0}\Pi_K$ and $\wt(F_2)\in\mathbb{Z}_{\geq
0}\left(\Delta^-\setminus \Delta_K^-\right)$. Now set $\nu :=
\lambda+\wt(F_1)$, and note that $F_1 v_{\lambda}\in (V_K)_{\nu} =
V_\nu$, where $V_K := U(\mathfrak{g}_K) v_\lambda$.\medskip

\noindent \textit{\underline{Step 2:}
There are now two cases. If $\nu \in \wt L_K^{\max}(\lambda)$, then $\mu
\in S(\lambda, \mathcal{H}_V)$.}\smallskip

Indeed, $\mu = \nu + \wt(F_2) \in \wt M(\lambda, K)$ (by~\eqref{Epvm}),
and this is in $S(\lambda, \mathcal{H}_V)$ since $K \in
\mathfrak{J}(V)$.\medskip

\noindent \textit{\underline{Step 3:}
Thus, henceforth $\nu \not\in \wt L_K^{\max}(\lambda)$. We claim there
exists $i \in K \subset J_\lambda$ such that
(i)~$V_{s_i \bullet \lambda} = f_i^{m_i} V_\lambda \neq 0$, where
$m_i := \langle \lambda, \alpha_i^\vee \rangle + 1$, and
(ii)~$V' := U (\mathfrak{g}) \cdot V_{s_i \bullet \lambda}$ has a nonzero
$\mu$-weight space.}\smallskip

To see why, choose and fix the highest weight vector $m_\lambda \in
M(\lambda)_\lambda$ which maps to $v_\lambda$ under $M(\lambda)
\twoheadrightarrow V$. Also let $M_K(\lambda) \cong U(\mathfrak{n}_K^-)$
be the Verma $\mathfrak{g}_K$-module. Since
\[
\nu \not\in \wt L_K^{\max}(\lambda) =
\wt\left[ \frac{M_K(\lambda)}{\sum\limits_{t\in K}
U(\mathfrak{n}^-_K)f_t^{m_t} m_\lambda}\right],
\]
the $\nu$-weight space of $M_K(\lambda)$ equals that of its submodule
$\sum_{t\in K}U(\mathfrak{n}^-_K)f_t^{m_t} m_\lambda$. Hence the
$\nu$-weight space of $V_K = U(\mathfrak{g}_K) v_\lambda$ equals that of
$\sum_{t\in K}U(\mathfrak{n}^-_K)f_t^{m_t} v_\lambda$. Now $F_1 v_\lambda
\in V_\nu = (V_K)_\nu$ is a linear combination of vectors $X_t f_t^{m_t}
v_\lambda$, with $X_t \in U(\mathfrak{n}^-_K)$. Since $F_2 F_1 \cdot
v_\lambda \neq 0$, it follows that there exist a node $i\in K$ and a PBW
monomial $F_3$ in $U(\mathfrak{n}_K^-)$ such that
\[
z' := F_2 \cdot F_3 \cdot f_i^{m_i} v_{\lambda} \in V_\mu
\]
is nonzero. Defining $V' := U(\mathfrak{g}) \cdot f_i^{m_i} v_\lambda$,
this proves both assertions~(i) and~(ii), since $z' \in V'_\mu$.\medskip

\noindent \textit{\underline{Step 4:}
If $H \in \mathcal{H}_V$ is a hole, then $\emptyset \neq H \setminus \{ i
\} \in \mathcal{H}_{V'}$, where $V' \neq 0$ is as in Step~3.}\smallskip

The previous three steps helped us arrive at $V' = U(\mathfrak{g})
f_i^{m_i} V_\lambda$, to which we now apply the induction hypothesis. The
present step does not use the previous three steps, except the definition
of $V'$. We begin with the definition of $\mathcal{H}_{V'}$, using the
description of the highest weight line $V'_{s_i \bullet \lambda}$:
\[
\mathcal{H}_{V'} = \left\{ H' \in {\rm Indep}(J_{s_i \bullet \lambda}) \
\bigg| \ \left( \prod_{h \in H'} f_h^{\langle s_i \bullet \lambda,
\alpha_h^\vee \rangle + 1} \right) f_i^{m_i} \cdot V_\lambda = 0
\right\}.
\]

Fix a hole $H \in \mathcal{H}_V$. Since $V'$ (or its highest weight line)
is nonzero, $\{ i \} \not\in \mathcal{H}_V$, and hence $H \setminus \{ i
\} \neq \emptyset$. Clearly, $H \setminus \{ i \}$ is also independent,
and it is easy to verify that
\[
H \setminus \{ i \} \subset J_\lambda \setminus \{ i \} \subset J_{s_i
\bullet \lambda}.
\]

It remains to verify the final defining condition for $\mathcal{H}_{V'}$
(above), for $H' := H \setminus \{ i \}$. If all nodes in $H'$ are
disconnected from $i$ -- which includes the case $i \in H$ -- then all
$f_h$ commute with $f_i$ and $\langle s_i \bullet \lambda, \alpha_h^\vee
\rangle = \langle \lambda, \alpha_h^\vee \rangle$, and the desired
equality follows from the definition of $\mathcal{H}_V$:
\begin{equation}\label{ESerre}
\left( \prod_{h \in H'} f_h^{\langle s_i \bullet
\lambda, \alpha_h^\vee \rangle + 1} \right) f_i^{m_i} \cdot V_\lambda = 0
\end{equation}

Otherwise, $i \not\in H$ (so $H' = H$) and at least one node $h \in H'$
is adjacent to $i$ in the Dynkin diagram of $\mathfrak{g}$. In this case,
it suffices to show that
\[
\left( \prod_{h \in H'} f_h^{\langle s_i \bullet
\lambda, \alpha_h^\vee \rangle + 1} \right) f_i^{m_i} \cdot V_\lambda \in
U(\mathfrak{n}^-) \prod_{h \in H'} f_h^{\langle \lambda, \alpha_h^\vee
\rangle + 1} \cdot V_\lambda.
\]

In what follows, define and use
\[
c_h := \langle s_i \bullet \lambda, \alpha_h^\vee \rangle + 1, \qquad
\forall h \in H' = H \setminus \{ i \} = H.
\]
(Recall, $m_h := \langle \lambda, \alpha_h^\vee \rangle + 1$.) Since $i
\not\in H$, one has
\begin{equation}\label{Ecalculation}
c_h = \langle \lambda, \alpha_h^\vee \rangle + 1 - m_i \langle \alpha_i,
\alpha_h^\vee \rangle = m_h + m_i | \langle \alpha_i, \alpha_h^\vee
\rangle |.
\end{equation}

Now fix any ordering of $H' = H$, say $H' = \{ h_1,
\dots, h_k \}$, and apply Lemma~\ref{LSerre}, with $M = m_{h_1}$ and $N =
m_i$. Then via~\eqref{Ecalculation},
\[
f_{h_1}^{c_{h_1}} f_i^{m_i} \cdot V_\lambda = X_1 f_{h_1}^{m_{h_1}} \cdot
V_\lambda,
\]
with $X_1 \in U(\mathfrak{n}^-)$ a linear combination of words, each with
exactly $m_i$ occurrences of $f_i$. Next,
\[
f_{h_2}^{c_{h_2}} f_{h_1}^{c_{h_1}} f_i^{m_i} \cdot V_\lambda =
f_{h_2}^{c_{h_2}} X_1 f_{h_1}^{m_{h_1}} \cdot V_\lambda =
X_2 f_{h_2}^{m_{h_2}} f_{h_1}^{m_{h_1}} \cdot V_\lambda,
\]
for some $X_2 \in U(\mathfrak{n}^-)$ as above -- again applying
Lemma~\ref{LSerre}. Repeating this procedure,
\[
\left( \prod_{h \in H'} f_h^{\langle s_i \bullet
\lambda, \alpha_h^\vee \rangle + 1} \right) f_i^{m_i} \cdot V_\lambda
\in U(\mathfrak{n}^-) \prod_{t=1}^k f_{h_t}^{m_{h_t}} \cdot V_\lambda,
\]
and this vanishes, by the definition of $H = H' \in
\mathcal{H}_V$.\medskip

\noindent \textit{\underline{Step 5:}
Concluding the proof.}\smallskip

By Step~3(ii) and the induction hypothesis for $(s_i \bullet \lambda, V',
\mu)$, there exists $J' \in \mathfrak{J}(V')$ such that $\mu \in \wt
M(s_i \bullet \lambda, J')$. 
Define $J := J' \cap J_\lambda$; then $i \not\in J$ since $i \not\in
J_{s_i \bullet \lambda}$. Now using the integrable slice
decomposition~\eqref{Eslice} twice,
\begin{align*}
\wt M(\lambda, J) = \quad \bigsqcup_{\xi\in \mathbb{Z}_{\geq 0}
\Pi_{I\setminus J}} \quad \wt L_J^{\max}(\lambda-\xi)
\supseteq &\ \bigsqcup_{m_i \alpha_i \leq \xi\in \mathbb{Z}_{\geq 0}
\Pi_{I\setminus J}} \wt L_J^{\max}(\lambda-\xi)\\
= & \ \quad \bigsqcup_{\xi'\in \mathbb{Z}_{\geq 0} \Pi_{I \setminus J}}
\quad \, \wt L_J^{\max}((\lambda-m_i \alpha_i) - \xi')\\
= &\ \wt M(s_i \bullet \lambda, J) \supseteq \wt M(s_i \bullet \lambda,
J'),
\end{align*}
noting that $J \subset J' \subset J_{s_i \bullet \lambda}$.
Thus $\mu \in \wt M(s_i \bullet \lambda, J') \subset \wt M(\lambda, J)$.
We now assert that $J \in \mathfrak{J}(V)$, which concludes the proof of
the induction step in the claim at the beginning. Indeed, if $H \in
\mathcal{H}_V$,
\[
\emptyset \subsetneq J' \cap (H \setminus \{ i \}) \subset J' \cap H
= J' \cap (J_\lambda \cap H) = J \cap H,
\]
where the first inclusion follows from Step~4, since $J' \in
\mathfrak{J}(V')$.
\end{proof}

\begin{remark}
In the spirit of Remark~\ref{Rworking1}, the above proof should have
worked with quadruples $(\mathfrak{g}, \lambda', V', \mu')$, where
$\widetilde{\mathfrak{g}} \twoheadrightarrow \mathfrak{g}
\twoheadrightarrow \overline{\mathfrak{g}}$ for a fixed generalized
Cartan matrix. We have suppressed the additional variable $\mathfrak{g}$,
since the proof only uses weight-sets of parabolic Verma modules in the
proof (and in the formula for $S(\lambda, \mathcal{H})$), and these
remain invariant across $\mathfrak{g}$.
\end{remark}
%}}}

%{{{1 Section 4 - Theorem B: Minkowski difference formula for highest weight modules
\section{{\color{black}Theorems~\ref{T2} \& \ref{Theorem wt-form. by composition series simples}: Weight formulas - Minkowski sum type \&  composition series based}}\label{ST2}

Following the proof of Theorem~\ref{T1}, this section quickly shows
Theorems\ref{T2} {\color{black} and \ref{Theorem wt-form. by composition series simples}}. We begin by isolating a key {\color{black}result}, which is
interesting in its own right: a family of Minkowski difference formulas
for parabolic Verma modules $M(\lambda,J)$, of which \eqref{Epvm} is one
extremal case.

\begin{prop}\label{Pminkowski}
Suppose $\lambda \in \mathfrak{h}^*$ and $J \subset J_\lambda$. Then for
all subsets $J'$ between $J$ and $J_\lambda$,
\begin{equation}\label{Epvm2}
\wt M(\lambda, J) = \wt L_{J'}^{\max}(\lambda) - \Z_{\geq 0} (\Delta^+
\setminus \Delta^+_J).
\end{equation}
\end{prop}

Before proving \eqref{Epvm2}, notice that if $L_{J'}^{\max}(\lambda)
\not\cong L_J^{\max}(\lambda)$ then \eqref{Epvm2} does not extend to the
level of representations via parabolic induction. For instance, in the
simplest case of $\mathfrak{g} = \mathfrak{sl}_2(\mathbb{C})$, $0 \neq
\lambda \in P^+$ (so $J_\lambda = I = \{ \alpha_1 \}$
by abuse of notation), $J = \emptyset$, and $J' = J_\lambda$, one has
\[
U\mathfrak{g} \otimes_{U\mathfrak{p}_J} L_J^{\max}(\lambda) = M(\lambda)
\cong M(0) \otimes \mathbb{C}_\lambda,
\]
where $\mathbb{C}_\lambda$ is a one-dimensional $\mathfrak{h}$-module
with eigenvalue $\lambda$. In contrast,
\[
U\mathfrak{g} \otimes_{U\mathfrak{p}_J} L_{J_\lambda}^{\max}(\lambda)
\cong M(0) \otimes L(\lambda),
\]
and this has a strictly larger character than $M(\lambda)$. Thus, the
family~\eqref{Epvm2} of Minkowski difference formulas appears to be a
novel one, and is valid on the level of weights but not for characters.

\begin{proof}[Proof of Proposition~\ref{Pminkowski}]
The formula for all $J'$ follows from the ones for $J' = J$
in~\eqref{Epvm} and for $J' = J_\lambda$, by sandwiching. Thus it
suffices to prove~\eqref{Epvm2} for $J' = J_\lambda$, i.e.\ that
\[
\wt L_J^{\max}(\lambda) - \Z_{\geq 0} (\Delta^+ \setminus \Delta^+_J) =
\wt L_{J_\lambda}^{\max}(\lambda) - \Z_{\geq 0} (\Delta^+ \setminus
\Delta^+_J).
\]
The forward inclusion is obvious, since $J \subset J_\lambda$. Next, the
parabolic Verma module over $\mathfrak{g}_{J_\lambda}$ for $(\lambda,J)$
surjects onto the maximal integrable $\mathfrak{g}_{J_\lambda}$-module
$L_{J_\lambda}^{\max}(\lambda)$. Hence by~\eqref{Epvm},
\[
\wt L_{J_\lambda}^{\max}(\lambda) \subset \wt L_J^{\max}(\lambda) -
\Z_{\geq 0} (\Delta^+_{J_\lambda} \setminus \Delta^+_J) \subset \wt
L_J^{\max}(\lambda) - \Z_{\geq 0} (\Delta^+ \setminus \Delta^+_J).
\]
Subtracting $\Z_{\geq 0} (\Delta^+ \setminus \Delta^+_J)$ from both
sides proves the reverse inclusion.
\end{proof}

\begin{remark}
We take a moment to explain how Theorem~\ref{T2} generalizes one case
in Proposition~\ref{Pminkowski}. Let $V = M(\lambda,J)$ in
Theorem~\ref{T1}. Then $\mathcal{H}_V$ is the upper-closure in ${\rm
Indep}(J_\lambda)$ of $\{ \{ j \} : j \in J \}$. Hence
$\mathbb{M}(\mathcal{H}_V) = M(0,J)$ in Theorem~\ref{T2}, and so
by~\eqref{Epvm} we recover the $J' = J_\lambda$ case of~\eqref{Epvm2}:
\[
\wt M(\lambda,J) = \wt V = \wt L_{J_\lambda}^{\max}(\lambda) + \wt M(0,J)
= \wt L_{J_\lambda}^{\max}(\lambda) - \Z_{\geq 0} (\Delta^+ \setminus
\Delta^+_J).
\]
\end{remark}

%Now one shows:

\begin{proof}[Proof of Theorem~\ref{T2}]
By Theorem~\ref{T1}, and Proposition~\ref{Pminkowski} for $J' =
J_\lambda$, and recalling $\mathfrak{J}(V)$ from~\eqref{EJV},
\[
\wt V
%= &\ \bigcup_{J \in \mathfrak{J}(V)} \left( \wt L_J^{\max}(\lambda) -
%\Z_{\geq 0} (\Delta^+ \setminus \Delta^+_J) \right)\\
= \bigcup_{J \in \mathfrak{J}(V)} \left( \wt
L_{J_\lambda}^{\max}(\lambda) - \Z_{\geq 0} (\Delta^+ \setminus
\Delta^+_J) \right)
= \wt L_{J_\lambda}^{\max}(\lambda) + \bigcup_{J \in \mathfrak{J}(V)}
- \Z_{\geq 0} (\Delta^+ \setminus \Delta^+_J).
\]

Next, consider the highest weight module $M =
\mathbb{M}(\mathcal{H}_V)$. Here $\lambda = 0$, and $\mathcal{H}_M =
\mathcal{H}_V$ by Lemma~\ref{Luhwm}. Again applying Theorem~\ref{T1},
this time to the right-hand side of~\eqref{ET2},
\begin{equation}\label{EaltB}
\wt L_{J_\lambda}^{\max}(\lambda) + \wt \mathbb{M}(\mathcal{H}_V) = 
\wt L_{J_\lambda}^{\max}(\lambda) + \bigcup_{J \in \mathfrak{J}(V)} \wt
M(0,J),
\end{equation}
and via~\eqref{Epvm}, this equals the final expression in the previous
computation.

This shows~\eqref{ET2}. For the penultimate assertion, Theorem~\ref{T1}
yields 
\begin{equation}\label{Eqn wt form. by higher order Vermas}
\wt \mathbb{M}(\lambda, \mathcal{H}_V)\  =\  \wt V
\end{equation}
(which also
implies the final assertion of course), so the map $\Psi_\lambda$ is
surjective. Also by Remark~\ref{Ruhwm}(4), the upper-closed subset
$\mathcal{H} \subset {\rm Indep}(J_\lambda)$ is proper, if and only if
$\emptyset \not\in \mathcal{H}$, if and only if $\mathbb{M}(\lambda,
\mathcal{H}) \neq 0$. Now to show injectivity, suppose $\mathcal{H}_1
\neq \mathcal{H}_2$ are proper upper-closed subsets of ${\rm
Indep}(J_\lambda)$.
Choose a minimal set $H$ in their symmetric difference, say $H \in
\mathcal{H}_1$. Then the one-dimensional weight space
\[
\prod_{h \in H} f_h^{\langle \lambda, \alpha_h^\vee \rangle + 1} \cdot
M(\lambda)_\lambda
\]
is easily seen to be quotiented in $\mathbb{M}(\lambda, \mathcal{H}_1)$
but not in $\mathbb{M}(\lambda, \mathcal{H}_2)$. Thus (e.g.\ by
Remark~\ref{RT1}),
\[
\lambda - \sum_{h \in H} (\langle \lambda, \alpha_h^\vee \rangle + 1)
\alpha_h \in \wt \mathbb{M}(\lambda, \mathcal{H}_2) \setminus \wt
\mathbb{M}(\lambda, \mathcal{H}_1),
\]
and so the map $\Psi_\lambda$ is injective as well.
\end{proof}
\begin{proof}[{Proof of Theorem \ref{Theorem wt-form. by composition series simples}}]
 {\color{black} 
 Observe for any $M(\lambda)\twoheadrightarrow V$,  the result in part (b) of the theorem immediately shows (a).
The proof of part (b) involves some standard arguments for integrable simples over parabolic subalgebras, and the proof strategy in \cite[Proposition 1.16]{MDWF}.
Recall the usual $\mathbb{Z}_{\leq 0}\Pi$-gradation of $U(\mathfrak{n}^-)$.\\
Useful \underline{fact} : \ \ \ Fix $\lambda'\in \mathfrak{h}^*$ and let $\phi: M(\lambda')\longrightarrow L(\lambda')$ be the canonical $\mathfrak{g}$-module map.
Then
                 \begin{equation}\label{Eqn simplicity criterion}
\text{we have \ for any } x \in M(\lambda') \ \ -\qquad  \phi(x)\neq 0 \ \text{in}\  L(\lambda') \quad \iff \quad  m_{\lambda'}\ \in  \ U(\mathfrak{n}^+)\cdot x.
\end{equation}
By the simplicity of $L(\lambda')$, the same result holds true if we replace $m_{\lambda'}$ on the r.h.s. of the above equivalence by any $y\in M(\lambda')$ with $\phi(y)\neq 0$; since $L(\lambda')_{\lambda'}\subseteq U(\mathfrak{n}^+)\phi(y)$.
To see fact \eqref{Eqn simplicity criterion}, let $X:=U(\mathfrak{n}^+)x\  \supset\  \{x\}\ \neq 0$.
$U(\mathfrak{n}^-)X$ is a submodule of $M(\lambda')$ by the triangular decomposition of $\mathfrak{g}$.
Further, $\big(U(\mathfrak{n}^-)X\big)_{\lambda'}\subseteq X_{\lambda'}$, since $\wt\big(U(\mathfrak{n}^-)X\big) \preceq \wt X \preceq \lambda'$.
If both $\phi(x)\neq 0 \text{ in }L(\lambda')$ and $X_{\lambda'}=0$ happen, then $U(\mathfrak{n}^-)\phi(X)\neq 0$ leads to a proper submodule in $L(\lambda')$ $\Rightarrow\!\Leftarrow$.
So
\begin{equation}\label{Eqn raising property of simples for proposition (b)}
\text{ given } \ 0\neq \ v \ \in \ L(\lambda'), \qquad\quad  e_{j_1}\ \cdots \ e_{j_r} \ \cdot \ v \ \neq 0 \ \  \in L(\lambda)_{\lambda} \quad \text{ for some }\ \ j_1,\ldots, j_r\ \in\ I.
\end{equation}
Conversely, if $X_{\lambda'}\neq 0$, then $U(\mathfrak{n}^-)X = M(\lambda')$ by definitions, and so $\phi(x)\neq 0\text{ in } L(\lambda')$.

We begin the proof of part (b) by fixing $\mu\longleftrightarrow (c_i)_{i\in I},\  J,\ H$ (possibly $=\emptyset$) as in the theorem.
Note $\lambda-\mathbb{Z}_{\geq 0}\Pi_{J^c}\subseteq \wt V$ as $J^c$ avoids every minimal hole of $V$, as explained in the proof of Proposition~\ref{Pinclusion}.
Next we fix the maximal submodule $V_1$ of the highest weight (sub-)module $U(\mathfrak{n}^-)\prod_{h\in H}f_h^{\lambda(\alpha_h^{\vee})+1}v_{\lambda}$ inside $V$; for a highest weight vector $0\neq v_{\lambda}\in V_{\lambda}$.
  By the $W_J$-conjugation if necessary, we assume $\mu$ to be dominant integral w.r.t. $\Pi_J$.
 The non-degeneracy of $\mu$ w.r.t. $\lambda-\sum_{i\in J^c}c_i\alpha_i\in \wt V$ implies by \cite[Proposition 11.2]{Kac book} over $\mathfrak{g}_J$ that $\mu\in \wt L_J\Big(\lambda -\sum_{i\in J^c}c_i\alpha_i\Big)$.
 We assume henceforth $J^c\cap \supp(\lambda-\mu)\neq \emptyset$; else we are done by $\mu\in \wt L_J(\lambda)\subseteq \wt L(\lambda)$.

Let us fix using \cite[Algorithm 7.3]{MDWF} --  or \cite[Algorithm 1.15]{MDWF} when $\lambda\in P^+$ -- an enumeration $\big(J^c \cap \supp(\lambda-\mu) \big)\cap J_{\lambda}=\{i_1,\ldots, i_n\}$ and its terminating independent set $H=\{i_{m+1}, \ldots i_n\}$ for $m< n$;
we assume $H\neq \emptyset$ as otherwise we are done as in the last line of the above paragraph. 
Next when $\big(J^c\cap \supp(\lambda-\mu)\big)\cap J_{\lambda}^c\neq \emptyset$, we fix for it any enumeration $\{i_{n+1},\ldots, i_r\}$.
By these enumerations, we have that the weight vector $v:= f_{i_1}^{c_{i_1}}\cdots f_{i_m}^{c_{i_m}} \Big(\prod\limits_{h\in H}f_h^{c_h}\Big) \big(f_{i_{n+1}}^{c_{i_{n+1}}}\cdots f_{i_r}^{c_{i_r}}\big) v_{\lambda}$ survives in $V_{\lambda-\sum_{i\in J^c}c_i\alpha_i }$. 
The proof of this is the same as that of \cite[Non-vanishing property/equation (7.12)]{MDWF}, namely: $\Big( \prod_{h\in H}e_h^{c_{h}-\lambda\big(\alpha_h^{\vee}\big)-1} \Big)\big(e_{i_{n+1}}^{c_{i_{n+1}}}\cdots e_{i_r}^{c_{i_r}}\big)  \big(e_{i_m}^{c_{i_m}} \cdots e_{i_1}^{c_{i_1}}\big)\cdot v \ \in \big(\mathbb{C}\setminus\{0\}\big)\prod\limits_{h\in H}f_h^{\lambda(\alpha_h^{\vee})+1}v_{\lambda}\ = V_{s_{H}\bullet \lambda}\setminus\{0\}$.
So by \eqref{Eqn simplicity criterion}, $v$ must survive in the simple subquotient $\frac{U(\mathfrak{n}^-)\prod\limits_{h\in H}f_h^{\lambda\big(\alpha_h^{\vee}\big)+1}v_{\lambda}}{V_1}\underset{\mathfrak{g}-mod}{\simeq} L\big(s_H \bullet\lambda\big)$.
This means \big(as in the proof of  \cite[Proposition 1.16 (c)]{MDWF}\big)  $\lambda-\sum_{i\in J^c}c_i\alpha_i\in \wt L\big(s_H\bullet \lambda\big)\subseteq \wt V$.
On the other hand, observe that the weight vector $v\in V_{\lambda-\sum_{i\in J^c}c_i\alpha_i}$ is maximal for the $\mathfrak{n}^+_J$-action, and so it generates the highest weight module $U(\mathfrak{n}^-_J)v$ over $\mathfrak{g}_J$.
So $\wt L_J\Big(\lambda -\sum_{i\in 
 J^c}c_i\alpha_i\Big)\ \subseteq \wt \big(U(\mathfrak{n}^+_J)v\big)\ \subseteq \wt V$.
 Moreover by $\mu\in \wt L_J\Big(\lambda-\sum_{i\in J^c}c_i\alpha_i\Big)$, we have a homogeneous element $F\in U(\mathfrak{n}^-)_{-\sum_{j\in J}c_j\alpha_j}$ and (not necessarily distinct) nodes $j_1,\ldots j_r\in J$ with $e_{j_1}\cdots e_{j_r} F v\in (\mathbb{C}\setminus \{0\})v$.
Putting together the above observed non-vanishings under the raising actions along both $J$ and $J^c$ directions, fact \eqref{Eqn simplicity criterion} shows $\mu\in \wt L\big(s_H\bullet \lambda\big)$, as required.
 }
\end{proof}
%}}}

%{{{1 Section 5 - Theorem C; higher order approximations, integrability,
%and weight-formulas
\section{Theorem~\ref{T3}: Higher order approximations, integrability,
and weight-formulas}\label{ST3}

This section begins by proving our next main result in somewhat greater
detail than may be necessary, to help understand the subsequent examples.
Following the examples, the three subsections discuss the other parts of
the section-title, for every highest weight $\mathfrak{g}$-module.

\begin{proof}[Proof of Theorem~\ref{T3}]
Given two highest weight modules $V,V'$ with common highest weight
$\lambda$, Theorem~\ref{T2} asserts:
(1)~$\wt V = \wt V'$ if and only if $\mathcal{H}_V = \mathcal{H}_{V'}$;
and (2)~$\mathbb{M}(\lambda, \mathcal{H}_V) \twoheadrightarrow V$, with
equality of weights. This reasoning implies that $V^{\max}(\wt V) =
\mathbb{M}(\lambda, \mathcal{H}_V)$.

To construct $V^{\min}$, we adopt a more ``natural''
notation. By Theorem~\ref{T2}, instead of $S = \wt V$, one can
equivalently work via $\Psi_\lambda^{-1}$ with proper upper-closed
subsets $\mathcal{H}'$ of ${\rm Indep}(J_\lambda)$.
(Thus $V^{\max}(\mathcal{H}') = \mathbb{M}(\lambda, \mathcal{H}')$.)
Now the construction of $V^{\min}(\mathcal{H}')$ is a familiar one: first
define $N(\lambda, \mathcal{H}') \subset M(\lambda)$ to be the sum of all
submodules $N$ of $M(\lambda)$ for which the weight space
\begin{equation}\label{Econdition}
N_{\lambda_H} = 0, \ \ \text{where} \ \ \lambda_H = \lambda -
\sum_{h \in H} (\langle \lambda, \alpha_h^\vee \rangle + 1 ) \alpha_h,
\qquad \forall H \in {\rm Indep}(J_\lambda) \setminus \mathcal{H}'.
\end{equation}
Next, define the highest weight module
\begin{equation}
\mathbb{L}(\lambda, \mathcal{H}') := M(\lambda) / N(\lambda,
\mathcal{H}').
\end{equation}

We claim this is precisely $V^{\min}(\mathcal{H}')$. Indeed, note
$\mathcal{H}_{\mathbb{L}(\lambda, \mathcal{H}')} = \mathcal{H}',$ which
shows via $\Psi_\lambda$ that $\wt \mathbb{L}(\lambda, \mathcal{H}') =
\wt \mathbb{M}(\lambda, \mathcal{H}')$, proving one implication in the
desired property. 
%To show this claim, if $H \in \mathcal{H}'$ then $N = U\mathfrak{g}
%\bff_H \cdot M(\lambda)$ can be verified to
%satisfy~\eqref{Econdition}, so that $H \in
%\mathcal{H}_{\mathbb{L}(\lambda, \mathcal{H}')}$. If instead $H \not\in
%\mathcal{H}'$ then $N(\lambda, \mathcal{H}') = \sum_N 0 = 0$, so $\bff_H
%\cdot \mathbb{L}(\lambda, \mathcal{H}')_\lambda = \bff_H \cdot
%M(\lambda)_\lambda \neq 0$. Hence $H \not\in
%\mathcal{H}_{\mathbb{L}(\lambda, \mathcal{H}')}$, showing the claim.
%
Next, if $\wt V = \wt \mathbb{M}(\lambda, \mathcal{H}')$, consider the
exact sequence $0 \to N_V \to M(\lambda) \to V \to 0$. Since
$\mathcal{H}_V = \mathcal{H}'$ (via $\Psi_\lambda$), the definition of
$\mathcal{H}_V$ implies~\eqref{Econdition} for $N = N_V$. But then $N_V
\subset N(\lambda, \mathcal{H}')$, so $V \twoheadrightarrow
\mathbb{L}(\lambda, \mathcal{H}')$.
\end{proof}

\begin{remark}\label{RLwts}
Since $\wt \mathbb{L}(\lambda, \mathcal{H}_V) = \wt V = \wt
\mathbb{M}(\lambda, \mathcal{H}_V)$, the finite collection $\{ \wt
\mathbb{L}(\lambda, \mathcal{H}_V) \}_V = \{ \wt \mathbb{L}(\lambda,
\mathcal{H}) \}_{\mathcal{H}}$ also exhausts all weight-sets of highest
weight modules $M(\lambda) \twoheadrightarrow V$.
\end{remark}

\begin{example}
As a special case of Remark~\ref{RLwts}: say $\mathfrak{g} =
\mathfrak{sl}_2^{\oplus n}$ for some $n \geq 1$. Then for every highest
weight module $M(\lambda) \twoheadrightarrow V$, we have
$\mathbb{M}(\lambda, \mathcal{H}_V) = V = \mathbb{L}(\lambda,
\mathcal{H}_V)$, since their weights agree and all weight spaces are
$1$-dimensional. (In particular, $\mathbb{M}(\lambda, \mathcal{H}) =
\mathbb{L}(\lambda, \mathcal{H})$ for all $\lambda \in \mathfrak{h}^*$
and $\mathcal{H}$.)
\end{example}

As mentioned in the proof, the construction of $\mathbb{L}(\lambda,
\mathcal{H}')$ should sound familiar to the reader. We illustrate with
three special cases from previous literature, classical and modern; the
third provides an alternate proof/solution to a question (unpublished)
posed by Lepowsky, as we explain below.

\begin{example}
The original ``zeroth order'' construction (as is explained presently)
along these lines is that of the simple module $L(\lambda)$. Indeed, that
is the special case where one quotients $M(\lambda)$ by the sum of all
proper submodules $N$, i.e.\ submodules $N$ for which $N_\lambda = 0$.
Thus,
\begin{equation}\label{E0thorder}
L(\lambda) = \mathbb{L}(\lambda, \mathcal{H}'), \quad \text{where} \quad
\mathcal{H}' = {\rm Indep}(J_\lambda) \setminus \{ \emptyset \}.
\end{equation}
\end{example}

\begin{example}
The second-named author recently showed in \cite[Theorem B]{MDWF} the
existence of $V^{\max}(S), V^{\min}(S)$ for $S = \wt M(\lambda,J)$.
Clearly, this is a special case of Theorem~\ref{T3}.
\end{example}

The third example is slightly different in flavor (and motivates the next
two subsections); it comes from \cite{DK}. The authors first explain that
associated to every module is its \textit{first order information}:

\begin{theorem}[Dhillon--Khare, \cite{DK}]\label{T1storder}
Given $\lambda \in \mathfrak{h}^*$ and a highest weight
$\mathfrak{g}$-module $M(\lambda) \twoheadrightarrow V$, the following
``first order data'' are equivalent, i.e., can each be recovered from the
others:
\begin{enumerate}
\item The integrability,
$J_V := \{ i \in J_\lambda \ | \ f_i^{\langle \lambda, \alpha_i^\vee
\rangle + 1} V_\lambda = 0 \}$ (as in~\eqref{Eint}).

\item The Weyl group symmetry of $\wt V$.

\item The convex hull $\conv (\wt V)$.
\end{enumerate}
Thus the parabolic Verma module $M(\lambda, J_V) \twoheadrightarrow V$,
and they have the same convex hull of weights.

Moreover, for simple or parabolic Verma modules $V$, these data are
further equivalent to $(4)$~the weights of $V$.
\end{theorem}

Dhillon--Khare next write down what is our third example. Notice in the
definition of $\mathcal{H}_V$ in Theorem~\ref{T1}, the highest weight
$\lambda$ is the ``$0$th order hole'' in $V$. Given
Theorem~\ref{T1storder}, we would similarly like to call the integrable
simple directions $J_V$ the ``$1$st order holes''. This is supported by

\begin{example}\label{ExDK}
Dhillon--Khare \cite{DK} (following Khare \cite{Kh1} in some cases)
showed the existence of unique largest and smallest modules $M(\lambda,
J)$ and $L(\lambda,J)$, respectively, with a prescribed integrability $J
\subset J_\lambda$ -- or by Theorem~\ref{T1storder}, with a prescribed
shape of the convex hull of weights. The construction is as above:
$M(\lambda,J) = \mathbb{M}(\lambda, \{ \{ j \} : j \in J \}) =
\mathbb{M}(\lambda, \{ H \in \mathcal{H}_{M(\lambda,J)} : |H| \leq 1
\})$, and the authors introduced $L(\lambda, J)$. In the language of this
paper,
\begin{equation}\label{E1storder}
L(\lambda,J) = \mathbb{L}(\lambda, \mathcal{H}'), \quad \text{where} \quad
\mathcal{H}' = {\rm Indep}(J_\lambda) \setminus \left( \{ \emptyset \}
\sqcup \{ \{ i \} : i \in J_\lambda \setminus J \} \right).
\end{equation}
Note that $\mathcal{H}'$ is indeed upper-closed in ${\rm
Indep}(J_\lambda)$.
\end{example}

\subsection{Universal modules approximating a highest weight module}

Example~\ref{ExDK} refines the familiar chain of surjections
$M(\lambda) \twoheadrightarrow V \twoheadrightarrow L(\lambda)$ to
\[
M(\lambda) \twoheadrightarrow M(\lambda,J_V) \twoheadrightarrow V
\twoheadrightarrow L(\lambda,J_V) \twoheadrightarrow L(\lambda).
\]
The ``zeroth''/outermost boundary-terms share the same highest weight as
$V$, while the ``first''/next inner terms share the same integrability as
well. Above, we have now produced a refinement of this chain, by
replacing the innermost $V$ by the surjections
\[
\cdots \twoheadrightarrow \mathbb{M}(\lambda, \mathcal{H}_V)
\twoheadrightarrow V
\twoheadrightarrow \mathbb{L}(\lambda, \mathcal{H}_V) 
\twoheadrightarrow \cdots
\]
These refinements, and the comments after Theorem~\ref{T1storder},
motivate us to define the following chain of highest weight modules:

\begin{defn}\label{Dapprox}
Fix $\mathfrak{g}, \lambda$, and an upper-closed subset $\mathcal{H}
\subset {\rm Indep}(J_\lambda)$, and denote by $\mathcal{H}^c$ its
complement. Given an integer $0 \leq k \leq \infty$, define the universal
``upper'' and ``lower'' modules
\begin{align}
\mathbb{M}_k(\lambda, \mathcal{H}) := &\ \mathbb{M}(\lambda, \{ H \in
\mathcal{H} : |H| \leq k \}),\label{EMk}\\
\mathbb{L}_k(\lambda, \mathcal{H}) := &\ \mathbb{L}(\lambda, \{ H \in
\mathcal{H}^c : |H| \leq k \}^c).\label{ELk}
\end{align}
(Note, the $\mathcal{H}$-set on the right in~\eqref{ELk} is
upper-closed.)
Now given a module $M(\lambda) \twoheadrightarrow V$, define its
\textit{$k$th order upper- and lower-approximations} to be
$\mathbb{M}_k(\lambda, \mathcal{H}_V)$
and $\mathbb{L}_k(\lambda, \mathcal{H}_V)$, respectively.
\end{defn}

Once the definitions are in place, the following is straightforward.

\begin{prop}\label{P01approx}
For any nonzero module $M(\lambda) \twoheadrightarrow V$,
\begin{alignat}{3}
& \mathbb{M}_0(\lambda, \mathcal{H}_V) = M(\lambda), & \qquad &
\mathbb{M}_1(\lambda, \mathcal{H}_V) = M(\lambda, J_V),\\
& \mathbb{L}_0(\lambda, \mathcal{H}_V) = L(\lambda), & &
\mathbb{L}_1(\lambda, \mathcal{H}_V) = L(\lambda, J_V).
\end{alignat}
\end{prop}

This explains the precise sense in which Verma modules and parabolic
Verma modules are the $0$th and $1$st order upper-approximations,
respectively, of every highest weight module.

\begin{remark}
The modules $\mathbb{M}_k, \mathbb{L}_k$ clearly refine the above chain
of surjections, since the $\mathcal{H}$-sets in~\eqref{EMk} are
increasing in $k$, and in~\eqref{ELk} are decreasing in $k$. That is, the
$\mathcal{H}$-sets in all terms in
\[
M(\lambda) \twoheadrightarrow M(\lambda,J_V) \twoheadrightarrow
\cdots \twoheadrightarrow \mathbb{M}(\lambda, \mathcal{H}_V)
\twoheadrightarrow V
\twoheadrightarrow \mathbb{L}(\lambda, \mathcal{H}_V) 
\twoheadrightarrow \cdots
\twoheadrightarrow L(\lambda,J_V) \twoheadrightarrow L(\lambda)
\]
(except the central $V$) increase from $\emptyset$ at the left, to ${\rm
Indep}(J_\lambda) \setminus \{ \emptyset \}$ at the right.
Moreover, the above chain of the $\mathbb{M}_k$ stabilizes, in that
$\mathbb{M}(\lambda, \mathcal{H}_V) = \mathbb{M}_\infty(\lambda,
\mathcal{H}_V) = \mathbb{M}_K(\lambda, \mathcal{H}_V)$, where $K$ is the
size of any largest (in size) independent set in $J_\lambda$. Similarly
for the chain of $\mathbb{L}_k$.
\end{remark}

As an \textbf{application} of these modules, we (re-)solve when the
integrability of a highest weight module determines its weights. Here is
one of the main results of Dhillon--Khare \cite{DK}:

\begin{theorem}\label{Tlepowsky}
Fix ($\mathfrak{g}$ and) a highest weight $\lambda$.
The integrability $J \subset J_\lambda$ of a highest weight module
determines its weights, if and only if the Dynkin diagram of $J_\lambda
\setminus J$ is a complete graph.
\end{theorem}

This affirmatively answers a question asked by Lepowsky \cite{lepo2} (to
Khare) in connection with Theorem~\ref{T1storder}.
Namely, Lepowsky asked (in the language of this paper) whether or not the
holes in the weight-sets obtained from ``potential integrability'' as in
Example~\ref{Exrank2}, are the only obstructions to determining the set
of weights. Theorem~\ref{Tlepowsky} is now transparent from our main
results and higher order Verma modules -- indeed, the calculation now
reduces to one of set-theory:

\begin{proof}
By the construction of $L(\lambda,J)$ (in \cite{DK}, or
see~\eqref{E1storder}), it suffices to understand if $\wt M(\lambda, J) =
\wt L(\lambda, J)$ (or from above, when $\wt \mathbb{M}_1(\lambda,
\mathcal{H}) = \wt \mathbb{L}_1(\lambda, \mathcal{H})$). Via
Theorem~\ref{T2} and $\Psi_\lambda^{-1}$, this is if and only if the
upper-closure of the singleton sets in $J$ equals the right-hand side
of~\eqref{E1storder}, i.e., 
\[
{\rm Indep}(J_\lambda) \setminus {\rm Indep}(J_\lambda \setminus J) =
{\rm Indep}(J_\lambda) \setminus \left( \{ \emptyset \}
\sqcup \{ \{ i \} : i \in J_\lambda \setminus J \} \right).
\]
Taking complements, this happens if and only if $J_\lambda \setminus J$
is complete.
\end{proof}

\subsection{Higher order integrability, and stratification of the set of
highest weight modules}\label{Sspeculate}

The vigilant reader may have noticed that for the titular ``universal''
modules $\mathbb{M}_k(\lambda, \mathcal{H}), \mathbb{L}_k(\lambda,
\mathcal{H})$, while Theorem~\ref{T3} describes how they are individually
universal, we have not mentioned the sense in which (for each fixed $k$)
they are jointly so! (And indeed, the individual universalities disagree,
since the upper-closed -- equivalently, minimal -- $\mathcal{H}$-sets for
$\mathbb{M}_k(\lambda, \mathcal{H})$ and $\mathbb{L}_k(\lambda,
\mathcal{H})$ are unequal.)

We now explain the sought-for common universality, for each $k>0$ (for
$k=0$ it is simply the identification of the highest weight). The next
result shows in particular that this universality refines the common
universal property for the equal-weighted modules $\mathbb{M}(\lambda,
\mathcal{H}), \mathbb{L}(\lambda, \mathcal{H})$.

\begin{prop}[$k$th order universal property]\label{PMkLk}
Fix Kac--Moody $\mathfrak{g}$, a weight $\lambda \in \mathfrak{h}^*$, and
an integer $k \geq 1$. For a subset $\mathcal{X} \subset {\rm
Indep}(J_\lambda)$, write $\mathcal{X}_{\leq k}$ for the subset $\{ H \in
\mathcal{X} : |H| \leq k \}$.
\begin{enumerate}
\item Given an upper-closed subset $\mathcal{H} \subset {\rm
Indep}(J_\lambda)$, there exist unique smallest and largest upper-closed
subsets $\underline{\mathcal{H}}, \overline{\mathcal{H}}$ respectively,
such that
$\underline{\mathcal{H}}_{\leq k} = \mathcal{H}_{\leq k} =
\overline{\mathcal{H}}_{\leq k}$.

Moreover, these sets are precisely the upper-closures of the ones
occurring in the definitions of the modules
$\mathbb{M}_k(\lambda, \mathcal{H}), \mathbb{L}_k(\lambda, \mathcal{H})$
respectively, i.e.\ in Equations~\eqref{EMk} and~\eqref{ELk}.

\item In particular, for each upper-closed subset $\mathcal{H}$,
$\mathbb{M}_k(\lambda, \mathcal{H})$ and $\mathbb{L}_k(\lambda,
\mathcal{H})$ are the unique largest and smallest highest weight modules
with ``$k$th order integrability data'' $\mathcal{H}_{\leq k}$.
For $|J_\lambda| \leq k \leq \infty$, this specializes to
Theorem~\ref{T3}, i.e.\ the common universal property for 
$\mathbb{M}(\lambda, \mathcal{H}), \mathbb{L}(\lambda, \mathcal{H})$.

\item Any two intervals $[\mathbb{L}_k(\lambda, \mathcal{H}),
\mathbb{M}_k(\lambda, \mathcal{H})]$ and $[\mathbb{L}_k(\lambda,
\mathcal{H}'), \mathbb{M}_k(\lambda, \mathcal{H}')]$ are disjoint or
equal.
\end{enumerate}
\end{prop}

\begin{proof}
Briefly: once there is a claimed formula for the sets
$\underline{\mathcal{H}}_{\leq k}, \overline{\mathcal{H}}_{\leq k}$, it
is easily verified. This shows~(1).
The first claim in~(2) is now the common universal property of the pair
$(\mathbb{M}_k(\lambda, \mathcal{H}), \mathbb{L}_k(\lambda,
\mathcal{H}))$, and follows from Theorem~\ref{T3}. The second claim
follows from the definitions, since $\underline{\mathcal{H}}_{\leq k} =
\overline{\mathcal{H}}_{\leq k} = \mathcal{H}$ if $k \geq |J_\lambda|$.
Finally, (3) is immediate: if $M(\lambda) \twoheadrightarrow V$ is such
that $(\mathcal{H}_V)_{\leq k} = \mathcal{H}_{\leq k} =
\mathcal{H}'_{\leq k}$, then
\[
\mathbb{M}_k(\lambda, \mathcal{H}) =
\mathbb{M}(\lambda, \mathcal{H}_{\leq k}) =
\mathbb{M}(\lambda, \mathcal{H}'_{\leq k}) =
\mathbb{M}_k(\lambda, \mathcal{H}')
\]
where $\mathcal{H}, \mathcal{H}'$ are upper-closed without loss of
generality. A similar proof works for the $\mathbb{L}_k$.
\end{proof}

If the restriction of upper-closedness is removed from
Proposition~\ref{PMkLk}(1), then $\overline{\mathcal{H}}_{\leq k}$
remains unchanged. Moreover, the $\mathbb{M}_k$-module would remain
unchanged even if $\underline{\mathcal{H}}_{\leq k}$ is reduced to the
subsets of $\mathcal{H}^{\min}$ of size $\leq k$, where
$\mathcal{H}^{\min}$ denotes the subset of $\mathcal{H}$ of minimal
(hence pairwise incomparable) elements. As a consequence of this and
Proposition~\ref{PMkLk}, we now introduce

\begin{defn}\label{Dhigher}
Given $\mathfrak{g}, \lambda$, and a module $M(\lambda)
\twoheadrightarrow V$, define the \textit{$k$th order integrability} of
$V$ for an integer $k \geq 0$, to be $(\mathcal{H}^{\min}_V)_{\leq k}$ if
$k>0$ and $\lambda$ if $k=0$.
\end{defn}

This is precisely what is captured by the Verma module $M(\lambda)$ when
$k=0$, and by the parabolic Verma module $M(\lambda, J_V)$ when $k=1$.

\begin{remark}\label{Rstrata}
The modules $\mathbb{M}_k, \mathbb{L}_k$ serve to ``stratify'' the poset
(under surjection) $\mathfrak{X} := {\tt HW}(\mathfrak{g})$ of highest
weight $\mathfrak{g}$-modules with all highest weights. Clearly,
$U\mathfrak{g} / \sum_{i \in I} (U \mathfrak{g}) e_i$
is the ``infinity'' element, surjecting onto all of $\mathfrak{X}$.
At the zeroth level, $\mathfrak{X} = \bigsqcup_{\lambda \in
\mathfrak{h}^*} \mathfrak{X}_\lambda$, with $\mathfrak{X}_\lambda =
[L(\lambda), M(\lambda)]$ the interval of modules with highest weight
$\lambda$. 
(Moreover, Theorem~\ref{T2} shows passing to weights yields a finite set
from each $\mathfrak{X}_\lambda$.)
At the next level, the modules in $\mathfrak{X}_\lambda$ are partitioned
by integrability:
\[
\mathfrak{X}_\lambda = \bigsqcup_{J \subset J_\lambda}
[L(\lambda, J), M(\lambda, J)]
 = \bigcup_{\substack{\mathcal{H} \subset \, {\rm Indep}(J_\lambda) \\
 {\rm upper\ closed}} }
[\mathbb{L}_1(\lambda, \mathcal{H}), \mathbb{M}_1(\lambda, \mathcal{H})]
\]
One can continue sub-stratifying each stratum, via
Proposition~\ref{PMkLk}, with the following results.
(i)~At each stage, one obtains a partition into intervals
$[\mathbb{L}_k(\lambda, \mathcal{H}), \mathbb{M}_k(\lambda,
\mathcal{H})]$.
(ii)~One sub-stratifies $\mathfrak{X}_\lambda$ only $K$-many times, with
$K$ the size of any largest (in size) independent set in $J_\lambda$.
(iii)~At the final, innermost level, every partitioned set
$[\mathbb{L}(\lambda, \mathcal{H}), \mathbb{M}(\lambda, \mathcal{H})]$
comprises highest weight modules with the same set of weights (and hence,
the same apex $\lambda$, convex hull/integrability
$(\mathcal{H}_V^{\min})_{\leq 1}$, \dots).
\end{remark}

Remark~\ref{Rstrata} leads to the following concrete recipe. To
successively better approximate the set $\wt V = \wt \mathbb{M}(\lambda,
\mathcal{H}_V) = \wt \mathbb{M}_\infty(\lambda, \mathcal{H}_V)$, one
starts with $U \mathfrak{g} / \sum_{i \in I} (U \mathfrak{g}) e_i$. The
$0$th order hole, i.e.\ $\lambda$, uniquely fixes the apex of the highest
weight cone containing $\wt V$. Next, the $1$st order upper-approximation
refines the dimension $1$ faces, i.e.\ edges, of the $0$th order shape
$\conv (\wt M(\lambda))$ by truncating the semi-infinite rays at
$\lambda$ along the $J_V$ directions. And so on.

Viewed representation-theoretically, the integrability refines the
weights of some of the rank-$1$ highest weight submodules $\{ U
\mathfrak{g}_i \cdot M(\lambda)_\lambda : i \in J_V \}$.
Similarly, the $k$th order approximation/integrability refines the
submodules $\{ U \mathfrak{g}_H \cdot M(\lambda)_\lambda : H \in
(\mathcal{H}_V^{\min})_{\leq k} \}$ for \textit{minimal} holes $H$ of
size at most $k$, by truncating an interior portion of the corresponding
faces of $\conv (\wt V)$ containing $\lambda$ (and more from the
interior of $\conv( \wt V)$). This refining is transferred to other
vertices / faces of $\conv(\wt V)$ via the Weyl group symmetries of $\wt
V$.

\subsection{$k$th order weight-formulas}\label{S5final}

In the above spirit, we end this section with ``$k$th order
weight-formulas'' that specialize to Theorems~\ref{T1} and~\ref{T2}.
First recall from the proof of Theorem~\ref{T1} that
\begin{equation}\label{Eadm}
\wt V \quad = \bigcup_{J \subset J_\lambda\, :\, J \cap H \neq
\emptyset\; \forall H \in \mathcal{H}_V} \wt M(\lambda,J) \quad = \quad
\wt \mathbb{M}(\lambda, \mathcal{H}_V).
\end{equation}

Given the above discussion in this section, the first equality shows $\wt
V$ to be a union of weights of many ``first order Verma
modules'', i.e.\ parabolic Vermas $M(\lambda, J) = \mathbb{M}(\lambda, \{
\{ j \} : j \in J \})$. The second shows $\wt V$ to be the weight-set of
exactly one ``$\infty$-order Verma module'', i.e.\ $\mathbb{M}(\lambda,
\mathcal{H}_V)$.

\begin{remark}
Just as a $0$th order Verma module $M(\lambda)$ is also a $1$st order
Verma module $M(\lambda, J)$, we adopt the convention that
$\mathbb{M}(\lambda, \mathcal{H})$ is a $k$th order Verma module (for $0
\leq k \leq \infty$) if $|H| \leq k$ for all $H \in \mathcal{H}^{\min}$.
This is compatible with the $k$th order approximations $\mathbb{M}_k,
\mathbb{L}_k$ earlier in this section.
\end{remark}

Given these weight-formulas and remarks, it is natural to ask if there
exist ``intermediate'' $k$th order formulas for each $k \geq 1$. Namely,
formulas that show $\wt V$ is a finite union of sets of the form $\wt
\mathbb{M}(\lambda, \mathcal{H}^{(k)}_i)$, $i \geq 1$ where each hole in
each $\mathcal{H}^{(k)}_i \subset {\rm Indep}(J_\lambda)$ has size at
most $k$.

Our next result provides a positive answer for each $k$.
These are positive weight-formulas for $\wt V$ that naturally interpolate
from only singletons in each $\mathcal{H}^{(1)}_i$ (i.e.\ parabolic
Vermas $\wt M(\lambda, J)$ in~\eqref{Eweights}) to the unique set
$\mathcal{H}^{(\infty)} = \mathcal{H}_V$ (i.e.\ $\wt V = \wt
\mathbb{M}(\lambda, \mathcal{H}_V)$). We need the following notion.

\begin{defn}\label{Dadm}
Enumerate a subset $\mathcal{H}_0 = \{ H_1, \dots, H_N \} \subset {\rm
Indep}(I)$. Given an integer $1 \leq k \leq \infty$, we say a set
$\mathcal{H}^{(k)} \subset {\rm Indep}(I)$ is
\textit{$\mathcal{H}_0$-admissible of order $k$} if there are subsets
$H'_t \subset H_t$ of size $\min(k, |H_t|)$, such that
$\mathcal{H}^{(k)}$ consists of the distinct sets among $H'_1, \dots,
H'_N$.
\end{defn}

\begin{example}\label{Exadm}
Here are two examples where $\mathcal{H}_0 = \mathcal{H}_V$ for a highest
weight $\mathfrak{g}$-module $M(\lambda) \twoheadrightarrow V$. First, if
$k=1$ then an order $1$ $\mathcal{H}_V$-admissible set
$\mathcal{H}^{(1)}$ consists of singleton sets $\{ j \}$ for $j \in
J_\lambda$, such that every hole in $\mathcal{H}_V$ contains one or more
of these $j$, and every $j$ is in at least one $H \in \mathcal{H}_V$. The
universal module for this first order hole-set is indeed the ``first
order'' parabolic Verma module: $\mathbb{M}(\lambda, \mathcal{H}^{(1)}) =
M(\lambda, J)$, where $J = \{ j : \{ j \} \in \mathcal{H}^{(1)} \}$.

The second example is when $k = \infty$ (or $k \gg 0$). Then there is
only one $\mathcal{H}_V$-admissible set of order $k$:
$\mathcal{H}^{(\infty)} = \mathcal{H}_V$ itself. The corresponding
higher order Verma module is $\mathbb{M}(\lambda, \mathcal{H}_V)$.
\end{example}

\begin{prop}\label{Padm}
Given a $\mathfrak{g}$-module $M(\lambda) \twoheadrightarrow V$ (for
arbitrary $\mathfrak{g}, \lambda$), and an integer $1 \leq k \leq
\infty$,
\begin{equation}\label{Eadm2}
\wt V = \bigcup_{\mathcal{H}^{(k)}}
\wt \mathbb{M}(\lambda, \mathcal{H}^{(k)}),
\end{equation}
where the union runs over all $\mathcal{H}_V$-admissible sets of order
$k$.
\end{prop}

Notice that the weight-formulas~\eqref{Eadm} in (the proof of)
Theorem~\ref{T1} are the $k = 1, \infty$ special cases of
Proposition~\ref{Padm}, in light of Example~\ref{Exadm}.

\begin{proof}
Denote by $\Psi(k)$ the right-hand side of~\eqref{Eadm2}, which the
result claims is independent of $k$:
\begin{equation}
\Psi(k) := \bigcup_{\mathcal{H}^{(k)}} \wt \mathbb{M}(\lambda,
\mathcal{H}^{(k)}), \qquad 1 \leq k \leq \infty.
\end{equation}

We now claim the inclusions
$\Psi(1) \subset \Psi(2) \subset \cdots \subset \Psi(\infty)$.
The result then follows from~\eqref{Eadm} (shown while proving
Theorem~\ref{T1}), which says: $\Psi(\infty) = \wt V = \Psi(1)$. To be
precise, \eqref{Eadm} does not explicitly say $\wt V = \Psi(1)$, but one
then notes that defining $\supp \mathcal{H}_V := \bigcup_{H \in
\mathcal{H}_V} H$, the set of $J$ in the union in~\eqref{Eadm} may be
reduced to only the $J \subset \supp \mathcal{H}_V$ -- equivalently,
every $J$ may be replaced by $J \cap \supp \mathcal{H}_V$, since $\wt
M(\lambda, J) \subset \wt M(\lambda, J \cap \supp \mathcal{H}_V)$. From
this one checks: $\wt V = \Psi(1)$.

To show the claim, fix $k' < k''$ in $[1,\infty]$, and list all order
$k'$ and order $k''$ $\mathcal{H}_V$-admissible sets as
\[
\mathcal{H}^{(k')}_1, \dots, \mathcal{H}^{(k')}_{m'};
\qquad \text{respectively,} \qquad
\mathcal{H}^{(k'')}_1, \dots, \mathcal{H}^{(k'')}_{m''}.
\]

Also list $\mathcal{H}^{(\infty)} = \mathcal{H}_V = \{ H_1, \dots, H_N
\}$. Now given $i' \in [1,m']$, use Definition~\ref{Dadm} to write the
elements of $\mathcal{H}^{(k')}_{i'}$ as a multiset,
$\mathcal{H}^{(k')}_{i'} \longleftrightarrow H'_1, \dots, H'_N$ with
$H'_t \subset H_t\ \forall t \in [1,N]$. Now arbitrarily choose $H''_t$
of size $\min(k'', |H_t|)$ such that $H'_t \subset H''_t \subset H_t\
\forall t \in [1,N]$. Then the distinct sets among the $H''_t$ comprise
$\mathcal{H}^{(k'')}_{\varphi(i')}$ for some function $\varphi : [1,m']
\to [1,m'']$. This implies
\[
\mathbb{M}(\lambda, \mathcal{H}^{(k'')}_{\varphi(i')}) \twoheadrightarrow
\mathbb{M}(\lambda, \mathcal{H}^{(k')}_{i'}) \qquad \implies \qquad
\wt \mathbb{M}(\lambda, \mathcal{H}^{(k')}_{i'}) \subset \wt
\mathbb{M}(\lambda, \mathcal{H}^{(k'')}_{\varphi(i')}).
\]
As this holds for all $1 \leq k' < k'' \leq \infty$ and $i' \in [1,m']$,
\[
\Psi(k') = \bigcup_{i'=1}^{m'} \wt \mathbb{M}(\lambda,
\mathcal{H}^{(k')}_{i'}) \subset \bigcup_{i'=1}^{m'} \wt
\mathbb{M}(\lambda, \mathcal{H}^{(k'')}_{\varphi(i')}) \subset
\bigcup_{i''=1}^{m''} \wt \mathbb{M}(\lambda, \mathcal{H}^{(k'')}_{i''})
= \Psi(k'').
\]
This shows the claim, and completes the proof.
\end{proof}

A consequence of Proposition~\ref{Padm} is a family of weight-formulas
that generalizes Theorem~\ref{T2}. We begin with the extreme cases, which
were shown in~\eqref{EaltB}:
\begin{equation}\label{EaltBB}
\wt V = \wt L_{J_\lambda}^{\max}(\lambda) + \wt \mathbb{M}(0,
\mathcal{H}_V) = \wt L_{J_\lambda}^{\max}(\lambda) + \bigcup_{J \subset
J_\lambda\, :\, J \cap H \neq \emptyset\; \forall H \in \mathcal{H}_V}
\wt M(0,J).
\end{equation}

These turn out to be the $k=\infty$ and $k=1$ cases of the following
result:

\begin{cor}
Given a $\mathfrak{g}$-module $M(\lambda) \twoheadrightarrow V$ (for
arbitrary $\mathfrak{g}, \lambda$), and an integer $1 \leq k \leq
\infty$,
\begin{equation}
\wt V = \wt L_{J_\lambda}^{\max}(\lambda) + \bigcup_{\mathcal{H}^{(k)}}
\wt \mathbb{M}(0, \mathcal{H}^{(k)}),
\end{equation}
where the union runs over all $\mathcal{H}_V$-admissible sets of order
$k$.
\end{cor}

This follows from~\eqref{EaltBB} and Proposition~\ref{Padm},
using $(\lambda, V) \leadsto (0, \mathbb{M}(0, \mathcal{H}_V))$ (via
e.g.\ Lemma~\ref{Luhwm}).
%}}}

%{{{1 Section 6 - Theorem D: Higher order parabolic category O, enough
%projectives, and BGG reciprocity
\section{Theorem~\ref{T4}: Higher order parabolic category $\mathcal{O}$,
enough projectives, and BGG reciprocity}\label{SO}

We next discuss refinements of the BGG Category $\mathcal{O}$ \cite{BGG}.
In keeping with the perspective of higher order approximations and
integrability, we begin by noting how the usual Category $\mathcal{O}$
and the parabolic Category $\mathcal{O}^{\mathfrak{p}_J}$ are zeroth and
first order special cases of the categories
\[
\mathcal{O}^{\mathcal{H}}, \qquad \mathcal{H} \subset {\rm Indep}(I)
\]
which we introduced in Definition~\ref{DOH}: the objects in $\mathcal{O}$
on which all ${\bf f}_H$, $H \in \mathcal{H}$ act locally nilpotently.
As mentioned there, \textbf{in this section
$\mathfrak{g}$ is assumed to be of finite type}.

Indeed, Definition~\ref{DOH} of $\mathcal{O}^{\mathcal{H}}$ specializes
to the usual and parabolic categories as follows:
\begin{enumerate}
\item If $\mathcal{H}$ is empty, then $\mathcal{O}^\emptyset$ is just
Category $\mathcal{O}$, and contains all Verma modules
$\mathbb{M}(\lambda, \emptyset) = M(\lambda)$.

\item More generally, if $\mathcal{H}_J = \{ \{ j \} : j \in J \}$ for $J
\subset I$, then (see e.g.\ \cite[Section 9.3]{H2} and use that
$\mathfrak{g}$ is semisimple, to show that) $\mathcal{O}^{\mathcal{H}_J}$
is precisely the parabolic Category $\mathcal{O}^{\mathfrak{p}_J}$, and
it contains the parabolic Verma modules $M(\lambda, J) =
\mathbb{M}(\lambda, \mathcal{H}_J)$ for all $J$-dominant integral weights
$\lambda$.
\end{enumerate}

To the best of our knowledge, these subcategories of $\mathcal{O}$ have
not been studied beyond the parabolic categories
$\mathcal{O}^{\mathfrak{p}_J}$. Yet, they naturally generalize
$\mathcal{O}^{\mathfrak{p}_J}$, and are intimately linked with
higher order Verma modules -- having higher order holes/integrability.
This section initiates their study.

\subsection{Enough projectives}

We show the first half of Theorem~\ref{T4} -- $\mathcal{O}^{\mathcal{H}}$
has enough projectives -- following several intermediate results.
In this section and the next, Remark~\ref{Ruhwm}(2) is useful: we will
often work not with $\mathcal{H}$, but instead with
$\mathcal{H}^{\min}$,
the collection of minimal holes in $\mathcal{H}$, which form a pairwise
incomparable collection. Thus, e.g.\ the final assertion in
Theorem~\ref{T2} says that the finite collection of ``sets of
incomparable subsets of ${\rm Indep}(J_\lambda)$'' is in bijection with
$\{ \wt V : M(\lambda) \twoheadrightarrow V \}$.

Coming to properties of $\mathcal{O}^{\mathcal{H}}$, a first ``sanity
check'' is that Definition~\ref{DOH} of $\mathcal{O}^{\mathcal{H}}$ fits
well with minimal elements and upper-closures, just like the modules
$\mathbb{M}(\lambda, \mathcal{H})$ do (see Remark~\ref{Ruhwm}):

\begin{lemma}\label{Lsanity}
If $\mathcal{H} \subset {\rm Indep}(I)$,
$\mathcal{O}^{\mathcal{H}} = \mathcal{O}^{\mathcal{H}^{\min}} =
\mathcal{O}^{\overline{\mathcal{H}} }$,
where $\overline{\mathcal{H}}$ is the upper-closure of $\mathcal{H}$
in ${\rm Indep}(I)$.
\end{lemma}

\begin{proof}
Here is a proof for completeness. From the definitions,
$\mathcal{O}^{\overline{\mathcal{H}} } \subset \mathcal{O}^{\mathcal{H}}
\subset \mathcal{O}^{\mathcal{H}^{\min}}$.
Now it suffices to show the reverse inclusion to the second one, since
$({\overline{\mathcal{H}} })^{\min} = \mathcal{H}^{\min}$. Let $M \in
\mathcal{O}^{\mathcal{H}^{\min}}$, $H \in \mathcal{H}$, and let $v$ be a
weight vector in $M$. It suffices to show that ${\bf f}_H$ acts
nilpotently on $v$. Choose a minimal hole $H_0 \subset H$ with $H_0 \in
\mathcal{H}^{\min}$; now ${\bf f}_{H_0}^n v = 0$ for some $n$. But then
${\bf f}_H^n v = 0$.
\end{proof}

Next, these generalizations of the parabolic category
$\mathcal{O}^{\mathfrak{p}_J}$ share the same basic properties as it:

\begin{lemma}\label{Lbasic}
Fix a subset $\mathcal{H} \subset {\rm Indep}(I)$.
Then $\mathcal{O}^{\mathcal{H}}$ is an abelian subcategory of
$\mathcal{O} \subset \mathfrak{g}$-${\rm Mod}$ that is closed under:
taking submodules, quotients, finite direct sums, extensions in
$\mathcal{O}$, restricted duals, and tensor products with finite
dimensional $\mathfrak{g}$-modules.
In particular, if $M \in \mathcal{O}^\mathcal{H}$ splits in $\mathcal{O}$
according to the action of the center $Z(U \mathfrak{g})$ into $M =
\bigoplus_\chi M^\chi$, then each $M^\chi$ also lies in
$\mathcal{O}^\mathcal{H}$.
\end{lemma}

\begin{proof}
We only outline the proof for restricted duals $M \mapsto M^\vee$ and
tensoring $M \mapsto M \otimes L(\lambda)$ for $\lambda$ dominant
integral, assuming the closure of $\mathcal{O}^{\mathcal{H}}$ under the
other operations is shown.

First, duals. Since $\mathfrak{g}$ is of finite type, every module $M \in
\mathcal{O}^{\mathcal{H}}$ has finite length, and via the other
operations listed, every simple subquotient is in
$\mathcal{O}^{\mathcal{H}}$. Construct $M^\vee \in \mathcal{O}$ by
dualizing a Jordan--H\"older series for $M$; every simple factor is
$L(\mu)^\vee \cong L(\mu)$. Now if ${\bf f}_H$ acts locally
nilpotently on each such factor, then it does so on extensions in
$\mathcal{O}$ between them, and hence on $M^\vee$ as desired.

Next, tensoring. Given weights $\mu, \nu \in \mathfrak{h}^*$ and nonzero
weight vectors $m_\mu \in M_\mu, v_\nu \in L(\lambda)_\nu$, it suffices
to show that ${\bf f}_H$ acts nilpotently on $m_\mu \otimes v_\nu$ for
every hole $H \in \mathcal{H}$. By definition, there exists $K > 0$
such that ${\bf f}_H^K m_\mu = 0$; similarly, there exists $N > 0$ such
that $f_h^N v_\nu = 0\ \forall h \in H$. Now,
\[
{\bf f}_H^n (m_\mu \otimes v_\nu) = \sum_{k_h \in [0,n]\, \forall h \in
H} \ \prod_{h \in H} \binom{n}{k_h} \cdot \left( {\textstyle \prod_{h \in
H} f_h^{k_h} \cdot m_\mu \otimes \prod_{h \in H} f_h^{n-k_h} \cdot v_\nu}
\right), \qquad \forall n \geq 0.
\]
Hence, every term on the right vanishes if $n \geq K+N-1$.
\end{proof}

We next identify the highest weight modules, simple objects and their
universal covers in $\mathcal{O}^{\mathcal{H}}$.

\begin{defn}
Given $\lambda \in \mathfrak{h}^*$ and $\mathcal{H} \subset {\rm
Indep}(I)$, extend Definition~\ref{DMlambdaH} and define
$\mathcal{H}'_\lambda$ as follows:
\begin{equation}\label{EHprimelambda}
\mathbb{M}(\lambda, \mathcal{H}) := \frac{M(\lambda)}{\displaystyle
\sum_{H \in \mathcal{H}^{\min}} U \mathfrak{g} \left( \prod_{h \in
J_\lambda \cap H} f_h^{\langle \lambda, \alpha_h^\vee \rangle + 1}
\right) M(\lambda)_\lambda}, \qquad
\mathcal{H}'_\lambda := \{ J_\lambda \cap H : H \in \mathcal{H}^{\min}
\}.
\end{equation}
Thus $\mathbb{M}(\lambda, \mathcal{H}) = \mathbb{M}(\lambda,
\mathcal{H}'_\lambda)$ for all $\lambda$ and $\mathcal{H}$; we use this
fact below without reference. In particular, if $J_\lambda$ does not
intersect some $H \in \mathcal{H}^{\min}$, then $\mathbb{M}(\lambda,
\mathcal{H}) = 0$.
\end{defn}

\begin{prop}\label{Psimple}
Fix a nonempty subset $\mathcal{H} \subset {\rm Indep}(I)$
and a weight $\lambda \in \mathfrak{h}^*$.
\begin{enumerate}
\item The module $L(\lambda) \in \mathcal{O}^{\mathcal{H}}$ if and only
if $J_\lambda \cap H \neq \emptyset\ \forall H \in \mathcal{H}^{\min}$,
if and only if $\mathbb{M}(\lambda, \mathcal{H}) \neq 0$.

\item In this case, the universal highest weight cover in
$\mathcal{O}^{\mathcal{H}}$ of $L(\lambda)$ is $\mathbb{M}(\lambda,
\mathcal{H}) = \mathbb{M}(\lambda, \mathcal{H}'_\lambda)$.

\item Suppose $L(\lambda) \in \mathcal{O}^{\mathcal{H}}$.
A highest weight module $M(\lambda) \twoheadrightarrow V$
belongs to $\mathcal{O}^{\mathcal{H}}$ if and only if
$\mathcal{H}'_\lambda \subset \mathcal{H}_V$,
if and only if $\mathbb{M}(\lambda, \mathcal{H}_V) \in
\mathcal{O}^{\mathcal{H}}$.
\end{enumerate}
\end{prop}

Note that the condition in part~(1) is reminiscent of the set
$\mathfrak{J}(V)$ used in proving Theorem~\ref{T1}. This is made precise
in Equation~\eqref{Ealtwts}.

As a special case, recall that the simples in
$\mathcal{O}^{\mathfrak{p}_J}$ are $L(\lambda)$ for $\lambda$
$J$-dominant integral; and their universal covers are the parabolic Verma
modules $M(\lambda,J)$. This follows from Proposition~\ref{Psimple}:
set
\[
\mathcal{H} = \mathcal{H}_J = \{ \{ j \} : j \in J \} =
\mathcal{H}^{\min},
\]
in which case $\mathcal{H}'_\lambda = \mathcal{H}_J$ as well, and so
$\mathbb{M}(\lambda, \mathcal{H}'_\lambda) = M(\lambda,J)$.

\begin{defn}\label{Dunivcover}
Given a subset $\mathcal{H} \subset {\rm Indep}(I)$, and a weight
$\lambda$ such that $L(\lambda) \in \mathcal{O}^{\mathcal{H}}$, define
its \textit{universal cover} in $\mathcal{O}^{\mathcal{H}}$ to be
$\mathbb{M}(\lambda, \mathcal{H}'_\lambda)$ (see
Proposition~\ref{Psimple}(2)). Also define \textit{standard objects} to
be all modules $\mathbb{M}(\lambda, \mathcal{H}_0) \in
\mathcal{O}^{\mathcal{H}}$.
\end{defn}

The proof of Proposition~\ref{Psimple} requires one last lemma, in
addition to the two above:

\begin{lemma}\label{L3}
Suppose $\lambda \in \mathfrak{h}^*$, $H \in {\rm Indep}(J_\lambda)$, and
$M(\lambda) \twoheadrightarrow V$ is a nonzero highest weight module. If
${\bf f}_H$ is nilpotent on the highest weight line $V_\lambda$, then $H
\in \mathcal{H}_V$.
\end{lemma}

This result and Proposition~\ref{Psimple} repeatedly use the following
fact, as useful in each $\mathcal{O}^{\mathcal{H}}$ as it was in
$\mathcal{O}$. Namely, if $v_\lambda$ is a maximal vector for a raising
operator $e_h$, with weight $\lambda$, then
\begin{equation}\label{Esl2}
e_h^n \cdot f_h^n v_\lambda \in \mathbb{C}^\times v_\lambda, \qquad
\text{whenever }
\langle \lambda, \alpha_h^\vee \rangle \not\in \Z_{\geq 0} \ni n \
\text{ or } \ 0 \leq n \leq \langle \lambda, \alpha_h^\vee \rangle \in
\Z.
\end{equation}

\begin{proof}
Let $n > 0$ denote the smallest power such that ${\bf f}_H^n \cdot
v_\lambda = 0$, where we fix a nonzero highest weight vector $v_\lambda
\in V_\lambda$. Also define $m_h := \langle \lambda, \alpha_h^\vee
\rangle + 1 \in \Z_{>0}$ for $h \in J_\lambda$. Now if $n < m_h\
\forall h \in H$, then applying $e_h^n$ for all $h$ to the equation ${\bf
f}_H^n v_\lambda = 0$ yields $v_\lambda = 0$, which is false. Thus $H_1
:= \{ h  \in H : n \geq m_h \}$ is nonempty. Applying $\prod_{h \in H_1}
e_h^{n - m_h} \prod_{h \in H \setminus H_1} e_h^n$ to ${\bf f}_H^n
V_\lambda = 0$ -- via~\eqref{Esl2} -- yields $\prod_{h \in H_1} f_h^{m_h}
\cdot V_\lambda = 0$. Thus $H_1$ lies in the upper-closed set
$\mathcal{H}_V$, hence so does $H$.
\end{proof}

\begin{proof}[Proof of Proposition~\ref{Psimple}]\hfill
\begin{enumerate}
\item First suppose $J_\lambda$ intersects every hole in
$\mathcal{H}^{\min}$. Then ${\bf f}_H = {\bf f}_{J_\lambda \cap H} {\bf
f}_{H \setminus J_\lambda}\ \forall H \in \mathcal{H}$, where the two
factors on the right commute and ${\bf f}_{J_\lambda \cap H}$ acts
nilpotently on the highest weight line $L(\lambda)_\lambda$. Hence it
acts nilpotently on other vectors in $L(\lambda)$ as well, using
arguments similar to the proof of Lemma~\ref{LSerre}. But then so does
${\bf f}_H$. Hence $L(\lambda) \in \mathcal{O}^{\mathcal{H}}$.

Conversely, suppose $L(\lambda) \in \mathcal{O}^{\mathcal{H}}$, and say
there exists $H \in \mathcal{H}^{\min}$ which is disjoint from
$J_\lambda$. If ${\bf f}_H^n \cdot L(\lambda)_\lambda = 0$, then applying
$\prod_{h \in H} e_h^n$ via~\eqref{Esl2} yields: $L(\lambda)_\lambda =
0$, a contradiction.

\item We claim the upper-closure of $\mathcal{H}'_\lambda$ -- defined in
\eqref{EHprimelambda} -- is the smallest upper-closed subset
$\mathcal{H}_0 \subset {\rm Indep}(I)$ such that $\mathbb{M}(\lambda,
\mathcal{H}_0) \in \mathcal{O}^{\mathcal{H}}$.
To see why, first fix a nonzero highest weight vector $v_\lambda \in
\mathbb{M}(\lambda, \mathcal{H}'_\lambda)_\lambda$. For $H \in
\mathcal{H}^{\min}$, write ${\bf f}_H = {\bf f}_{J_\lambda \cap H} {\bf
f}_{H \setminus J_\lambda}$ as above. Then $J_\lambda \cap H \in
\mathcal{H}'_\lambda$, so ${\bf f}_H$ acts nilpotently on $v_\lambda$,
hence acts locally nilpotently on $\mathbb{M}(\lambda,
\mathcal{H}'_\lambda)$ -- for all $H \in \mathcal{H}^{\min}$. Thus,
$\mathbb{M}(\lambda, \mathcal{H}'_\lambda) \in
\mathcal{O}^{\mathcal{H}^{\min}} = \mathcal{O}^{\mathcal{H}}$ (by
Lemma~\ref{Lsanity}).

Now suppose $\mathbb{M}(\lambda, \mathcal{H}_0) \in
\mathcal{O}^{\mathcal{H}}$, and assume henceforth that $\mathcal{H}_0$ is
upper-closed. We claim that $\mathcal{H}'_\lambda \subset \mathcal{H}_0$.
To see why, fix a nonzero highest weight vector $v_\lambda \in
\mathbb{M}(\lambda, \mathcal{H}_0)_\lambda$, and let $H \in
\mathcal{H}^{\min}$. Then ${\bf f}_H = {\bf f}_{J_\lambda \cap H} {\bf
f}_{H \setminus J_\lambda}$ acts nilpotently on $v_\lambda$.
Say its $n$th power annihilates $v_\lambda$. Applying $\prod_{h \in H
\setminus J_\lambda} e_h^n$ as above via~\eqref{Esl2}, ${\bf
f}_{J_\lambda \cap H}$ acts nilpotently on $v_\lambda$. Now applying
Lemma~\ref{L3} with $J_\lambda \cap H$ in place of $H$ shows that
$J_\lambda \cap H \in \mathcal{H}_0$. As this holds for all $H \in
\mathcal{H}^{\min}$, the desired conclusion follows:
$\mathcal{H}'_\lambda \subset \mathcal{H}_0$.

\item If $\mathcal{H}'_\lambda \subset \mathcal{H}_V$ then by the
definitions, $\mathbb{M}(\lambda, \mathcal{H}'_\lambda)
\twoheadrightarrow \mathbb{M}(\lambda, \mathcal{H}_V) \twoheadrightarrow
V$, and the first of these lies in $\mathcal{O}^{\mathcal{H}}$, which is
closed under quotienting. Conversely, say $V \in
\mathcal{O}^{\mathcal{H}}$, and $H \in \mathcal{H}^{\min}$ (so $J_\lambda
\cap H \neq \emptyset$ by part (1)). Now ${\bf f}_H$ is nilpotent on the
highest weight line $V_\lambda$, hence so is ${\bf f}_{J_\lambda \cap H}$
by~\eqref{Esl2}. By Lemma~\ref{L3}, $J_\lambda \cap H \in \mathcal{H}_V$
for all $H \in \mathcal{H}^{\min}$, which finishes the proof.
\qedhere
\end{enumerate}
\end{proof}

With these results now shown, we begin proving our next main theorem, on
$\mathcal{O}^{\mathcal{H}}$.

\begin{proof}[Proof of Theorem~\ref{T4}, first part]
The third assertion -- involving BGG reciprocity -- is shown in
Theorem~\ref{TBGGrec}. Here we prove the rest, starting by showing that
$\mathcal{O}^{\mathcal{H}}$ has enough projectives.
Recall that the BGG Category $\mathcal{O}$ decomposes
as a direct sum over twisted $W$-orbits: $\mathcal{O} = \bigoplus
\mathcal{O}^{W \bullet \lambda}$, using central characters and
Harish-Chandra's theorem. Hence so does the subcategory
$\mathcal{O}^{\mathcal{H}} \subset \mathcal{O}$, via Lemma~\ref{Lbasic}.
It suffices to work in one such intersection
\begin{equation}
\mathcal{A} :=
\mathcal{O}^{\mathcal{H}} \cap \mathcal{O}^{W \bullet \lambda},
\end{equation}
where we \textbf{fix} $\lambda \in \mathfrak{h}^*$ satisfying:
$L(\lambda) \in \mathcal{O}^{\mathcal{H}}$.
Indeed, if we show there exist enough projectives $P$ in each such
category, and run over all dot-orbits $W \bullet \lambda$, then all such
$P$ are in fact projectives in $\mathcal{O}^{\mathcal{H}}$.

Now one shows -- using Proposition~\ref{Psimple} -- that
$\mathcal{A}$ has enough projectives. This is via the sufficient
criterion in \cite[Theorem 3.2.1]{BGS}, wherein one verifies five
conditions (not six, by Ringel's subsequent remark in \cite{BGS}).
As the verification is mostly standard, it is deferred to
Appendix~\ref{Sappendix} -- we do include it because for some simples,
there are multiple standard objects that get used in BGG reciprocity --
and because the proof of BGG reciprocity also uses similar arguments, see
Section~\ref{SBGG2}.

This proves the first assertion; we now turn to the second. We claim the
following equalities for every highest weight module $M(\lambda)
\twoheadrightarrow V$:
\begin{equation}\label{Ealtwts}
\wt \mathbb{M}(\lambda, (\mathcal{H}_V)'_\lambda) = \wt V =
\bigcup_{K \subset J_\lambda\, :\, L(w_{J_\lambda \setminus K} \bullet
\lambda) \in \mathcal{O}^{\mathcal{H}_V}} \wt M(\lambda, K),
\end{equation}
where $w_J$ is the longest element of $W_J$ for any $J \subset I$.
Note that the second equality yields an alternate weight-formula to
Theorem~\ref{T1}.

We begin with the first equality, which proves the second assertion in
the theorem via Proposition~\ref{Psimple}(2). This equality follows
because as shown above, $\wt V = \wt \mathbb{M}(\lambda, \mathcal{H}_V)$;
now apply Proposition~\ref{Psimple}(3) using that $V \in
\mathcal{O}^{\mathcal{H}_V}$.

This completes the proof of Theorem~\ref{T4}(2). We conclude with the
proof of the second equality in~\eqref{Ealtwts}; this follows from the
claim that
\[
\mathfrak{J}(V) = \{ K \subset J_\lambda \ | \ L(w_{J_\lambda \setminus
K} \bullet \lambda) \in \mathcal{O}^{\mathcal{H}_V} \},
\]
where $\mathfrak{J}(V)$ is as in~\eqref{EJV}. To see the claim, first
note that $H \subset J_\lambda$ if $H \in \mathcal{H}_V^{\min}$. By this
and Proposition~\ref{Psimple}(1), $L(w_{J_\lambda \setminus K} \bullet
\lambda) \in \mathcal{O}^{\mathcal{H}_V}$ if and only if $J_{w_{J_\lambda
\setminus K} \bullet \lambda} \cap (J_\lambda \cap H) \neq \emptyset$ for
all $H \in \mathcal{H}_V^{\min}$. Now by~\eqref{EJV}, it suffices to show
for $K \subset J_\lambda$ that $J_{w_{J_\lambda \setminus K} \bullet
\lambda} \cap J_\lambda = K$. One inclusion is because $\lambda$ is
$J_\lambda \setminus K$-dominant integral, so $-w_{J_\lambda \setminus K}
\bullet \lambda$ is strictly dominant integral for $J_\lambda \setminus
K$. The reverse inclusion follows from writing $w_{J_\lambda \setminus K}
\bullet \lambda = \lambda - \sum_{j \in J_\lambda \setminus K} l'_j
\alpha_j$ for some $l'_j \in \Z_{\geq 0}$, and evaluating against
$\langle -, \alpha_k^\vee \rangle$ for $k \in K$.
\end{proof}

We conclude with a well known consequence of standard facts on finite
length abelian categories $\mathcal{A}$ with finitely many simple objects
and enough projectives.

\begin{cor}
Every simple object $L(\lambda) \in \mathcal{O}^{\mathcal{H}}$ has a
projective cover $P^{\mathcal{H}}(\lambda) \in
\mathcal{O}^{\mathcal{H}}$, and so
\[
P^{\mathcal{H}}(\lambda) \twoheadrightarrow \mathbb{M}(\lambda,
\mathcal{H}'_\lambda) \twoheadrightarrow L(\lambda).
\]
Moreover, for all objects $M \in \mathcal{O}^{\mathcal{H}}$, one has
$\dim {\rm Hom}_{\mathcal{O}^{\mathcal{H}} } (P^{\mathcal{H}}(\lambda),
M) = [M : L(\lambda)]$, the number of Jordan--H\"older factors of $M$
isomorphic to $L(\lambda)$.
\end{cor}

\subsection{Properties of standard filtrations}

Having proved that the categories $\mathcal{O}^{\mathcal{H}}$ all have
enough projectives, it is natural to ask if these projectives have
``standard filtrations''; and if they do, then does BGG reciprocity hold
in some form. We begin by defining the former notion.

\begin{defn}
Fix a subset $\mathcal{H} \subset {\rm Indep}(I)$. An object $M \in
\mathcal{O}^{\mathcal{H}}$ is said to have a \textit{standard filtration
in $\mathcal{O}^{\mathcal{H}}$} if there exists a subcategory
$\mathcal{O}^{\mathcal{H}'}$ of $\mathcal{O}^{\mathcal{H}}$ that contains
$M$, and a finite filtration
\[
0 = M_0 \subset M_1 \subset \cdots \subset M_k = M,
\]
with each subquotient $M_i / M_{i-1} \cong \mathbb{M}(\lambda_i,
\mathcal{H}') = \mathbb{M}(\lambda_i, \mathcal{H}'_{\lambda_i})$ for some
$\lambda_i \in \mathfrak{h}^*$ (by Proposition~\ref{Psimple}(2)).
\end{defn}

\begin{remark}
It is also possible to define a weaker notion: an object $M \in
\mathcal{O}^{\mathcal{H}}$ has a \textit{weakly standard filtration in
$\mathcal{O}^{\mathcal{H}}$} if there exists a finite filtration
$0 = M_0 \subset M_1 \subset \cdots \subset M_k = M$,
with each subquotient $M_i / M_{i-1} \cong \mathbb{M}(\lambda_i,
\mathcal{H}_i)$ for some $\lambda_i \in \mathfrak{h}^*$, $\mathcal{H}_i
\subset {\rm Indep}(J_{\lambda_i})$. However, we work with the above,
stronger notion -- which we prove holds for all projectives in
$\mathcal{O}^{\mathcal{H}}$, over $\mathfrak{g} = \mathfrak{sl}_2^{\oplus
n}$ below.
\end{remark}

The result in this part that will be useful in showing BGG reciprocity,
is the natural one:

\begin{prop}\label{Pdirect}
Suppose $\mathfrak{g}$ is semisimple and $\mathcal{H} \subset {\rm
Indep}(I)$. Given objects $M', M'' \in \mathcal{O}^{\mathcal{H}}$, their
direct sum $M' \oplus M''$ has a standard filtration in
$\mathcal{O}^{\mathcal{H}}$, if and only if both $M'$ and $M''$ do.
\end{prop}

That said, this result, and indeed the treatment of category
$\mathcal{O}^{\mathcal{H}}$ for general $\mathcal{H}$, differ from its
``zeroth order'' and ``first order'' (parabolic) special cases in the
literature, in that now one is no longer working with $U
\mathfrak{n}'$-free modules (for a nonzero Lie subalgebra $\mathfrak{n}'
\subset \mathfrak{n}^-$) if $\mathcal{H}$ is more general. In particular,
``standard objects'' $\mathbb{M}(\lambda,\mathcal{H})$ are not always
obtained by induction from $U\mathfrak{n}'$ to $U \mathfrak{n}^-$. E.g.\
over $\mathfrak{g} = \mathfrak{sl}_2^{\oplus 2}$ (see~\eqref{EM00}), the
universal cover
\[
V_{00} = M(0,0) / M(-2,-2) = \mathbb{M}((0,0), \{ \{ 1, 2 \} \}) \cong
\mathbb{C}[ f_1, f_2 ] / (f_1 f_2).
\]
This is due to non-singleton sets in $\mathcal{H}$, and it makes the
proofs in this section diverge from the literature -- including for the
next lemma.

\begin{lemma}\label{Lext}
Suppose $\mathcal{H} \subset {\rm Indep}(I)$ is such that $0 \to N \to M
\to \mathbb{M}(\lambda, \mathcal{H}) \to 0$ is a short exact sequence in
the category $\mathcal{O}^{\mathcal{H}}$. If $\lambda$ is maximal in $\wt
M$, then the extension $M$ splits.
\end{lemma}

\begin{proof}
Note that $M_\lambda \neq 0$, so $L(\lambda) \in
\mathcal{O}^\mathcal{H}$, so $\mathbb{M}(\lambda, \mathcal{H}) =
\mathbb{M}(\lambda, \mathcal{H}'_\lambda)  \neq 0$ by
Proposition~\ref{Psimple}.
Pick a weight vector $0 \neq v_\lambda \in \mathbb{M}(\lambda,
\mathcal{H})_\lambda$ and its preimage $m_\lambda \in M_\lambda$. Then $V
:= U \mathfrak{g} \cdot m_\lambda \in \mathcal{O}^{\mathcal{H}}$, and
$M(\lambda) \twoheadrightarrow V$ because $\mathfrak{n}^+ m_\lambda = 0$
by assumption.
We claim that $M = V \oplus N$. Indeed, by Proposition~\ref{Psimple}(3),
$\mathcal{H}'_\lambda \subset \mathcal{H}_V$, which gives a sequence of
surjections whose composite is an isomorphism:
\[
\mathbb{M}(\lambda, \mathcal{H}'_\lambda) \twoheadrightarrow
\mathbb{M}(\lambda, \mathcal{H}_V) \twoheadrightarrow V
\twoheadrightarrow \mathbb{M}(\lambda, \mathcal{H}) =
\mathbb{M}(\lambda, \mathcal{H}'_\lambda).
\]
Thus the final map is an isomorphism, so $V \cap N = 0$, yielding the
desired splitting.
\end{proof}

\begin{lemma}\label{Lfiltration}
Suppose $\mathcal{H} \subset {\rm Indep}(I)$, and $M \in
\mathcal{O}^{\mathcal{H}}$ has a standard filtration in
$\mathcal{O}^{\mathcal{H}}$. If $\lambda$ is maximal in $\wt M$, then
there exists a submodule $M'$ of $M$ satisfying:
(i)~$M' \cong \mathbb{M}(\lambda, \mathcal{H})$, and
(ii)~$M / M'$ has a standard filtration in $\mathcal{O}^{\mathcal{H}}$.
\end{lemma}

Lemma~\ref{Lfiltration} is proved using Lemma~\ref{Lext} and similar
arguments to the classical case of $\mathcal{O}$. In turn, it implies
Proposition~\ref{Pdirect}. The proofs are similar to e.g.\ those in
\cite[Section 3.7]{H2}.

\subsection{BGG reciprocity -- subtlety in the higher order
case, over all $\mathfrak{g}$ of rank $\geq 3$}\label{SBGG}

With the above machinery and results at hand, we turn to the remainder of
Theorem~\ref{T4} -- i.e., BGG reciprocity~\eqref{EBGGrec} in all
categories $\mathcal{O}^{\mathcal{H}}$ over $\mathfrak{g} =
\mathfrak{sl}_2^{\oplus n}$. That BGG reciprocity holds in the
zeroth/first order cases (i.e., the usual/parabolic categories
$\mathcal{O}$) is well known, see \cite{BGG}, \cite{Rocha},
\cite[Chapters 3, 9]{H2}.

Before working over $\mathfrak{sl}_2^{\oplus n}$, we first show that the
situation in $\mathcal{O}^{\mathcal{H}}$ has a subtlety when
$\mathcal{H}$ has higher order holes -- over \textit{any} $\mathfrak{g}$
of rank at least $3$. This is because \textit{multiple} ``standard
objects'' $V = \mathbb{M}(\lambda, \mathcal{H}_0)$ exist over a given
$L(\lambda) \in \mathcal{O}^{\mathcal{H}}$ (for certain $\mathcal{H}$) --
recall, these were classified in Proposition~\ref{Psimple}(3). It turns
out that standard filtrations for different projectives in a block can
feature more than one such standard object $\mathbb{M}(\lambda,
\mathcal{H}_0)$, but only one of these is the universal cover
$\mathbb{M}(\lambda, \mathcal{H}'_\lambda)$ (see
Proposition~\ref{Psimple} and Definition~\ref{Dunivcover}).
This is already a break from the parabolic case \cite[Theorem
6.1]{Rocha}. We begin by illustrating this in an even simpler case -- in
rank two. The key object is again $V_{00} = M(0,0) / M(-2,-2)$
from~\eqref{EM00}, and its generalization $\mathbb{M}(\lambda, \{ \{ 1, 2
\} \})$.

\begin{example}[$\mathfrak{g} = \mathfrak{sl}_2 \oplus
\mathfrak{sl}_2$]\label{Erank2}
We present the complete picture over this algebra $\mathfrak{g}$, to
provide familiarity before tackling the case of $\mathfrak{sl}_2^{\oplus
n}$ for general $n$. By Lemma~\ref{Lsanity}, one needs to consider the
subcategories $\mathcal{O}^{\mathcal{H}}$, with $\mathcal{H}$ from among
the following five upper-closed subsets of ${\rm Indep}(I) = 2^I$:
\[
\mathcal{H} \quad = \quad \emptyset, \quad
\{ \{ 1 \}, \{ 1, 2 \} \}, \quad
\{ \{ 2 \}, \{ 1, 2 \} \}, \quad
\{ \{ 1 \}, \{ 2 \}, \{ 1, 2 \} \}, \quad
\{ \{ 1, 2 \} \}.
\]
The first case is that of the usual category $\mathcal{O}$, and the next
three cases are of its parabolic subcategories. These were addressed in
\cite{BGG} and \cite{Rocha}, respectively.

Thus, henceforth fix $\mathcal{H} = \{ \{ 1, 2 \} \} =
\mathcal{H}^{\min}$. If $J_\lambda = \{ 1 \}$ and $\langle \lambda,
\alpha_2^\vee \rangle + 1$ is either zero or a non-integer, then the
linkage class is $\{ \lambda > s_1 \bullet \lambda \}$, and $L(s_1
\bullet \lambda) \not\in \mathcal{O}^{\mathcal{H}}$, so the block of
$\mathcal{O}^{\mathcal{H}} \cap \mathcal{O}^{W \bullet \lambda}$
containing $L(\lambda)$ has only one simple object -- which is also
parabolic Verma and projective in that block. (The analogous story for
$J_\lambda = \{ 2 \}$ and $\langle \lambda, \alpha_1^\vee \rangle + 1$ as
above, also holds.)

The only remaining case is when $\mathcal{H} = \{ \{ 1, 2 \} \} =
\mathcal{H}^{\min}$ and $\lambda$ lies in a block with a dominant
integral element -- which we can set to be $\lambda$. Then the block
$\mathcal{O}^{\mathcal{H}} \cap \mathcal{O}^{W \bullet \lambda}$ has
three simples: $L(\lambda)$, $L(s_1 \bullet \lambda)$, $L(s_2 \bullet
\lambda)$. Their universal covers in $\mathcal{O}^{\mathcal{H}}$ are,
respectively:
\[
\mathbb{M}(\lambda, \{ \{ 1, 2 \} \}) = \frac{M(\lambda)}{M(s_1 s_2
\bullet \lambda)} \twoheadrightarrow L(\lambda), \quad \mathbb{M}(s_1
\bullet \lambda, \{ \{ 2 \} \}) = L(s_1 \bullet \lambda), \quad
\mathbb{M}(s_2 \bullet \lambda, \{ \{ 1 \} \}) = L(s_2 \bullet \lambda).
\]
In particular, there is a unique standard object of highest weight $s_i
\bullet \lambda$ for $i=1,2$. However, there are four standard objects in
$\mathcal{O}^{\mathcal{H}} \cap \mathcal{O}^{W \bullet \lambda}$ of
highest weight $\lambda$: $\mathbb{M}(\lambda, \{ J \})$ for $\emptyset
\neq J \subset \{ 1, 2 \}$, and $L(\lambda) \cong \mathbb{M}(\lambda, \{
\{ 1 \}, \{ 2 \} \})$. Moreover, the projective cover of the ``highest''
simple $L(\lambda)$ is its ``Verma cover'' $\mathbb{M}(\lambda, \{ \{ 1,
2 \} \})$ -- which has length $3$ -- while those of the other two simple
modules turn out to be their projective covers in smaller categories --
in fact, in parabolic categories $\mathcal{O}^{\mathfrak{p}_J} =
\mathcal{O}^{\mathcal{H}_J}$:
\begin{align}\label{Eses}
\begin{aligned}
&\ 0 \to M(\lambda, \{ 1 \}) \to P^{\mathcal{H}_{\{ 1 \}} }(s_2 \bullet
\lambda) \to M(s_2 \bullet \lambda, \{ 1 \}) \to 0\\
&\ 0 \to M(\lambda, \{ 2 \}) \to P^{\mathcal{H}_{\{ 2 \}} }(s_1 \bullet
\lambda) \to M(s_1 \bullet \lambda, \{ 2 \}) \to 0
\end{aligned}
\end{align}
(this follows from Theorem~\ref{TBGGrec} below).
Now one can easily verify ``BGG-type'' reciprocity in
$\mathcal{O}^{\mathcal{H}}$. Moreover, the ``Cartan matrix'' for
this block with respect to the indices $[\lambda]_\mathcal{H} :=
(\lambda, s_1 \bullet \lambda, s_2 \bullet \lambda)$ is
\begin{equation}\label{E3x3}
C := ([P^\mathcal{H}(\mu) : L(\mu')])_{\mu, \mu' \in
[\lambda]_\mathcal{H}} = \begin{pmatrix} 1 & 1 & 1 \\ 1 & 2 & 0 \\ 1 & 0
& 2\end{pmatrix}.
\end{equation}
It is easy to check (via inner products) that $C \neq D^T D$ for any
``multiplicity'' matrix $D$ with entries in $\mathbb{Z}_{\geq 0}$. This
shows ``numerically'' that BGG reciprocity does not hold in
$\mathcal{O}^\mathcal{H} \cap \mathcal{O}^{W \bullet \lambda}$ -- i.e.,
there cannot exist an intermediate class of standard objects for which
BGG reciprocity holds on the nose.
\end{example}

\begin{remark}\label{RBGGrec}
Thus, BGG reciprocity is more subtle even in rank $2$, when $\mathcal{H}
= \{ \{ 1, 2 \} \}$. Specifically, using multiple standard objects for
some of the weights in the block is now necessary, as can already be seen
here via our original motivating example~\eqref{EM00}. Namely, in Example
\ref{Erank2}, the standard filtrations for the projectives involve three
standard objects with highest weight $\lambda$: $\mathbb{M}(\lambda, \{ J
\})$ for $\emptyset \neq J \subset \{ 1, 2 \}$ (even for $\lambda =
(0,0)$).
\end{remark}

The preceding remark suggests $\mathcal{O}^{\mathcal{H}}$ is not a
highest weight category for $\mathcal{H} = \{ \{ 1, 2 \} \}$. To see why,
we discuss if any of the three standard objects $\mathbb{M}(\lambda, \{ J
\})$ for $\emptyset \neq J \subset \{ 1, 2 \}$ can be ``avoided'', via
alternate standard filtrations for the projective objects in
$\mathcal{O}^{\{ \{ 1, 2 \} \}}$. However, this is not possible:
\begin{enumerate}
\item First, $\mathbb{M}(\lambda, \mathcal{H}'_\lambda) =
\mathbb{M}(\lambda, \{ \{ 1, 2 \} \})$ cannot occur in any filtration of
$P^{\mathcal{H}}(s_i \bullet \lambda)$ in~\eqref{Eses} -- because even
the simpler statement $\wt \mathbb{M}(\lambda, \mathcal{H}'_\lambda)
\subset \wt P^{\mathcal{H}}(s_i \bullet \lambda)$ is false.

\item One can ask if the above notion of standard filtration could be
broadened to require the top quotient to merely be a standard object
$\mathbb{M}(\nu, \mathcal{H}')$ rather than the universal cover
$\mathbb{M}(\nu, \mathcal{H}'_\nu)$. By Proposition~\ref{Psimple}(3),
such a set $\mathcal{H}'$ could be an upper-closed subset containing
$\mathcal{H}'_\nu$. This weakening could enable using
$\mathbb{M}(\lambda, \mathcal{H}'')$ for $\mathcal{H}'' \neq
\mathcal{H}'_\lambda$, in the standard filtration for
$P^\mathcal{H}(\lambda)$.

Unfortunately, this hope is also in vain, in that even in the above
example with $\mathcal{H} = \{ \{ 1, 2 \} \}$, it leads to violating the
requirement that the remaining standard factors in the filtration of
$P^{\mathcal{H}}(\nu)$ should have highest weights $\mu > \nu$. Indeed,
set $\mu = \nu = \lambda$; now there are four upper-closed subsets
$\mathcal{H}'$ containing $\mathcal{H}'_\lambda = \{ \{ 1, 2 \} \}$:
\[
\mathcal{H}' \quad = \quad \{ \{ 1, 2 \} \}, \quad
\{ \{ 1 \}, \{ 1, 2 \} \}, \quad
\{ \{ 2 \}, \{ 1, 2 \} \}, \quad
\{ \{ 1 \}, \{ 2 \}, \{ 1, 2 \} \}.
\]
So if any other module $\mathbb{M}(\lambda, \mathcal{H}'')$ is
used in the standard filtration for $P^{\mathcal{H}}(\lambda) =
\mathbb{M}(\lambda, \mathcal{H}'_\lambda)$, then as $\lambda$ is maximal
in its dot-orbit, the kernel of $P^{\mathcal{H}}(\lambda)
\twoheadrightarrow \mathbb{M}(\lambda, \mathcal{H}'')$ has all factors
with highest weight $\mu < \lambda$ -- but highest weight categories and
BGG reciprocity require $\mu > \lambda$.
\end{enumerate}

\begin{remark}
The above discussion shows that $\mathcal{O}^{\mathcal{H}}$ is not a
highest weight category for general $\mathcal{H}$ in the sense of
Cline--Parshall--Scott \cite{CPS} -- as early as $\mathfrak{g} =
\mathfrak{sl}_2^{\oplus 2}$ and $\mathcal{H} = \{ \{ 1, 2 \} \}$. In this
sense, the category $\mathcal{O}^{\mathcal{H}}$ diverges in higher order,
from the zeroth and first order (parabolic) category $\mathcal{O}$.
\end{remark}

We end this part by showing the same fact over every $\mathfrak{g}$ of
higher rank, as promised above.

\begin{prop}
Suppose $\mathfrak{g}$ is semisimple of rank at least $3$. Then there
exist $\lambda \in \mathfrak{h}^*$ and an upper-closed set $\mathcal{H}
\subset {\rm Indep}(I)$ such that $\mathcal{O}^\mathcal{H} \cap
\mathcal{O}^{W \bullet \lambda}$ is not a highest weight category.
\end{prop}

\begin{proof}
Since the Dynkin diagram contains at least two leaves (in particular, it
is not complete), choose two ``orthogonal'' simple roots and label them
by $\alpha_1, \alpha_2$. Let $J := I \setminus \{ 1, 2 \}$ and fix a
generic weight $\lambda \in {\rm span}_{\mathbb{C}} \{ \varpi_j : j \in J
\}$ such that $W \bullet \lambda \cap (\lambda + \Z \Pi) = W_{\{ 1, 2 \}}
\bullet \lambda$; here $\varpi_j$ denotes the fundamental weight
corresponding to $j \neq 1, 2$. Now the integrabilities are computed as:
\[
J_\lambda = \{ 1, 2 \}, \qquad
J_{s_1 \bullet \lambda} = \{ 2 \}, \qquad
J_{s_2 \bullet \lambda} = \{ 1 \}, \qquad
J_{s_1 s_2 \bullet \lambda} = \emptyset.
\]
Let $\mathcal{H} = \{ \{ 1, 2 \} \}$; then there are only three simple
objects in the block $\mathcal{O}^\mathcal{H} \cap \mathcal{O}^{W \bullet
\lambda}$, and this reduces to the above situation over $\mathfrak{sl}_2
\oplus \mathfrak{sl}_2$.
\end{proof}

\subsection{Proof of BGG reciprocity over $\mathfrak{sl}_2^{\oplus
n}$}\label{SBGG2}

The above remarks explain why one needs to refine the ``usual'' notion of
BGG reciprocity. We now do so over $\mathfrak{g} =
\mathfrak{sl}_2^{\oplus n}$ for all $n$, thereby completing the proof of
Theorem~\ref{T4}. 

\begin{theorem}\label{TBGGrec}
Suppose $\mathfrak{g} = \mathfrak{sl}_2^{\oplus n}$ for some $n \geq 1$,
and $\lambda \in \mathfrak{h}^*$. Define $w_K := \prod_{k \in K} s_k$ for
$K \subset I = \{ 1, \dots, n \}$ and also fix $\mathcal{H} \subset {\rm
Indep}(I) = 2^I$.
\begin{enumerate}
\item If $L(w_K \bullet \lambda) \in \mathcal{O}^{\mathcal{H}}$ for $K
\subset J_\lambda$, its projective cover $P^\mathcal{H}(w_K \bullet
\lambda)$ has a ``standard filtration'' by objects $\mathbb{M}(\mu,
\mathcal{H}'_{w_K \bullet \lambda})$, with topmost quotient the
``maximal'' standard object $\mathbb{M}(w_K \bullet \lambda,
\mathcal{H}'_{w_K \bullet \lambda})$ over $L(w_K \bullet \lambda)$, and
all other subquotients of highest weight $\mu \in (W \bullet \lambda)_{>
w_K \bullet \lambda}$.

\item For all highest weights $\mu$ ``in'' this filtration, a modified
form of BGG reciprocity holds:
\begin{equation}\label{EBGGrec}
\sum_{\substack{\mathcal{H}_0 \supseteq \mathcal{H}'_\mu,\\ \mathcal{H}_0
\; \text{\rm upper-closed} \\ \text{\rm in Indep}(J_\mu) }}
[P^{\mathcal{H}}(w_K \bullet \lambda) : \mathbb{M}(\mu, \mathcal{H}_0)] =
[P^{\mathcal{H}}(w_K \bullet \lambda) : \mathbb{M}(\mu, \mathcal{H}'_{w_K
\bullet \lambda})] = [\mathbb{M}(\mu, \mathcal{H}'_\mu) : L(w_K \bullet
\lambda)].
\end{equation}

\item The Cartan matrix for $\mathcal{O}^\mathcal{H}$ is symmetric,
i.e., if $L(\mu), L(\mu') \in \mathcal{O}^\mathcal{H}$, then
\begin{equation}\label{ECartanmatrix2}
[P^{\mathcal{H}}(\mu) : L(\mu')] = [P^{\mathcal{H}}(\mu') : L(\mu)].
\end{equation}
\end{enumerate}
\end{theorem}

Note the presence of the summation on the left side in~\eqref{EBGGrec},
in contrast to the BGG reciprocity formulas in the zeroth and first order
parabolic categories. This summation -- i.e.\ using multiple standard
objects over a given simple object -- is indeed needed when discussing
BGG reciprocity for general $\mathcal{O}^{\mathcal{H}}$, as was explained
above over $\mathfrak{g} = \mathfrak{sl}_2 \oplus \mathfrak{sl}_2$ and
all higher rank $\mathfrak{g}$.

The proof of Theorem~\ref{TBGGrec} will require computing the
integrability of the highest weights of the simples in the block
$\mathcal{O}^\lambda$. This is achieved by the following lemma, which
also bounds the integrabilities of all weights in the interval $[w_K
\bullet \lambda, \lambda]$, in greater generality than
$\mathfrak{sl}_2^{\oplus n}$.

\begin{lemma}\label{Lcover}
Fix semisimple $\mathfrak{g}$, a subset $\mathcal{H} \subset {\rm
Indep}(I)$, and a weight $\lambda \in \mathfrak{h}^*$ that is maximal in
its block in $\mathcal{O}$. Suppose
(a)~the integrability $J_\lambda$ is an independent set of nodes, that
further satisfies
(b)~the integrability of $w_{J_\lambda} \bullet \lambda$ is empty, where
$w_K = \prod_{k \in K} s_k$ for $K \subset J_\lambda$.
\begin{enumerate}
\item For all $K \subset J_\lambda$ one has $J_{w_K \bullet \lambda} =
J_\lambda \setminus K$.

\item For all subsets $K' \subset K \subset J_\lambda$ and weights $w_K
\bullet \lambda \leq \mu \leq w_{K'} \bullet \lambda$, one has the
inclusion of integrabilities: $J_{w_K \bullet \lambda} \subset J_\mu
\subset J_{w_{K'} \bullet \lambda}$.
In particular, if $L(w_K \bullet \lambda) \in \mathcal{O}^\mathcal{H}$
then $L(\mu) \in \mathcal{O}^\mathcal{H}$.

\item Suppose $L(w_K \bullet \lambda) \in \mathcal{O}^{\mathcal{H}}$
for some $K \subset J_\lambda$. Then $\mathbb{M}(w_K \bullet \lambda,
\mathcal{H}'_{w_K \bullet \lambda})$ has a subquotient $L(w_{K'}
\bullet \lambda)$ for some $K' \subset J_\lambda$, if and only if
(i)~$L(w_{K'} \bullet \lambda) \in \mathcal{O}^{\mathcal{H}}$ and
(ii)~$K' \supseteq K$.
\end{enumerate}
\end{lemma}

Here and below, we use $\bff_{H'} := \prod_{h \in H'} f_h^{\langle
\lambda, \alpha_h^\vee \rangle + 1}$ for $\lambda \in \mathfrak{h}^*$ and
$H' \subset J_\lambda$ an independent subset.
Also note that the hypothesis ``$w_{J_\lambda} \bullet \lambda$ has empty
integrability'' does not follow from the remaining hypotheses in the
lemma -- consider e.g.\ $\mathfrak{g} = \mathfrak{sl}_3$, $\lambda =
\alpha_2$.

\begin{proof}
Recall the hypothesis $J_{w_{J_\lambda} \bullet \lambda} = \emptyset$,
which is now used extensively without further reference. Also note that
$W_{J_\lambda} \bullet \lambda$ is in bijection with $W_{J_\lambda}$.
\begin{enumerate}
\item If $k \in K$, then 
$\langle w_K \bullet \lambda, \alpha_k^\vee \rangle = \langle (w_K
s_k) \bullet (s_k \bullet \lambda), \alpha_k^\vee \rangle = \langle s_k
\bullet \lambda, \alpha_k^\vee \rangle = - \langle \lambda, \alpha_k^\vee
\rangle - 2 < 0.$

Now write $w_K \bullet \lambda = w_{J_\lambda} \bullet \lambda + \sum_{j
\in J_\lambda \setminus K} l_j \alpha_j$, with all $l_j \in \Z_{\geq 0}$.
Given $i \in I \setminus J_\lambda$,
\[
\langle w_K \bullet \lambda, \alpha_i^\vee \rangle =
\langle w_{J_\lambda} \bullet \lambda, \alpha_i^\vee \rangle + \sum_{j
\in J_\lambda \setminus K} l_j \langle \alpha_j, \alpha_i^\vee \rangle
\not\in \Z_{\geq 0}.
\]
This shows one inclusion; the reverse inclusion is shown similarly.

\item Begin by writing:
$\mu = w_K \bullet \lambda + \sum_{k \in K \setminus K'} l_k \alpha_k =
\lambda - \sum_{k \in K} l'_k \alpha_k$, where $l_k, l'_k \in \Z_{\geq
0}.$
Now suppose $i \not\in J_{w_{K'} \bullet \lambda}$, so $i \not\in
J_{w_K \bullet \lambda}$ by~(1). Then
\[
\langle \mu, \alpha_i^\vee \rangle = \langle w_K \bullet \lambda,
\alpha_i^\vee \rangle + \sum_{k \in K \setminus K'} l_k \langle \alpha_k,
\alpha_i^\vee \rangle \notin \Z_{\geq 0}.
\]
This shows one inclusion. Next if $i
\in J_\lambda \setminus K$, then
$\langle \mu, \alpha_i^\vee \rangle = \langle \lambda,
\alpha_i^\vee \rangle - \sum_{k \in K} l'_k \langle
\alpha_k, \alpha_i^\vee \rangle \in \Z_{\geq 0}$. This shows the other
inclusion -- and also implies that
if $H \in \mathcal{H}^{\min}$ and $L(w_K \bullet \lambda) \in
\mathcal{O}^\mathcal{H}$, then
$J_\mu \cap H \supseteq J_{w_K \bullet \lambda} \cap H \neq \emptyset$,
so $L(\mu) \in \mathcal{O}^\mathcal{H}$ by
Proposition~\ref{Psimple}(1).

\item The result is straightforward if $\mathcal{H} = \emptyset$, since
one now works in the Verma module $M(w_K \bullet \lambda)$. Thus, assume
henceforth that $\mathcal{H} \neq \emptyset$.
Set $M := \mathbb{M}(w_K \bullet \lambda, \mathcal{H}'_{w_K \bullet
\lambda})$. The necessity of (i), (ii) easily follows from
Lemma~\ref{Lsanity} and the $\mathcal{H} = \emptyset$ case. Conversely,
if $K'$ satisfies (i), (ii) then it suffices to show the weight space $L
:= {\bf f}^{(w_K \bullet \lambda)}_{K' \setminus K} M_{w_K \bullet
\lambda}$ is nonzero.
As the preimage in $M(w_K \bullet \lambda) \cong U \mathfrak{n}^-$ of $L$
is a line, it suffices to show that ${\bf f}^{(w_K \bullet \lambda)}_{K'
\setminus K}$ is not in the left-ideal
\begin{equation}\label{Enec}
{\bf f}^{(w_K \bullet \lambda)}_{K' \setminus K} \not\in U
(\mathfrak{n}^-) \cdot \langle {\bf f}^{(w_K \bullet
\lambda)}_{(J_\lambda \setminus K) \cap H} : H \in \mathcal{H}^{\min}
\rangle.
\end{equation}

To show this, work with a ``PBW basis'' of $\mathfrak{n}^-$ in which
$f_k, k \in K$ occur to the right, preceded by $f_j, j \in J_\lambda
\setminus K$, then preceded by all other root vectors in
$\mathfrak{n}^-$. Since the roots indexed by $K' \setminus K$ are
pairwise orthogonal, no nontrivial Lie brackets among them exist. Thus if
\eqref{Enec} is false, then ${\bf f}^{(w_K \bullet \lambda)}_{(K'
\setminus K) \cap H} = {\bf f}^{(w_K \bullet \lambda)}_{(J_\lambda
\setminus K) \cap H}$ for some $H \in \mathcal{H}^{\min}$. Hence by
part~(1), $J_{w_{K'} \bullet \lambda} \cap H = (J_\lambda \setminus K')
\cap H = \emptyset$, which contradicts $L(w_{K'} \bullet \lambda) \in
\mathcal{O}^{\mathcal{H}}$ by Proposition~\ref{Psimple}(1). \qedhere
\end{enumerate}
\end{proof}

With Lemma~\ref{Lcover} at hand, we now have:

\begin{proof}[Proof of Theorem~\ref{TBGGrec}]
It suffices to work with the objects in a block / linkage class
$[\lambda] = W \bullet \lambda \cap (\lambda + \Z \Pi)$ that moreover lie
in $\mathcal{O}^\mathcal{H}$. Note that we may take $\lambda \in
\mathfrak{h}^*$ to be maximal in the block $[\lambda]$; now the simples
in $[\lambda]$ (in $\mathcal{O}$, not $\mathcal{O}^\mathcal{H}$) are
indexed by $W_{J_\lambda} \bullet \lambda$. Define
\begin{equation}\label{Ebracket}
[\lambda]_{\mathcal{H}} := \{ \mu \in [\lambda] : L(\mu)
\in \mathcal{O}^\mathcal{H} \}.
\end{equation}
(Thus, $[\lambda] = [\lambda]_\emptyset$.)
Throughout the rest of this proof, we will work in the category
\begin{equation}\label{EKstar}
\mathcal{O}^{\mathcal{H}, [\lambda]} :=
\mathcal{O}^{[\lambda]_{\mathcal{H}} } = \mathcal{O}^{\mathcal{H}} \cap
\mathcal{O}^{[\lambda]} \subset \mathcal{O}^{\mathcal{H}} \cap
\mathcal{O}^{W \bullet \lambda}.
\end{equation}

\noindent We also use without reference that the hypotheses of
Lemma~\ref{Lcover} hold over $\mathfrak{g} = \mathfrak{sl}_2^{\oplus n}$
for all $\lambda \in \mathfrak{h}^*$.

We first show BGG reciprocity at $\lambda$, by \textbf{claiming} that
the projective cover of $L(\lambda)$ in the block
$\mathcal{O}^{\mathcal{H}, [\lambda]}$ is $P^{\mathcal{H}}(\lambda) =
\mathbb{M}(\lambda, \mathcal{H}'_\lambda)$, where
\[
\mathcal{H}'_\lambda = \{ J_\lambda \cap H : H \in \mathcal{H}^{\min} \}
\]
as in Proposition~\ref{Psimple}(2).
Indeed, this object in $\mathcal{O}^{\mathcal{H}, [\lambda]}$ is
indecomposable and surjects onto $L(\lambda)$, and a relatively standard
argument (see e.g.\ the proof of Step~$4$ in showing that
$\mathcal{O}^{\mathcal{H}, [\lambda]}$ has enough projectives, in the
Appendix) shows the functorial isomorphism
\[
{\rm Hom}_{\mathcal{O}^{\mathcal{H}} }(\mathbb{M}(\lambda,
\mathcal{H}'_\lambda), M) \cong M_\lambda, \qquad \forall M \in
\mathcal{O}^{\mathcal{H}, [\lambda]},
\]
i.e., that ${\rm Hom}_{\mathcal{O}^{\mathcal{H}} }(\mathbb{M}(\lambda,
\mathcal{H}'_\lambda), -)$ co-represents the $\lambda$-weight space in
$\mathcal{O}^{\mathcal{H}, [\lambda]}$. This proves the claim, and BGG
reciprocity~\eqref{EBGGrec} involving $P^{\mathcal{H}}(\lambda),
L(\lambda)$ now follows e.g.\ by Lemma~\ref{Lcover}(3).\medskip

\noindent \textit{\underline{BGG reciprocity at $w_K \bullet \lambda$,
for $\emptyset \subsetneq K \subset J_\lambda$:}}

Next, we work with every other simple in $[\lambda]$ that occurs in
$\mathcal{O}^{\mathcal{H}}$. From above, we call it $L(w_K \bullet
\lambda)$, where $\emptyset \subsetneq K \subset J_\lambda$ is fixed and
$w_K = \prod_{k \in s_k}$. Proving reciprocity requires working with
those simples in $\mathcal{O}^\mathcal{H} \cap \mathcal{O}^{W \bullet
\lambda}$ which lie above $w_K \bullet \lambda$ in the standard ordering.
These simples are indexed by $W_K \bullet \lambda$. We identify $W_K
\overset{\psi}{\simeq} (\Z / 2\Z)^{\oplus K} \simeq \{ 0, 1 \}^K$, and so
list $W_K = \{ w_{K'} = \prod_{k \in K'} s_k : K' \subset K \}$.

We now prove BGG reciprocity at the weight $\lambda_K := w_K \bullet
\lambda$. For ease of reading, the remainder of this proof is split into
steps.\medskip

\noindent \textit{\underline{Step 1:}
The BGG construction of a cyclic module, and its standard filtration.}
Recalling that $m_i = \langle \lambda, \alpha_i^\vee \rangle + 1$, we
first define and study -- akin to \cite{BGG} for $\mathcal{O}$ -- the
cyclic module $P := U \mathfrak{g} / I_K$, where
\begin{equation}\label{Eideal2}
I_K := U \mathfrak{g} \cdot \left( \{ h - \lambda_K(h) : h \in
\mathfrak{h} \}, \ \{ e_{\alpha_k}^{m_k + 1} : k \in K \}, \ \{ e_\alpha
: \alpha \in \Delta^+ \setminus \Delta_K^+ \}, \ \{ {\bf
f}^{(\lambda_K)}_{J_{\lambda_K} \cap H} : H \in \mathcal{H}^{\min} \}
\right).
\end{equation}
%As an initial observation, note via the hypotheses that
%\begin{equation}\label{EfK}
%{\bf f}^{(\lambda_K)}_{J_{\lambda_K} \cap H} = \prod_{h \in (K_*
%\setminus K) \cap H} f_h^{\langle w_K \bullet \lambda, \alpha_h^\vee
%\rangle + 1} = \prod_{h \in (K_* \setminus K) \cap H} f_h^{\langle
%\lambda, \alpha_h^\vee \rangle + 1} = {\bf f}^{(\lambda)}_{(K_* \setminus
%K) \cap H}.
%\end{equation}

\noindent (Note that $\Delta^+ \setminus \Delta^+_K = \Pi_{I \setminus
K}$.)
Let ${\bf p} = {\bf p}_{\lambda_K}$ denote the image of $1_{U
\mathfrak{g}}$ in $P$. Then there is a lattice of $\prod_{k \in K} (m_k +
1)$-many submodules of $P$, indexed by integer tuples ${\bf l} =
(l_k)_{k \in K}$:
\begin{equation}\label{Efiltration}
P = P_{\bf 0} := U \mathfrak{g} \cdot {\bf p} \ \supseteq \quad
P_{\bf l} := U \mathfrak{g} \cdot \prod_{k \in K}
e_{\alpha_k}^{l_k} \cdot {\bf p} \quad \supseteq \
P_{\bf m} := U \mathfrak{g} \cdot \prod_{k \in K} e_{\alpha_k}^{m_k}
\cdot {\bf p} \ \supseteq \ 0, \qquad 0 \leq l_k \leq m_k\ \forall k.
\end{equation}
This yields $\prod_{k \in K} (m_k + 1)$-many subquotients, each of the
form
\[
Q_{\bf l} := \frac{U \mathfrak{g} \cdot \prod_{k \in K}
e_{\alpha_k}^{l_k} \cdot {\bf p}}{\sum_{k' \in K} U \mathfrak{g} \cdot
e_{\alpha_{k'}} \prod_{k \in K} e_{\alpha_k}^{l_k} \cdot {\bf p}}, \qquad
{\bf l} \in \times_{k \in K} [0, m_k], \text{ i.e., } {\bf 0} \leq {\bf
l} \leq {\bf m},
\]
where the inequalities are coordinatewise. 
Each subquotient $Q_{\bf l}$ is generated by a maximal vector
\[
\prod_{k \in K} e_{\alpha_k}^{l_k} \cdot {\bf p}, \qquad \text{of weight
} \mu_{\bf l} := w_K \bullet \lambda + \sum_{k \in K} l_k \alpha_k,
\]
so $Q_{\bf l}$ is a highest weight module. We now show that the
filtration by the $Q_{\bf l}$ is indeed standard:

\begin{lemma}
Given a tuple ${\bf l} \in [{\bf 0}, {\bf m}]$, with $\mu_{\bf l}$ as
above,
$Q_{\bf l} \cong \mathbb{M}(\mu_{\bf l}, \mathcal{H}'_{w_K \bullet
\lambda}) \in \mathcal{O}^{\mathcal{H}}$.
\end{lemma}

In particular, by Lemma~\ref{Lbasic}, the cyclic module $P \in
\mathcal{O}^{\mathcal{H}}$, and this lemma provides a ``standard
filtration'' for it, albeit in $\mathcal{O}^{\mathcal{H}}$ and not in a
single block $\mathcal{O}^{\mathcal{H}, [\lambda]}$.

\begin{proof}
Note that $Q_{\bf l}$ has highest weight $\mu_{\bf l} := w_K \bullet
\lambda + \sum_{k \in K} l_k \alpha_k$, and so by Lemma~\ref{Lcover}(2),
its integrability $J_{\mu_{\bf l}} \supseteq J_{w_K \bullet \lambda}$. It
follows that $\mathbb{M}(\mu_{\bf l}, \mathcal{H}'_{w_K \bullet \lambda})
\in \mathcal{O}^{\mathcal{H}}$, by applying Proposition~\ref{Psimple}(3)
with $\lambda \leadsto \mu_{\bf l}$ and $\mathcal{H}_V \leadsto$ the
upper-closure of $\mathcal{H}'_{w_K \bullet \lambda}$.

We now show $Q_{\bf l} \cong \mathbb{M}(\mu_{\bf l},
\mathcal{H}'_{w_K \bullet \lambda})$, starting with an upper bound on
$Q_{\bf l}$, via the generator-coset
\[
{\bf p} = 1 + {\rm span}_{U \mathfrak{g}} \{ h - (w_K \bullet
\lambda)(h), \ e_{\alpha_k}^{m_k + 1}, \ e_\alpha, \ {\bf f}^{(w_K \bullet
\lambda)}_{J_{w_K \bullet \lambda} \cap H} \}
\]
(see \eqref{Eideal2} and Lemma~\ref{Lcover}(2)). Now note that
$e_k^{l_k}$ (for $k \in K$) commutes with all ${\bf f}^{(w_K \bullet
\lambda)}_{J_{w_K \bullet \lambda} \cap H}$, and
\[
e_{\alpha_k}^{l_k} (h - \nu(h)) = (h - (\nu
+ l_k \alpha_k)(h)) e_{\alpha_k}^{l_k}, \qquad \forall h \in
\mathfrak{h}, \ \nu \in \mathfrak{h}^*.
\]
Using these relations and that
\begin{equation}\label{Emul}
{\bf f}^{(\mu_{\bf l})}_{J_{w_K \bullet \lambda} \cap H} =
{\bf f}^{(w_K \bullet \lambda)}_{J_{w_K \bullet \lambda} \cap H}, \qquad
\forall {\bf 0} \leq {\bf l} \leq {\bf m}
\end{equation}
which follows from the independence of the set $J_\lambda$,
it follows that $Q_{\bf l}$ is a quotient of
\[
\frac{U \mathfrak{g}}{U \mathfrak{g} \left( \{ h - \mu_{\bf l}(h) : h \in
\mathfrak{h} \}, \ \{ e_i : i \in I \}, \ \{ {\bf f}^{(w_K \bullet
\lambda)}_{J_{w_K \bullet \lambda} \cap H} : H \in \mathcal{H}^{\min} \}
\right)}.
\]
But this is precisely $\mathbb{M}(\mu_{\bf l}, \mathcal{H}'_{w_K \bullet
\lambda})$.

This provides an upper bound on (the character of) $Q_{\bf
l}$; we now show it is also a lower bound, which will prove the
lemma.\medskip

\noindent \textbf{Claim.}
\textit{Fix an ordering on $I$, hence the ordered basis $\{ f_i : i \in I
\}$ of $\mathfrak{n}^-$, and denote by $\mathcal{B}$ the corresponding
PBW / monomial basis of $U \mathfrak{n}^- \simeq \mathbb{C}[ f_i : i \in
I]$. Now let $\mathcal{B}'$ denote the subset of monomials which are not
divisible by ${\bf f}^{(w_K \bullet \lambda)}_{J_{w_K \bullet \lambda}
\cap H}$ for any $H \in \mathcal{H}^{\min}$. Then the vectors
$\{ b' \cdot {\textstyle \prod}_{k \in K} e_{\alpha_k}^{l_k} \cdot {\bf
p} \ | \ b' \in \mathcal{B}' \}$ are independent in $Q_{\bf l}$.}\medskip

\noindent (By~\eqref{Emul}, $\mathcal{B}'$ is in bijection with a basis
of $\mathbb{M}(\mu_{\bf l}, \mathcal{H}'_{w_K \bullet \lambda})$,
yielding the desired lower bound on $Q_{\bf l}$.)

The claim will follow via the filtration of $P$ by the $Q_{\bf l}$ (so we
no longer fix ${\bf l}$ in the proof of this lemma), from the statement
that:
\[
\widetilde{\mathcal{B}} := \{ b' \cdot \prod_{k \in K} e_{\alpha_k}^{l_k}
\cdot {\bf p} \ : \ 0 \leq l_k \leq m_k, \ k \in K, \ b' \in \mathcal{B}'
\}
\]
\textit{is a basis of the module $P = U \mathfrak{g} / I_K$ (where the
left ideal $I_K$ was defined in~\eqref{Eideal2}).}\medskip

To show this statement, write using the PBW theorem and changing
variables:
\[
U \mathfrak{g} \simeq \mathbb{C}[ f_i : i \in I ] \otimes \mathbb{C}[
e_i : i \in I ] \otimes \mathbb{C}[ \alpha_i^\vee : i \in I ].
\]
Quotient by $I_K$ in stages: first quotienting by the relations $\{
\alpha_i^\vee - \langle w_K \bullet \lambda, \alpha_i^\vee \rangle : i
\in I \}$ yields
\begin{equation}\label{Einter1}
U \mathfrak{g} / ( \alpha^\vee_i : i \in I ) \simeq \mathbb{C}[ f_i : i
\in I ] \otimes \mathbb{C}[ e_i : i \in I ].
\end{equation}
Next, further quotienting this by all $e_{\alpha_k}^{m_k + 1}$ and
$e_{\alpha_i}$ for $k \in K, i \not\in K$ yields
\begin{equation}\label{Einter2}
\mathbb{C}[ f_i : i \in I ] \otimes {\rm span}_{\mathbb{C}} \lbrace
\prod_{k \in K} e_{\alpha_k}^{l_k} \cdot \overline{1_{U \mathfrak{g}} }
\rbrace.
\end{equation}
More precisely, one needs to quotient~\eqref{Einter1}, by the image in it
of the space of vectors
\[
\sum_{k \in K} X_k \cdot e_{\alpha_k}^{m_k + 1} + \sum_{i \not\in K} X'_i
\cdot e_{\alpha_i}, \qquad X_k, X'_i \in U \mathfrak{g}.
\]
Writing each $X_k, X'_i$ as linear combinations of PBW monomials in the
above ordered basis of $f_i, e_i, \alpha^\vee_i$, one
obtains~\eqref{Einter2}.

Finally, one obtains $P = U \mathfrak{g} / I_K$ by
quotienting~\eqref{Einter2} by the image in it of the space of vectors
\begin{equation}\label{Einter3}
\sum_{H \in \mathcal{H}^{\min}} X_H \cdot {\bf
f}^{(w_K \bullet \lambda)}_{J_{w_K \bullet \lambda} \cap H}, \quad
X_H \in U \mathfrak{g}.
\end{equation}

Write $X_H = \sum_t p_{H,t}(\{ f_i \}) \cdot q_{H,t}(\{ e_i \}) \cdot
r_{H,t}(\{ \alpha^\vee_i \})$, and note that the polynomial $r_{H,t}(\{
\alpha^\vee_i \})$ ``goes past'' ${\bf f}^{(w_K \bullet \lambda)}_{J_{w_K
\bullet \lambda} \cap H}$ yielding scalars, in the quotient
space~\eqref{Einter2}. Thus, we may suppose all $r_{H,t} \in \mathbb{C}$.
Next, use a standard $\mathfrak{sl}_2$-calculation with $m'_i := \langle
w_K \bullet \lambda, \alpha_i^\vee \rangle + 1 \in \mathbb{Z}_{>0}$ for
$i \in J_{w_K \bullet \lambda} \cap H$:
\begin{equation}\label{Einter4}
e_i f_i^{m'_i} = f_i^{m'_i} e_i + m'_i f_i^{m'_i - 1} (\alpha_i^\vee -
\langle w_K \bullet \lambda, \alpha_i^\vee \rangle) = f_i^{m'_i} e_i +
m'_i f_i^{m'_i - 1} (\alpha_i^\vee - \langle \lambda, \alpha_i^\vee
\rangle),
\end{equation}
where the final equality is because $i \in J_{w_K \bullet \lambda} \cap H
= (J_\lambda \cap H) \setminus K$.

It follows from~\eqref{Einter4} that if $e_{\alpha_i}$ divides any
$q_{H,t}$ for $i \in J_{w_K \bullet \lambda} \cap H$, then that monomial
vanishes in~\eqref{Einter2}. The same happens if $e_i | q_{H,t}$ for some
$i \not\in J_{w_K \bullet \lambda} \cap H$ and $i \in I \setminus K$, or
if $e_{\alpha_k}^{m_k + 1} | q_{H,t}$ for some $k \in K$. Thus, we reduce
to computing the quotient of~\eqref{Einter2} by the subspace of vectors
\[
\sum_{H \in \mathcal{H}^{\min}} X_H \cdot {\bf
f}^{(w_K \bullet \lambda)}_{J_{w_K \bullet \lambda} \cap H}, \qquad X_H =
\sum_{{\bf 0} \leq {\bf l} \leq {\bf m}} p_{H,{\bf l}}(\{ f_i \}) \cdot
\prod_{k \in K} e_{\alpha_k}^{l_k} \cdot \overline{1_{U \mathfrak{g}} }.
\]

As the second tensor factors in $X_H$ and~\eqref{Einter2} coincide, and
since ${\bf f}^{(w_K \bullet \lambda)}_{J_{w_K \bullet \lambda} \cap H}$
commutes with all $e_{\alpha_k}$, this yields precisely ${\rm
span}_{\mathbb{C}}(\mathcal{B}') \otimes {\rm span}_{\mathbb{C}} \lbrace
\prod_{k \in K} e_{\alpha_k}^{l_k} \cdot \overline{1_{U \mathfrak{g}} }
\rbrace$, which is indeed spanned by $\widetilde{\mathcal{B}}$.
\end{proof}

\noindent \textit{\underline{Step 2:} The cyclic module is projective,
and its $[\lambda]_\mathcal{H}$-summand co-represents the $(w_K \bullet
\lambda)$-weight space.}
The previous step shows that the module $P = U \mathfrak{g} / I_K \in
\mathcal{O}^{\mathcal{H}}$ has a standard filtration by the modules
$Q_{\bf l} \cong \mathbb{M}(\mu_{\bf l}, \mathcal{H}'_{w_K \bullet
\lambda}) \in \mathcal{O}^{\mathcal{H}}$ for ${\bf 0} \leq {\bf l} \leq
{\bf m}$. Moreover, since $K \neq \emptyset$ here, $w_K \bullet \lambda$
has strictly smaller integrability than $\lambda$ by
Lemma~\ref{Lcover}(2) (and $L(w_K \bullet \lambda) \in
\mathcal{O}^{\mathcal{H}}$). Thus the upper-closure in ${\rm Indep}(I)$
of $\mathcal{H}$ is contained in that of $\mathcal{H}'_{w_K \bullet
\lambda}$, which we denote henceforth as $\overline{\mathcal{H}'}$ for
convenience.

Now $P \in \mathcal{O}^{\overline{\mathcal{H}'}} \subset
\mathcal{O}^{\mathcal{H}}$. We claim that $P$ is projective in the
truncated subcategory of objects with all weights $\leq \lambda$, which
we denote by $\mathcal{O}^{\mathcal{H}}_{\leq \lambda}$. The claim is
shown (as above) by noting that ${\rm
Hom}_{\mathcal{O}^{\mathcal{H}}_{\leq \lambda}}(P,-)$ co-represents the
$(w_K \bullet \lambda)$-weight space.
Now use the block decomposition to write
\[
P = \oplus_{[\mu]} P^{[\mu]},\quad with \quad P^{[\mu]} \in
\mathcal{O}^{\overline{\mathcal{H}'}, [\mu]} \subset
\mathcal{O}^{\mathcal{H}, [\mu]}.
\]
Then $P^{[\lambda]}$ is projective in $\mathcal{O}^{\mathcal{H},
[\lambda]} \subset \mathcal{O}^{\mathcal{H}}_{\leq \lambda}$, and it
co-represents in $\mathcal{O}^{\mathcal{H}, [\lambda]}$ the $(w_K \bullet
\lambda)$-weight space.\medskip

\noindent \textit{\underline{Step 3:} The projective cover is the
$[\lambda]_\mathcal{H}$-summand -- hence, BGG reciprocity at $w_K \bullet
\lambda$.} Given the standard filtration (above) of the module $P =
\oplus_{[\mu]} P^{[\mu]}$ in $\mathcal{O}^{\overline{\mathcal{H}'}}$,
Proposition~\ref{Pdirect} now applies (with $\mathcal{H} \leadsto
\overline{\mathcal{H}'}$) to show that $P^{[\lambda]}$ has a standard
filtration in $\mathcal{O}^{\overline{\mathcal{H}'}} \subset
\mathcal{O}^{\mathcal{H}}$. Moreover, the subquotients occur from among
the $Q_{\bf l}$, hence each have highest weight $> w_K \bullet \lambda$;
and the ``topmost'' quotient is indeed $Q_{\bf 0} \cong \mathbb{M}(w_K
\bullet \lambda, \mathcal{H}'_{w_K \bullet \lambda})$.

This concludes the proof of the assertions on the standard filtration of
$P^{[\lambda]}$. We now turn to the proof of BGG reciprocity at $w_K
\bullet \lambda$. Note that the subquotient $Q_{\bf l}$ belongs to the
block $\mathcal{O}^{[\lambda]}$, if and only if every $l_k$ is either $0$
or $m_k$, in which case one obtains a highest weight of $w_{K'} \bullet
\lambda$ for some $\emptyset \subset K' \subset K$ -- moreover, for each
such $K'$ it is the highest weight of a unique subquotient $Q_{\bf l}$.

We now \textbf{claim} that \textit{the summand $P^{[\lambda]}$ of $P$ is
precisely the (indecomposable) projective cover of $L(w_K \bullet
\lambda)$ in $\mathcal{O}^{\overline{\mathcal{H}'}}$, hence in the larger
category $\mathcal{O}^{\mathcal{H}}$.} Now BGG reciprocity follows,
because
\[
\sum_{\mathcal{H}_0 \text{ upper-closed}} [P^{\mathcal{H}}(w_K \bullet
\lambda) : \mathbb{M}(w_{K'} \bullet \lambda, \mathcal{H}_0)] = {\bf
1}_{K \supseteq K'} =
[\mathbb{M}(w_{K'} \bullet \lambda, \mathcal{H}'_{w_{K'} \bullet
\lambda}) : L(w_K \bullet \lambda)].
\]
(The first equality is from the claim just above, and the second is by
Lemma~\ref{Lcover}(3).)

Thus, it remains to prove the claim. The following argument is due to
Gurbir Dhillon. As explained above, the summand $P^{[\lambda]}$
co-represents the $(w_K \bullet \lambda)$-weight space in
$\mathcal{O}^{\overline{\mathcal{H}'}, [\lambda]}$. On the other hand,
the projective cover $P' := P^{\overline{\mathcal{H}'}}(w_K \bullet
\lambda)$ is characterized by the equation
\[
\dim {\rm Hom}_{\mathcal{O}^{\overline{\mathcal{H}'}, [\lambda]}}(P',
L(\nu)) = \delta_{w_K \bullet \lambda, \; \nu}, \qquad \forall \nu \in
[\lambda]_\mathcal{H}. \]
It is also standard that $P^{[\lambda]} \twoheadrightarrow P'$, and that
the kernel is a sum of (copies of)
$P^{\overline{\mathcal{H}'}}(\nu)$ for $\nu > w_K \bullet \lambda$. Thus,
$P^{[\lambda]} \cong P'$ if and only if no other indecomposable
projective summands occur, i.e.\ if and only if $L(\nu)$ has no $(w_K
\bullet \lambda)$-weight space if $\nu \in [\lambda]_{> w_K \bullet
\lambda}$ (the contrapositive of this may be easier to see). This
indeed holds by $\mathfrak{sl}_2^{\oplus K}$-theory in
$[\lambda]_\mathcal{H}$, and shows BGG reciprocity at $w_K \bullet
\lambda$.

It remains to show~\eqref{ECartanmatrix2}. First note by definition
that in the Grothendieck group $K_0(\mathcal{O}^\mathcal{H})$,
\begin{equation}
[\mathbb{M}(\lambda, \mathcal{H})] = \sum_{J \in \mathfrak{J}(\lambda,
\mathcal{H})} L(w_J \bullet \lambda), \quad \text{where} \quad
\mathfrak{J}(\lambda, \mathcal{H}) := \{ J \subset J_\lambda : J
\not\supseteq J_\lambda \cap H \ \forall H \in \mathcal{H}^{\min}
\},
\end{equation}
since $\mathfrak{g} = \mathfrak{sl}_2^{\oplus n}$.
Now with $[\lambda]_{\mathcal{H}}$ as in~\eqref{Ebracket}, $\lambda$
maximal in $[\lambda]_{\mathcal{H}}$ as in the lines preceding it, and
from the above proof,
\[
[P^\mathcal{H}(w_K \bullet \lambda)] = \sum_{\emptyset \subset K' \subset
K} [\mathbb{M}(w_{K'} \bullet \lambda, \mathcal{H}'_{w_K \bullet
\lambda})] = \sum_{\emptyset \subset K' \subset K} \sum_{J_1 \in
\mathfrak{J}(w_{K'} \bullet \lambda, \mathcal{H}'_{w_K \bullet \lambda})}
[L(w_{J_1 \sqcup K'} \bullet \lambda)].
\]
We now show~\eqref{ECartanmatrix2} for $\mu = w_K \bullet \lambda, \mu' =
w_{K_1} \bullet \lambda \in [\lambda]_\mathcal{H}$ (so $K, K_1 \subset
J_\lambda$). First note that the pairs $(K',J_1)$ in the preceding sum that
contribute to
$[P^\mathcal{H}(w_K \bullet \lambda) : L(w_{K_1} \bullet \lambda)]$
satisfy: $K' \subset K \cap K_1, J_1 = K_1 \setminus K'$. Thus, these pairs
are in bijection with the pairs $(K', J = K \setminus K')$ contributing
to
$[P^\mathcal{H}(w_{K_1} \bullet \lambda) : L(w_K \bullet \lambda)]$.
Since $J_{w_K \bullet \lambda} = J_\lambda \setminus K$,
one checks that a pair $(K',J_1)$ does \textit{not} contribute a
multiplicity to
$[P^\mathcal{H}(w_K \bullet \lambda) : L(w_{K_1} \bullet \lambda)]$,
if and only if there exists $H \in \mathcal{H}^{\min}$ such that $H \cap
J_\lambda \subset K \cup K_1$. As this latter condition is symmetric in
$K$ and $K_1$, we are done.
\end{proof}

\subsection{Kazhdan--Lusztig combinatorics in
$\mathcal{O}^{\mathcal{H}}$}

As a related discussion, we study a \textit{quotient space} of the
regular representation of the Iwahori--Hecke algebra $\mathscr{H}(W)$,
which relates naturally to the Grothendieck group of
$\mathcal{O}^{\mathcal{H}}$. To begin, consider the regular block
$\mathcal{O}^{\mathcal{H},\, W \bullet 0}$ over $\mathfrak{g} =
\mathfrak{sl}_2^{\oplus n}$.

We adopt standard convention, found e.g.\ in Equation~\eqref{EKL}. Thus,
the Weyl group is $W = \{ w_K : K \subset \{ 1, \dots, n \} \} \simeq
S_2^n$, and the simples in the block are $L(w_K w_\circ \bullet 0)$. Here
$w_\circ$ is the longest element in $W$ and $w_K = \prod_{k \in K} s_k$
as above. Also note that
\[
\mathscr{H}(W) \cong \otimes_{i=1}^n \mathscr{H}(\langle 1, s_i \rangle)
= \otimes_{i=1}^n (R 1 \oplus R T_i)
\]
for a suitable Laurent polynomial ring $R$ over $\mathbb{Z}$. In
particular, there is a monomial basis:
\begin{equation}
\mathscr{H}(W) = {\rm span}_R \left\{ T_K := \prod_{k \in K} T_k \ \bigg| \
K \subset \{ 1, \dots, n \} \right\}.
\end{equation}

Moreover, each factor Hecke algebra has a Kazhdan--Lusztig basis $\{ 1,
C_i \}$, so $\mathscr{H}(W)$ has the corresponding Kazhdan--Lusztig basis
\begin{equation}
\mathscr{H}(W) = {\rm span}_R \left\{ C_K := \prod_{k \in K} C_k \ \bigg|
\ K \subset \{ 1, \dots, n \} \right\}.
\end{equation}

As is well known, in $\mathcal{O}$ one can specialize to $q=1$ and
interpret the change-of-basis relations between the $C_K$ and $T_{K'}$ in
the Grothendieck group $K_0(\mathcal{O}^{W \bullet \lambda})$, via:
\begin{equation}\label{EKLbases}
C_K \leadsto [L(w_K w_\circ \bullet 0)], \qquad
T_K \leadsto [M(w_K w_\circ \bullet 0)].
\end{equation}
These relations are precisely the ones in~\eqref{EKL} with $\lambda
\leadsto 0$; since one is working over $\mathfrak{sl}_2^{\oplus n}$, all
Kazhdan--Lusztig polynomials $P_{x,w} \equiv {\bf 1}_{x \leq w}$, with
$\leq$ the Bruhat order.

Our goal here is to explain that this decategorification phenomenon holds
more generally, in $\mathcal{O}^{\mathcal{H}}$ for all $\mathcal{H}$ over
$\mathfrak{sl}_2^{\oplus n}$. We first illustrate this in the special
case of Example \ref{Erank2} (above) over $\mathfrak{sl}_2^{\oplus 2}$.

\begin{example}
Let $\mathfrak{g} = \mathfrak{sl}_2 \oplus \mathfrak{sl}_2$ and
$\mathcal{H} = \{ \{ 1, 2 \} \}$. The block $\mathcal{O}^{\mathcal{H}, \,
W \bullet 0}$ has three simple objects $L(w_K w_\circ \bullet 0)$ for $K
\neq \emptyset$, and their universal covers $\mathbb{M}(w_K w_\circ
\bullet 0, \mathcal{H}'_{w_K w_\circ \bullet 0})$ were worked out in
Example \ref{Erank2}. We now write down the characters of the simple
objects in $K_0(\mathcal{O}^{\{ \{ 1, 2 \} \}})$ via~\eqref{EKLbases}:
\[
C_{\{ 1 \}} = T_1, \qquad
C_{\{ 2 \}} = T_2, \qquad
C_{\{ 1, 2 \}} = T_1 T_2 - T_1 - T_2.
\]
Notice, these relations are not correct on the nose in $\mathscr{H}(W)$,
since there are no coefficients of the unit $1 = T_{\emptyset}$ in any of
them.
However, these relations are the images of the usual Kazhdan--Lusztig
relations~\eqref{EKL} in the quotient free $R$-module $\mathscr{H}(W) /
(R \cdot T_{\emptyset})$ -- which reflects in
$K_0(\mathcal{O}^{\mathcal{H}, W \bullet 0})$. \qed
\end{example}

The story is similar over $\mathfrak{sl}_2^{\oplus n}$. One can compute
the character of $\mathbb{M}(\lambda, \mathcal{H}) \in
\mathcal{O}^{\mathcal{H}}$ in two ways:
\begin{enumerate}
\item Via a BGG-type resolution in terms of Verma modules in the usual
Category $\mathcal{O}$. This is worked out in slightly greater generality
in the next section -- see Theorem~\ref{TBGG2}.

\item Alternately, one works internally inside
$\mathcal{O}^{\mathcal{H}}$ itself. In this case, one needs to compute
the character (or the image in $K_0$) of $\mathbb{M}(\lambda,
\mathcal{H}) = \mathbb{M}(\lambda, \mathcal{H}'_\lambda)$, whenever
$L(\lambda) \in \mathcal{O}^{\mathcal{H}}$. This is worked out in the
next result.
\end{enumerate}

\begin{prop}\label{PKL}
Fix $\mathfrak{g} = \mathfrak{sl}_2^{\oplus n}$ and a nonempty subset
$\mathcal{H} \subset {\rm Indep}(I) = 2^I$. Given $\lambda \in
\mathfrak{h}^*$, define $K_* := \{ i \in I : \langle \lambda,
\alpha_i^\vee \rangle \in \mathbb{Z} \}$, and suppose $\lambda$ is
$K_*$-dominant integral. 
Let $[\lambda]_\mathcal{H} \subset [\lambda] = W_{K_*} \bullet \lambda$
index the set of simples in $\mathcal{O}^{\mathcal{H}}$ as in
\eqref{EKstar}. Define
\begin{equation}\label{EHeckequot}
\mathscr{H}^\mathcal{H}(W_{K_*}) :=
\frac{\mathscr{H}(W_{K_*})}{\sum_{J \not\in [\lambda]_{\mathcal{H}} } R
\cdot T_J}, \qquad C_K^\mathcal{H} := [L(w_K w_\circ \bullet \lambda)],
\qquad T_K^\mathcal{H} := [\mathbb{M}(w_K w_\circ \bullet \lambda,
\mathcal{H}'_{w_K w_\circ \bullet \lambda})]
\end{equation}
(the last two definitions extend \eqref{EKLbases}), where $w_\circ =
w_{K_*}$ and $K \in [\lambda]_{\mathcal{H}}$.
Then the ``truncated'' Kazhdan--Lusztig relations over $W_{K_*}$ hold in
$K_0(\mathcal{O}^{\mathcal{H}})$, i.e.\ in the space
$\mathscr{H}^\mathcal{H}(W_{K_*})$ with $q=1$:
\begin{equation}\label{EKLOH}
T_K^\mathcal{H} = \sum_{K' \subset K \, : \, w_{K'} w_\circ \bullet
\lambda \in [\lambda]_\mathcal{H}} C_{K'}^\mathcal{H},
\qquad C_K^\mathcal{H} = \sum_{K' \subset K \, : \, w_{K'} w_\circ
\bullet \lambda \in [\lambda]_\mathcal{H}} (-1)^{|K|-|K'|}
T_{K'}^\mathcal{H}.
\end{equation}
\end{prop}

\begin{proof}
The first equation follows directly from Lemma~\ref{Lcover}(3). The
subtlety here is that one is now working in the finite poset
$[\lambda]_\mathcal{H}$ rather than the full block $[\lambda]$. Set
$W_{K_*} \simeq S_2^{K_*}$ to be the Weyl group of the block $[\lambda]$,
where one identifies
\[
K \subset K_* \quad \longleftrightarrow \quad w_K = \prod_{k \in
K} s_k \quad \longleftrightarrow \quad w_K w_\circ \bullet \lambda.
\]
As the regular representation of $\mathscr{H}(W_{K_*})$ is the
Grothendieck ring of the full block $\mathcal{O}^{[\lambda]} =
\mathcal{O}^{\emptyset, [\lambda]}$, as above one works in the quotient
space $\mathscr{H}^\mathcal{H}(W_{K_*})$ in~\eqref{EHeckequot}.

With this modification in place, the rest is standard. The incidence
algebra of functions $f : \overline{[\lambda]}_{\mathcal{H}} \times
\overline{[\lambda]}_{\mathcal{H}} \to \Z$ (with $f(x,w) = 0$ if $x
\not\leq w$) acts on the space of functions ${\rm
Fun}(\overline{[\lambda]}_{\mathcal{H}},
K_0(\mathcal{O}^{\mathcal{H},[\lambda]}))$ via convolution. To show the
second equation in \eqref{EKLOH}, note that the first says: ${\rm id}_C
\ast \zeta = {\rm id}_T$, with
\[
{\rm id}_C (K) := C_K^\mathcal{H}, \qquad
{\rm id}_T (K) := T_K^\mathcal{H},
\]
and $\zeta(K',K) = {\bf 1}_{K' \subset K}$ the zeta function of the
incidence algebra (whose convolution-inverse is precisely the M\"obius
function).
The second equation in~\eqref{EKLOH} now follows by M\"obius inversion,
noting that since $[\lambda]_\mathcal{H} \subset [\lambda]$ is
upper-closed, it inherits the M\"obius function $(-1)^{|K|-|K'|}$ of
$[\lambda]$.
\end{proof}
%}}}

%{{{1 Section 7 - Theorem E: Computing characters and BGG resolutions using minimal holes
\section{Theorems~\ref{T5}, \ref{Theorem character of second order Verma via parabolic Vermas} and Proposition \ref{Prop A3 second order holes character}: Characters and BGG resolutions, by Weyl semigroups}\label{Sfinal}

In this concluding section, we initiate the study of characters of some
of the modules in this work. Given that the characters of $M(\lambda)$
and $M(\lambda,J)$ are well understood -- in fact, at the level of BGG
resolutions -- it is natural to seek the same for the more general class
of higher order Verma modules $\mathbb{M}(\lambda, \mathcal{H})$. We
obtain such resolutions in two settings.

\textbf{Fix a Kac--Moody $\mathfrak{g}$ and a (highest) weight $\lambda
\in \mathfrak{h}^*$} for this section.
Given an independent subset $H \subset J_\lambda$, recall the weight
$\lambda_H$ and the lowering operator-product $\bff_H$ defined above:
\begin{equation}\label{EfH}
\lambda_H := \lambda - \sum_{h \in H} (\langle \lambda, \alpha_h^\vee
\rangle + 1 ) \alpha_h = (\textstyle{\prod}_{h \in H} s_h ) \bullet
\lambda, \qquad 
\bff_H := \prod_{h \in H} f_h^{\langle \lambda, \alpha_h^\vee \rangle +
1}.
\end{equation}
%Thus, the weight space $M(\lambda)_{\lambda_H} = \bff_H \cdot
%M(\lambda)_\lambda$, and this line consists of maximal vectors that
%vanish in a quotient module $V$ if and only if $H \in \mathcal{H}_V$.

\subsection{Setting 1: Pairwise orthogonal holes, ``parabolic'' Weyl
group}\label{Subsection Setting 1}

We now turn to the first setting in which we compute $\ch
\mathbb{M}(\lambda, \mathcal{H})$: when the elements of
$\mathcal{H}^{\min}$ are pairwise orthogonal.

\begin{theorem}\label{TBGG1}
Fix Kac--Moody $\mathfrak{g}$, a weight $\lambda \in \mathfrak{h}^*$, and
an upper-closed set $\mathcal{H} \subset {\rm Indep}(J_\lambda)$ such
that $\mathcal{H}^{\min} \subset {\rm Indep}(J_\lambda)$ consists of
pairwise orthogonal subsets, say $H_1, \dots, H_k$. Then the module
$\mathbb{M}(\lambda, \mathcal{H}) = \mathbb{M}(\lambda,
\mathcal{H}^{\min})$ has a BGG resolution
\begin{equation}\label{EBGGH}
0 \longrightarrow
M_k \overset{d_k}{\longrightarrow}
M_{k-1} \overset{d_{k-1}}{\longrightarrow}
\cdots \overset{d_2}{\longrightarrow}
M_1 \overset{d_1}{\longrightarrow}
M_0 \overset{d_0}{\longrightarrow}
\mathbb{M}(\lambda, \mathcal{H}) \to 0.
\end{equation}
Here, $M_k = M(\lambda_{H_1 \sqcup \cdots \sqcup H_k})$ and $M_0 =
M(\lambda)$ (so $\lambda_\emptyset = \lambda$), and more generally, $M_t$
is the direct sum of the Verma modules $M(\lambda_{H_{i_1} \sqcup \cdots
\sqcup H_{i_t}})$ over all $t$-tuples of indices $1 \leq i_1 < \cdots <
i_t \leq k$.

As a consequence, the character of $\mathbb{M}(\lambda,\mathcal{H})$ is
given by the ``Weyl character formula''
\begin{equation}\label{EcharMH}
\ch \mathbb{M}(\lambda,\mathcal{H}) = \sum_{S \subset \{ 1, \dots, k \}}
(-1)^{|S|} \ch M(\lambda_{\sqcup_{i \in S} H_i}).
\end{equation}
\end{theorem}

\begin{remark}\label{Rworking2}
In the spirit of Remark~\ref{Rworking1}, note that Section~\ref{SO}
worked over $\mathfrak{g}$ of finite type, while the results before it
were independent of which Kac--Moody quotient algebra
$\widetilde{\mathfrak{g}} \twoheadrightarrow \mathfrak{g}
\twoheadrightarrow \overline{\mathfrak{g}}$ (fixing a generalized Cartan
matrix) was used. The formulas in this section, while true for each
quotient $\mathfrak{g}$, do not necessarily give the same answers
\textit{across} varying $\mathfrak{g}$. This is because the character of
the Verma module can depend on $\mathfrak{g}$, whereas its weights do
not.
\end{remark}

\begin{example}\label{ExM00}
In the fundamental example in this regard, $V_{00} = M(0,0) / M(-2,-2)$
as in~\eqref{EM00}, the resolution~\eqref{EBGGH} specializes to
\[
\hspace*{1.8in}0 \to M(-2,-2) \to M(0,0) \to V_{00} \to 0.
\hspace*{1.8in}\qed
\]
\end{example}

As the proof of Theorem~\ref{TBGG1} reveals, the complex in \eqref{EBGGH}
is a BGG resolution~\cite{BGGres,kbg}, in which one is working with the
finite type Weyl group $W(\mathfrak{sl}_2^{\oplus k}) = (\Z /
2\Z)^{\oplus k}$. In fact the differentials $d_t$ are defined (below)
using the Bruhat order in this group. Moreover, if one considers the
words
\begin{equation}\label{EsimpleH}
{\bf s}_{H_j} := \prod_{h \in H_j} s_h, \qquad 1 \leq j \leq k
\end{equation}
as ``order $2$ Coxeter generators'', then the dot-action of the
``parabolic'' Weyl subgroup
\begin{equation}\label{EWH}
W_{\mathcal{H}} :=  \langle {\bf s}_{H_1}, \dots, {\bf s}_{H_k} \rangle
\simeq (\Z / 2\Z)^{\oplus k} \qquad (\text{with ``natural'' length
function } \ell_{\mathcal{H}} : W_{\mathcal{H}} \twoheadrightarrow \{ 0,
\dots, k \})
\end{equation}
on $\lambda$ yields precisely the highest weights $\lambda_H$ that
occur in the BGG resolution \eqref{EBGGH}. And indeed, the final
equation~\eqref{EWH} above, brings us back full circle to the first
equations in this paper -- the Weyl--Kac character formulas~\eqref{EWeyl}
\eqref{EVerma}, \eqref{EAtiyahBott} -- via their
$W_{\mathcal{H}}$-analogue:

\begin{cor}\label{CWCF}
Given Kac--Moody $\mathfrak{g}$, $\lambda \in \mathfrak{h}^*$, and an
upper-closed subset $\mathcal{H} \subset {\rm Indep}(J_\lambda)$, suppose
$\mathcal{H}^{\min} = \{ H_1, \dots, H_k \}$ consists of pairwise
orthogonal independent subsets of $J_\lambda$. Then
\begin{equation}\label{EWCFH}
\ch \mathbb{M}(\lambda, \mathcal{H}) = \sum_{w \in W_{\mathcal{H}} }
\frac{(-1)^{\ell_{\mathcal{H}}(w)} e^{w \bullet \lambda}}{\prod_{\alpha
\in \Delta^+} (1 - e^{-\alpha})^{\dim \mathfrak{g}_\alpha}}.
\end{equation}
\end{cor}

\begin{proof}[Proof of Theorem~\ref{TBGG1}]
Briefly, \eqref{EBGGH} is the BGG resolution in the restricted (simpler)
case of $\mathfrak{g} = \mathfrak{sl}_2^{\oplus k}$, and so is the Koszul
resolution of $R / (Ry_1 + \cdots + Ry_k)$ for $R = \mathbb{C}[y_1,
\dots, y_k]$ -- where $y_j := \bff_{H_j}\ \forall j$ -- subsequently
tensored with the free $R$-module $M = U \mathfrak{n}^-$.
(Notice, this case is therefore easier than the proof for arbitrary Weyl
groups in \cite{BGGres,kbg}, given the simpler underlying Weyl group).

We give details for the interested reader. Let $W_{\mathcal{H}} = (\Z / 2
\Z)^{\oplus k} \overset{\psi}{\simeq} 2^{\{ 1, \dots, k \}}$, and via
$\psi$ write
\[
W_{\mathcal{H}} = \{ w_J := {\textstyle \prod_{j \in J}} {\bf s}_{H_j} \
| \ J \subset \{ 1, \dots, k \} \}, \qquad w_J w_K = w_{J \Delta K}.
\]

Next, we note the unique (up to scalar) embeddings $\iota(J',J)$ of the
various Verma modules in the resolution above, according to the poset
structure of the subsets of $\{ 1, \dots, k \}$ under inclusion. Namely,
if $J \subset J' \subset \{ 1, \dots, k \}$, and if we denote $H_J :=
\sqcup_{j \in J} H_j$, then
\[
M(\lambda_{H_{J'}}) \hookrightarrow M(\lambda_{H_J}) \hookrightarrow
M(\lambda)
\]
with $\lambda_H$ in~\eqref{EfH}. Concretely, choosing a highest weight
vector $m_\lambda \in M(\lambda)_\lambda$, these embeddings are:
\begin{equation}\label{Eembedding}
\iota(J', J) : U \mathfrak{g} \left( \bff_{H_{J'} \setminus H_J} \cdot
\bff_{H_J} \cdot m_\lambda \right) \hookrightarrow U \mathfrak{g} \left(
\bff_{H_J} \cdot m_\lambda \right) = U \mathfrak{g} ( m_{\lambda_{H_J}} )
\hookrightarrow U \mathfrak{g} ( m_\lambda ),
\end{equation}
where we \textbf{define} $m_{\lambda_{H_J}} := \bff_{H_J} \cdot
m_\lambda$ for all $J \subset \{ 1, \dots, k \}$. (Thus, $w_J \bullet
\lambda_{H_K} = \lambda_{H_{J \Delta K}}$ for $J,K \subset \{ 1, \dots, k
\}$.)
Moreover, the above Verma submodules have the expected intersection:
\begin{equation}\label{Eintersect}
J,K \subset \{ 1, \dots, k \} \quad \implies \quad
M(\lambda_{H_J}) \cap M(\lambda_{H_K}) = 
M(\lambda_{H_{J \cup K}})
\end{equation}
as submodules of $M(\lambda)$, by using weight space
decompositions and the PBW theorem. The elements
$y_j := \bff_{H_j}$
commute pairwise, and will be used to define -- via the formulas as in
the Koszul resolution for $R / (Ry_1 + \cdots + Ry_k)$ above -- the
differentials $d_t$, or more precisely, their coordinates. Namely,
\begin{equation}\label{Ed1}
d_1 \left( \sum_{j=1}^k X_j m_{\lambda_{H_j}} \right) := \sum_{j=1}^k X_j
\bff_{H_j} \cdot m_\lambda = \sum_{j=1}^k X_j y_j \cdot m_\lambda, \quad
X_j \in U \mathfrak{n}^-.
\end{equation}
Next if $t>1$, then $d_t(J',J): M(\lambda_{H_{J'}}) \to
M(\lambda_{H_J})$ is zero unless $J \subset J'$ with $t = |J'| = |J|+1$,
in which case
\begin{align}\label{Edifferential}
&\ J' = \{ i_1, \dots, i_t \ : \ 1 \leq i_1 < \cdots < i_t \leq k \}, \ \
X \in U\mathfrak{n}^-\notag\\
\implies &\ d_t \left( X m_{\lambda_{H_{J'}} } \right) := X \sum_{j=1}^t
(-1)^{j-1} \bff_{H_{i_j}} m_{\lambda_{H_{J' \setminus \{ i_j \}} }}
= X \sum_{j=1}^t (-1)^{j-1} y_{i_j} m_{\lambda_{H_{J' \setminus \{ i_j
\}} }}.
\end{align}

Observe that \eqref{Ed1}, \eqref{Edifferential} are precisely the
formulas for the differentials in the Koszul complex
\begin{equation}\label{EKoszulR}
0 \longrightarrow R \overset{d_k}{\longrightarrow} R^{\binom{k}{k-1}}
\overset{d_{k-1}}{\longrightarrow} \cdots 
\overset{d_3}{\longrightarrow} R^{\binom{k}{2}}
\overset{d_2}{\longrightarrow} R^k
\overset{d_1}{\longrightarrow} R
%\overset{d_0}{\longrightarrow} \frac{R}{Ry_1 + \cdots + Ry_k}
\longrightarrow 0,
\end{equation}
with $R = \mathbb{C}[y_1, \dots, y_k]$, and under the identification
wherein the free module $R^{\binom{k}{t}}$ has basis
\[
\left\{ m_{\lambda_{H_{J'}} } : J' = \{ i_1, \dots, i_t \ : \ 1 \leq i_1
< \cdots < i_t \leq k \} \right\}.
\]
Also note the same intersection property as~\eqref{Eintersect}:
\[
J,K \subset \{ 1, \dots, k \} \quad \implies \quad
{\textstyle R \prod_{t \in J} y_t \ \ \bigcap \ \ 
R \prod_{t \in K} y_t \ \ = \ \ 
R \prod_{t \in J \cup K} y_t.}
\]

Finally, transfer the Koszul complex from~\eqref{EKoszulR}
to~\eqref{EBGGH}, in the usual manner. Define the $R$-module $M :=
U\mathfrak{n}^-$; by the PBW theorem this is free over the polynomial
algebra $\mathbb{C}[ \{ f_h : h \in \sqcup_j H_j \} ]$, which is in turn
free over $R = \mathbb{C}[ \{ \bff_{H_j} : 1 \leq j \leq k \}]$.
Thus, one tensors the resolution~\eqref{EKoszulR} with $M$ to
obtain the BGG complex in~\eqref{EBGGH}, with the specified differential
maps. (Strictly speaking, one obtains~\eqref{EKoszulR} with $R$ replaced
by $M = U \mathfrak{n}^-$, and this is isomorphic as free $U
\mathfrak{n}^-$-modules to the complex in~\eqref{EBGGH}.) The
$R$-freeness implies that~\eqref{EBGGH} is indeed the desired resolution
for $\mathbb{M}(\lambda, \mathcal{H})$.
\end{proof}

\subsection{Higher order Weyl group action on characters}

As is well known, in a parabolic category $\mathcal{O}^{\mathfrak{p}_J}$,
the character of any object is $W_J$-invariant -- or as we now understand
in the language of holes, invariant under the minimal hole reflections
$s_j, j \in J$. We now explain a sense in which this phenomenon
generalizes to all holes -- when applied to the higher order Verma
modules $\mathbb{M}(\lambda, \mathcal{H})$.

We begin by first explaining not the invariance, but the \textit{partial
action} of the Weyl group on the weights of a highest weight module over
Kac--Moody $\mathfrak{g}$. Consider once again the basic example $V_{00}
= M(0,0) / M(-2,-2)$ over $\mathfrak{g} = \mathfrak{sl}_2 \oplus
\mathfrak{sl}_2$. This is a length $3$ highest weight module, and the
unique hole here is $H = \{ 1, 2 \}$. Now $s_1 s_2$ fixes the weight $0$,
but takes every other weight of $V_{00}$ to a non-weight. What does hold
is that $W$ stabilizes the weights of $L(0,0)$; $W_{\{ 1 \}} = \{ 1, s_1
\}$ stabilizes the weights of $L(s_2 \bullet 0) = M(s_2 \bullet 0, \{ 1
\})$; and similarly for $L(s_1 \bullet 0)$. Viewed differently, $\wt
V_{00} = \wt M(\lambda, \{ 1 \}) \cup \wt M(\lambda, \{ 2 \})$, and these
are stable under the action of $W_{\{ 1 \}}$ and $W_{\{ 2 \}}$,
respectively. (Hence their intersection $\wt L(0,0)$ is $W$-stable.)

The situation is similar in general, via Theorem~\ref{T1}. Let
$\mathfrak{g}$ be a Kac--Moody algebra, and $M$ an object in
$\mathcal{O}$. Then $M$ has a finite filtration by highest weight
modules, say $M(\lambda_i) \twoheadrightarrow V_i$, and so
\[
\wt M \ = \ \bigcup_{i \geq 0} \wt V_i \ = \ \bigcup_{i \geq 0} \ \
\bigcup_{J \subset J_{\lambda_i}\, :\, J \cap H \neq \emptyset\; \forall
H \in \mathcal{H}_{V_i}} \wt M(\lambda_i,J)
\]
by Theorem~\ref{T1}. The partial action now says that part of $W$ acts on
part of $\wt M$. Namely, given $\mu \in \wt M$, the orbit $W_J(\mu)
\subset M$ still, whenever $\mu \in \wt M(\lambda_i, J)$ for some
$(i,J)$.

As a special case, if $M \in \mathcal{O}^{\mathfrak{p}_{J'}}$ for some
$J' \subset I$, then in the above union every $J$ satisfies: $J \cap \{
j' \} \neq \emptyset$ for $j' \in J'$, and so $J \supseteq J'$ -- which
implies that $\wt M$ is $W_{J'}$-stable. In contrast, there need not be
any global symmetries in the higher order case. E.g.\ for $\mathfrak{g} =
\mathfrak{sl}_2 \oplus \mathfrak{sl}_2$ one has $M := L(s_1 \bullet 0)
\oplus L(s_2 \bullet 0) \in \mathcal{O}^{\{ \{ 1, 2 \} \}}$ -- but $\wt
M$ is stable only under $W_{\{ 1 \}} \cap W_{\{ 2 \}}$, i.e.\ the
identity.\medskip

Having discussed weights, we turn to characters. We return to the opening
paragraph of this subsection, and attempt to generalize it to higher
order holes. Begin once again with the above example $V_{00}$ over
$\mathfrak{sl}_2^{\oplus 2}$.
In this case, the unique hole is $\{ 1, 2 \}$, and one seeks to
understand if (and how) $s_1 s_2$ preserves the character of $V_{00}$.
The immediate approach would be to evaluate $s_1 s_2 \left( e^{(0,0)} +
\sum_{n>0} (e^{-n \alpha_1} + e^{-n \alpha_2}) \right)$, and it is easy
to check this does not leave the character unchanged. Instead, one needs
to rewrite the character as
\[
\ch V_{00} = \frac{e^{(0,0)}}{(1 - e^{-\alpha_1})(1 - e^{-\alpha_2})}
- \frac{e^{(-2,-2)}}{(1 - e^{-\alpha_1})(1 - e^{-\alpha_2})}.
\]
Now acting on both numerators and both denominators by $s_1 s_2$ leaves
this expression unchanged. This happens due to the ``correct'' way of
expanding both ratios (after applying $s_1 s_2$) -- via their ``highest
weight expansions''. As an illustration, if $\alpha$ is a positive root
in a Kac--Moody algebra, then $(1 - e^{-\alpha})^{-1}$ equals $1 +
e^{-\alpha} + e^{-2\alpha} + \cdots$, whereas $(1 - e^\alpha)^{-1}$ is
expanded differently:
\[
\frac{1}{1 - e^\alpha} = \frac{e^{-\alpha}}{e^{-\alpha} - 1} = -
e^{-\alpha} ( 1 + e^{-\alpha} + e^{-2\alpha} + \cdots ).
\]
This is originally due to Brion \cite{Brion} (for rational polytopes) and
Khovanskii--Pukhlikov \cite{KP}, Lawrence \cite{Lawrence}, and Varchenko
\cite{Varchenko} -- see Postnikov \cite{Postnikov} for deformation
arguments for generic nodes.

The invariance of the character of the higher (2nd) order Verma module
$V_{00}$ under the subgroup $\{ 1, s_1 s_2 \}$ is a ``higher order''
version of the $W_J$-invariance of the character of the module
$M(\lambda, J)$. We now extend the former phenomenon and parallel the
latter, in the situation discussed above:

\begin{prop}\label{PWchar}
Let $\mathfrak{g}$, $\lambda \in \mathfrak{h}^*$, and $\mathcal{H}^{\min}
= \{ H_1, \dots, H_k \} \subset {\rm Indep}(J_\lambda)$ be as in
Theorem~\ref{TBGG1} -- as also the subgroup $W_\mathcal{H} \simeq
(\Z/2\Z)^k$ of $W$. Then,
\begin{equation}\label{EWchar2}
w(\ch \mathbb{M}(\lambda, \mathcal{H})) = (-1)^{\ell(w) -
\ell_\mathcal{H}(w)} \ch \mathbb{M}(\lambda, \mathcal{H}), \quad \forall
w \in W_\mathcal{H}.
\end{equation}
\end{prop}

\begin{proof}
With notation as in~\eqref{EWH}, write $W_\mathcal{H} = \{ w_J = \prod_{j
\in J} {\bf s}_{H_j} \ | \ J \subset \{ 1, \dots, k \} \}$. Now compute,
starting from the Weyl--Kac type character formula~\eqref{EWCFH}:
\begin{align*}
w_K (\ch \mathbb{M}(\lambda, \mathcal{H})) = &\ \sum_{J \subset \{ 1,
\dots, k \}} w_K \left( \frac{(-1)^{|J|} e^{w_J \bullet
\lambda}}{\prod_{\alpha \in \Delta^+} (1 - e^{-\alpha})^{\dim
\mathfrak{g}_\alpha}} \right)\\
= &\ (-1)^{\ell(w_K)} \sum_{J \subset \{ 1, \dots, k \}} \frac{(-1)^{|J|}
e^{(w_K w_J) \bullet \lambda}}{\prod_{\alpha \in \Delta^+} (1 -
e^{-\alpha})^{\dim \mathfrak{g}_\alpha}} = (-1)^{\ell(w_K)} \cdot
(-1)^{|K|} \ch \mathbb{M}(\lambda, \mathcal{H}),
\end{align*}
since $(-1)^{|K \Delta J|} = (-1)^{|K|} (-1)^{|J|}$. This concludes the
proof.
\end{proof}

\begin{remark}
Suppose $\mathfrak{g}$ is semisimple, and $\mathcal{H} \subset {\rm
Indep}(I)$ is such that the minimal sets in $\mathcal{H}^{\min}$ are
pairwise orthogonal. 
We explain a sense in which Proposition~\ref{PWchar} extends to
$\mathcal{O}^{\mathcal{H}}$ the $W_J$-invariance of $\ch M$ for all
objects $M$ in the parabolic category $\mathcal{O}^{\mathfrak{p}_J}$.
Indeed, via a triangular change of bases, the $W_J$-invariance of all
characters in $\mathcal{O}^{\mathfrak{p}_J}$, is $K_0$-equivalent to that
of $\ch M(\lambda, J)$ for all $\lambda$ with $L(\lambda) \in
\mathcal{O}^{\mathfrak{p}_J}$. The higher order analogue of this is given
by Proposition~\ref{PWchar} (via Proposition~\ref{Psimple}(2)):

\textit{If $L(\lambda) \in \mathcal{O}^\mathcal{H}$, then the character
of $\mathbb{M}(\lambda, \mathcal{H})$ satisfies~\eqref{EWchar2} with
$W_{\mathcal{H}}$ replaced by $W_{\mathcal{H}'_\lambda}$,}

\noindent since if $\mathcal{H} = \mathcal{H}_J$ and $L(\lambda) \in
\mathcal{O}^\mathcal{H} = \mathcal{O}^{\mathfrak{p}_J}$, then
$\mathcal{H}'_\lambda = \mathcal{H}_J$ and $\ell_{\mathcal{H}_J} = \ell$.
\end{remark}

\subsection{Setting 2: Pairwise orthogonal integrable roots, and the
parabolic Weyl semigroup}\label{Subsection Setting 2}

We next prove BGG resolutions and character formulas for the modules
$\mathbb{M}(\lambda, \mathcal{H})$ in another setting: for
$\mathfrak{g}_{J_\lambda} = \mathfrak{sl}_2^{\oplus n}$ for some $n \geq
1$. In this case, the BGG resolution turns out to involve a
\textit{semigroup} action of $W_\mathcal{H} \cong (\Z / 2 \Z)^{\oplus k}$
(as sets, with $k = |\mathcal{H}^{\min}|$) on $\lambda$, as we first
explain for $k=2$:

\begin{theorem}\label{T2holes}
Suppose $\mathfrak{g} = \mathfrak{sl}_2^{\oplus n}$ for some $n \geq 1$,
$\lambda \in \mathfrak{h}^*$, and $\mathcal{H}^{\min} = \{ H_1, H_2 \}
\subset {\rm Indep}(J_\lambda) = 2^{J_\lambda}$.
Using the notation of $\lambda_H, \bff_H$ as in~\eqref{EfH}, the
module $\mathbb{M}(\lambda, \mathcal{H}^{\min})$ has a BGG resolution
\begin{equation}\label{EBGG2holes}
0 \longrightarrow M( \lambda_{H_1 \cup H_2} )
\overset{d_2}{\longrightarrow} M( \lambda_{H_1} ) \oplus M( \lambda_{H_2}
) \overset{d_1}{\longrightarrow} M( \lambda ) \
\overset{d_0}{\longrightarrow} \mathbb{M}(\lambda, \mathcal{H}^{\min})
\to 0,
\end{equation}
where the ``Koszul-type'' differentials are given by
\begin{align*}
d_1(X_1 m_{\lambda_{H_1}}, X_2 m_{\lambda_{H_2}}) := &\ \left( X_1
\bff_{H_1} + X_2 \bff_{H_2} \right) m_\lambda,\\
d_2(X m_{\lambda_{H_1 \cup H_2}}) := &\
(-X \bff_{H_2 \setminus H_1} m_{\lambda_{H_1}}, X \bff_{H_1 \setminus
H_2} m_{\lambda_{H_2}}), \qquad X, X_1, X_2 \in U\mathfrak{n}^-.
\end{align*}
\end{theorem}

This result can be verified by hand, and gives $\ch \mathbb{M}(\lambda,
\mathcal{H}^{\min})$ as the Euler characteristic. It also yields the same
Weyl character formula as previously, involving $w \bullet \lambda$ for
$w$ in the \textit{set}
\[
W_{\mathcal{H}} = W_{\mathcal{H}^{\min}} := \{ 1, \quad
w_{\{ 1 \}} = {\bf s}_{H_1}, \quad w_{\{ 2 \}} = {\bf s}_{H_2}, \quad
w_\circ = w_{\{ 1, 2\}} = {\bf s}_{H_1 \cup H_2} \}.
\]
Define the associated length function $\ell_{\mathcal{H}} =
\ell_{\mathcal{H}^{\min}}$, which sends $1 \mapsto 0$, $w_{\{ j \}}
\mapsto 1$, and $w_{\{ 1, 2 \}} \mapsto 2$.

\begin{cor}
Notation as above. Then:
\begin{equation}\label{EWCF2holes}
\ch \mathbb{M}(\lambda, \mathcal{H}^{\min}) = \sum_{w \in W_{\mathcal{H}}
} \frac{(-1)^{\ell_{\mathcal{H}}(w)} e^{w \bullet \lambda}}{\prod_{\alpha
\in \Delta^+} (1 - e^{-\alpha})^{\dim \mathfrak{g}_\alpha}}.
\end{equation}
\end{cor}

\begin{remark}\label{Rsemigroup}
While the Weyl character formula is unchanged from \eqref{EWCFH},
the action of $W_{\mathcal{H}}$ on the \textit{orbit} of $\lambda$ is now
different than in the previous case \eqref{EBGGH}. (This action is
useful in understanding the differential maps, which differ even for
$\mathcal{H}^{\min} = \{ H_1, H_2 \}$ over $\mathfrak{g} =
\mathfrak{sl}_2^{\oplus n}$.) Indeed, now we use
\[
w_J \cdot' w_K := w_{J \cup K}, \qquad
w_J \bullet' \lambda_{H_K} := \lambda_{H_{J \cup K}} = \lambda_{H_J \cup
H_K}, \qquad \forall J,K \subset \{ 1, 2 \}.
\]
Thus, $(W_{\mathcal{H}}, \cdot')$ is what we term the \textit{parabolic
Weyl semigroup} in this situation. Moreover, the map $\bullet'$ is a
semigroup action of $(W_{\mathcal{H}}, \cdot')$ on the orbit $\{ \lambda,
\lambda_{H_1}, \lambda_{H_2}, \lambda_{H_1 \cup H_2} \}$.
Of course, in the character formula \eqref{EWCF2holes}, the only element
on which $W_{\mathcal{H}}$ acts is $\lambda$ itself, which is dominant
integral for $\Delta_{H_1 \cup H_2}$, and so in those equations $\bullet'
= \bullet$.
\end{remark}

\begin{example}
For a concrete working example, the reader can consider e.g.\
$\mathfrak{g} = \mathfrak{sl}_2^{\oplus 3}$, $I = \{ 1, 2, 3 \}$,
$\lambda \in P^+$, and $\mathcal{H}^{\min} = \{ \{ 1, 2 \}, \{ 2, 3 \}
\}$. The overlap between elements of $\mathcal{H}^{\min}$ is what leads
to the parabolic Weyl semigroup $W_{\mathcal{H}} \cong (\Z/2\Z)^{\oplus
2}$ here (as sets), and its orbit is $W_{\mathcal{H}} \bullet' \lambda$.
\end{example}

We now write down the analogous picture -- over $\mathfrak{g}$ with
orthogonal integrable roots for $\lambda \in \mathfrak{h}^*$, i.e.,
$\mathfrak{g}_{J_\lambda} = \mathfrak{sl}_2^{\oplus n}$ -- with
$\mathcal{H}^{\min} = \{ H_1, \dots, H_k \} \subset {\rm
Indep}(J_\lambda)$. Our BGG-type resolution again turns out to yield the
same character formula \eqref{EWCF2holes}; however, one now requires
alternate notation from the earlier one, as is revealed by the simple
case
\begin{equation}\label{E122331}
\mathfrak{g} = \mathfrak{sl}_2^{\oplus 3}, \quad
\mathcal{H}^{\min} = \{ H_1 = \{ 1, 2 \}, \ H_2 = \{ 1, 3 \}, \ H_3 = \{
2, 3 \} \}.
\end{equation}
In this case, the terms $M_1 \to M_0 = M(\lambda) \to \mathbb{M}(\lambda,
\mathcal{H}^{\min})$ are as above, while in the above notation,
\[
\lambda_{H_{\{ 1, 2 \}} } = \lambda_{H_{\{ 1, 3 \}} } = \lambda_{H_{\{ 2,
3 \}} } = \lambda_{H_{\{ 1, 2, 3 \}} } = s_1 s_2 s_3 \bullet \lambda.
\]

To distinguish between the corresponding four Verma modules occurring in
$M_2$ and $M_3$, define
\begin{equation}\label{ElambdaJ}
\lambda(J) := \lambda_{H_J} = \lambda - \sum_{h \in \cup_{j \in J} H_j}
(\langle \lambda, \alpha_h^\vee \rangle + 1) \alpha_h \in \lambda -
\Z_{\geq 0} \Pi_{J_\lambda}, \qquad J \subset \{ 1, \dots, k \}.
\end{equation}
Choosing a maximal vector $m_\lambda \in M(\lambda)_\lambda$, the modules
$M(\lambda(J))$ then again embed into one another via the maps
$\iota(J',J)$ for $J \subset J' \subset \{ 1, \dots, k \}$ as
in~\eqref{Eembedding}, and via this embedding into $M(\lambda)$, have
intersections as in~\eqref{Eintersect}. Thus, also fix maximal vectors
$m_{\lambda(J)} \in M(\lambda(J))_{\lambda(J)}$ such that
$\iota(J',J)$ sends $m_{\lambda(J')}$ to $\bff_{H_{J'} \setminus H_J}
m_{\lambda(J)}$ for all $J \subset J' \subset \{ 1, \dots, k \}$. Now
define the modules
\[
M_t := \bigoplus_{J \subset \{ 1, \dots, k \},\ |J| = t} M(\lambda(J)),
\qquad 0 \leq t \leq k.
\]
Also define the differential $d_1 : M_1 \to M_0$ via
\[
d_1 \left( \sum_{j=1}^k X_j m_{\lambda(\{ j \})} \right) := \sum_{j=1}^k
X_j \bff_{H_j} m_\lambda, \qquad X_j \in U\mathfrak{n}^-
\]
and the differential $d_t, \ t > 1$ via its coordinates. Namely,
$d_t(J',J) : M(\lambda(J')) \to M(\lambda(J))$ is zero unless $J \subset
J'$ with $t = |J'| = |J|+1$, in which case
\begin{align}\label{Edifferential2}
&\ J' = \{ i_1, \dots, i_t \ : \ 1 \leq i_1 < \cdots < i_t \leq k \}, \ \
X \in U\mathfrak{n}^-\notag\\
\implies &\ d_t \left( X m_{\lambda(J')} \right) := X \sum_{j=1}^t
(-1)^{j-1} \bff_{H_{i_j} \setminus H_{J' \setminus \{ i_j \}} }
m_{\lambda(J' \setminus \{ i_j \})}.
\end{align}

\begin{remark}
The modules $M_t$ and differentials $d_t$ indeed specialize to their
counterparts in Theorem~\ref{TBGG1} as well as in Theorem~\ref{T2holes},
earlier in this section.
\end{remark}

As above, we now show this yields a resolution; notice that in
conjunction with Theorem~\ref{TBGG1} and Proposition~\ref{PWchar}, this
resolution would also imply our main theorem~\ref{T5}.

\begin{theorem}\label{TBGG2}
Fix Kac--Moody $\mathfrak{g}$ and $\lambda \in \mathfrak{h}^*$ such that
the nodes $J_\lambda$ have no edges, and let $\mathcal{H}^{\min} = \{
H_1, \dots, H_k \} \subset {\rm Indep}(J_\lambda)$. With the modules
$M_t$ and differentials $d_t$ as above, the complex
\begin{equation}\label{EBGGH2}
0 \longrightarrow
M_k \overset{d_k}{\longrightarrow}
M_{k-1} \overset{d_{k-1}}{\longrightarrow}
\cdots \overset{d_2}{\longrightarrow}
M_1 \overset{d_1}{\longrightarrow}
M_0 \overset{d_0}{\longrightarrow}
\mathbb{M}(\lambda, \mathcal{H}) \to 0
\end{equation}
is a free $U \mathfrak{n}^-$-resolution of $\mathbb{M}(\lambda,
\mathcal{H})$. As a consequence, with $\lambda(J)$ as
in~\eqref{ElambdaJ},
\begin{equation}\label{EcharMlH}
\ch \mathbb{M}(\lambda,\mathcal{H}) = \sum_{J \subset \{ 1, \dots, k \}}
(-1)^{|J|} \ch M(\lambda(J)).
\end{equation}
\end{theorem}

\begin{proof}
First consider the variant of this complex over $\mathfrak{g}_{J_\lambda}
\simeq \mathfrak{sl}_2^{\oplus n}$.
Writing $\bff_{H_{i_j} \setminus H_{J' \setminus \{ i_j \}} }$ as
$\displaystyle \bff_{H_{i_j} \setminus H_{J' \setminus \{ i_j \}} } =
\frac{ {\rm lcm} \{ \bff_{H_i} : i \in J' \}}{ {\rm lcm} \{ \bff_{H_i} :
i \in J', i \neq i_j \}},$
\eqref{EBGGH2} is precisely the Taylor resolution \cite{Taylor} (see also
\cite{dCEP}) for $\mathbb{M}(\lambda, \mathcal{H})$ over the commutative
ring -- in fact UFD -- $R := \mathbb{C}[\{ f_i : i \in J_\lambda \}]$.
Now the result follows over $U \mathfrak{g}$ by tensoring this Taylor
resolution with the $R$-module $U \mathfrak{n}^-$, which is $R$-free by
the PBW theorem.
\end{proof}

In particular, the Weyl character formula again follows from this
resolution, as for $k=2$. Let
\[
W_{\mathcal{H}} = \{ (w_J, J) \ | \ J \subset \{ 1, \dots, k \} \},
\qquad w_J := {\bf s}_{\bigcup_{j \in J} H_j}, \qquad (w_J, J) \cdot'
(w_K, K) := (w_{J \cup K}, J \cup K)
\]
for $J,K \subset \{ 1, \dots, k \}$ denote the \textbf{parabolic Weyl
semigroup} in this situation. Define the length of $(w_J,J)$ to be
$\ell_{\mathcal{H}}((w_J,J)) := |J|$. Then $(W_{\mathcal{H}}, \cdot')$
acts on the orbit $\{ \lambda_{H_K} : K \subset \{ 1, \dots, k \} \}$
of $\lambda_{H_\emptyset} = \lambda$ via: $(w_J, J) \bullet'
\lambda_{H_K} := \lambda_{H_{J \cup K}}$. On $\lambda$, this is the
dot-action: $(w_J, J) \bullet' \lambda = w_J \bullet \lambda\ \forall J$.

\begin{cor}
The resolution \eqref{EBGGH2} once again implies, for arbitrary
$\mathcal{H}^{\min}$ of size $k \geq 1$:
\begin{equation}\label{EWCF2holes2}
\ch \mathbb{M}(\lambda, \mathcal{H}) = \sum_{w = (w_J, J) \in
W_{\mathcal{H}} } \frac{(-1)^{\ell_{\mathcal{H}}(w)} e^{w \bullet
\lambda}}{\prod_{\alpha \in \Delta^+} (1 - e^{-\alpha})^{\dim
\mathfrak{g}_\alpha}}.
\end{equation}
\end{cor}
\subsection{Resolutions over dihedral groups}

Having obtained BGG resolutions and Weyl character formulas in the above
two settings using {\color{black}$W_{\mathcal{H}}$}, we briefly discuss {\color{black}}
potentially simple {\color{black}cases of settings (3) and (4), before we proceed to Theorem \ref{Theorem character of second order Verma via parabolic Vermas} and Proposition \ref{Prop A3 second order holes character}} in the next two subsections. 
{\color{black}Our observations here also serve as warm up to the proofs therein.
 First, a} small lemma on Coxeter generators in finite
Weyl groups.
As used in Corollary~\ref{CWCF}, if two sets $H_1, H_2
\subset J_\lambda$ of independent nodes are orthogonal, then $({\bf
s}_{H_1} {\bf s}_{H_2})^2 = 1$ in $W$. The next lemma computes this order
in the more general case when $H_1, H_2$ are merely pairwise disjoint.
Even more generally:

\begin{lemma}\label{L2holes}
Suppose $\mathfrak{g}$ is of finite type, and $H_1, \dots, H_k \in {\rm
Indep}(I)$ are pairwise disjoint. Let $H_1 \sqcup \cdots \sqcup H_k$ have
connected Dynkin components $J_1, \dots, J_l$. Then the product ${\bf
s}_{H_1} \cdots {\bf s}_{H_k}$ has order precisely ${\rm lcm}(c_1, \dots,
c_l)$, where $c_t$ is the Coxeter number of the parabolic Weyl subgroup
$W_{J_t}$.
\end{lemma}

\begin{proof}
For $1 \leq i \leq k$ and $1 \leq t \leq l$, define $J_{it} := H_i \cap
J_t$. It is clear that $\prod_{i=1}^k {\bf s}_{H_i} = \prod_{t=1}^l {\bf
s}_t$, where ${\bf s}_t := {\bf s}_{J_{1t}} \cdots {\bf s}_{J_{kt}}$ are
pairwise commuting. Hence the order of ${\bf s}_{H_1} \cdots {\bf
s}_{H_k}$ is the ${\rm lcm}$ of the orders of the ${\bf s}_t$. But each
${\bf s}_t$ is a Coxeter element for the Weyl group on $\sqcup_i J_{it} =
J_t$, hence has order $c_t$.
\end{proof}

\begin{remark}
The assumption of finite type is needed in Lemma~\ref{L2holes}, because
if $W$ is an irreducible, infinite Coxeter group then its Coxeter
elements have infinite order \cite{Howlett,Speyer}. 
\end{remark}

As an ``application'', let $\mathfrak{g}$ be of finite type and
$\mathcal{H}^{\min} = \{ H_1, H_2 \}$, with $H_1, H_2 \in {\rm
Indep}(J_\lambda)$ disjoint subsets of nodes. The corresponding subgroup
$W_{\mathcal{H}} = \langle {\bf s}_{H_1}, {\bf s}_{H_2} \rangle$ is then
dihedral in these two Coxeter generators, with the longest word in them
given by ${\bf w}_\circ$, say. If ${\bf s}_{H_1} {\bf s}_{H_2}$ has order
$m \geq 2$ (computed via Lemma~\ref{L2holes}), then {\color{black} one might expect to have
BGG type resolution \eqref{Edihedral} below, where-in the Verma summands' top weights are the dot conjugates under all the words on $\{s_{H_1}, s_{H_2}\}$.} 
\begin{align}\label{Edihedral}
0 \longrightarrow M({\bf w}_\circ \bullet \lambda)
\overset{d_m}{\longrightarrow}
M({\bf s}_{H_1} {\bf w}_\circ \bullet \lambda) \oplus &\ M({\bf s}_{H_2}
{\bf w}_\circ \bullet \lambda)
\overset{d_{m-1}}{\longrightarrow} \cdots
{\color{black}\overset{d_3}{\longrightarrow} M({\bf s}_{H_1} {\bf s}_{H_2} \bullet 0)
\oplus M({\bf s}_{H_2} {\bf s}_{H_1} \bullet 0)}\\
\overset{d_2}{\longrightarrow} &\
M({\bf s}_{H_1} \bullet \lambda) \oplus M({\bf s}_{H_2} \bullet \lambda)
\overset{d_1}{\longrightarrow}
M(\lambda) \overset{d_0}{\longrightarrow}
\mathbb{M}(\lambda, \mathcal{H}) \longrightarrow 0 \ {\large \bf ?}\notag
\end{align}
{\color{black} If this is true, then we immediately have} a Weyl character
formula akin to~\eqref{EWCF2holes2}.
% For instance, $d_1 : (X_1 m_{{\bf s}_{H_1} \bullet 0}, X_2 m_{{\bf s}_{H_2} \bullet 0}) \mapsto (X_1 f_1 f_3 + X_2 f_2) m_0$ for $X_1, X_2 \in U\mathfrak{n}^-$ with some fixed choice of highest weight vectors $m$, and similarly, one coordinate of $d_2$ is given by \[ d_2(X m_{{\bf s}_{H_1} {\bf s}_{H_2} \bullet 0}) := (X (f_2 f_1 f_3 + 2 f_{12} f_3 - 2 f_{23} f_1 + 2 f_{123}) m_{{\bf s}_{H_1} \bullet 0}, - X f_1^2 f_3^2 m_{{\bf s}_{H_2} \bullet 0}), \] where $f_{12} = [f_1, f_2], f_{23} = [f_2, f_3], f_{123}$ are the non-simple negative root vectors, and $X \in U \mathfrak{n}^-$.
\begin{remark}
{\color{black}However this speculated resolution is not true, even in the simplest case of
\[
\mathfrak{g} = \mathfrak{sl}_4(\mathbb{C});\  \lambda = 0;\ \mathcal{H} = \{ H_1 = \{ 1, 3 \}, H_2 = \{ 2 \} \};\  m=4\  \text{(by Lemma~\ref{L2holes})};\ {\bf w}_\circ = {\bf s}_{H_1} {\bf s}_{H_2} {\bf s}_{H_1} {\bf
s}_{H_2}.\]
This is due to the linear independence of the 16 exponentials in the actual character numerator found by Proposition \ref{Prop A3 second order holes character} in this case, with the 8 exponentials given by \eqref{Edihedral} at $m=4$.
We note before moving-on from type $A_3$ that} for all other subsets $\mathcal{H}' \subset {\rm Indep}(I)$ for
$\mathfrak{g} = \mathfrak{sl}_4, \lambda =0$, a BGG resolution of
$\mathbb{M}(0, \mathcal{H}')$ is {\color{black}classically known!

Here is another illuminating counter example to \eqref{Edihedral} from setting (3), that inspired Theorem~\ref{Theorem character of second order Verma via parabolic Vermas}.
\[
I=\{1,2,3\};\ \mathfrak{g}=\mathfrak{g}_{\{1\}}\oplus \mathfrak{g}_{\{2,3\}}\text{ with }\mathfrak{g}_{\{2,3\}}=\mathfrak{sl}_3(\mathbb{C});\ \mathcal{H}^{\min}= \big\{ \{2\}, \{1,3\}\big\}; \ {\bf w}_{\circ}= {\bf s}_{\{2\}}{\bf s}_{\{1,3\}}{\bf s}_{\{2\}}.
\]
For a resolution of $\mathbb{M}\big(0, \mathcal{H}\big)$, by looking at the top weights of Vermas across the complex in \eqref{Edihedral} written below, one arrives at the boxed part in the complex; and the other half is outlandish.
Theorem \ref{Theorem character of second order Verma via parabolic Vermas} proves the boxed sub-resolution to be the correct one.
\begin{align*}
& \boxed{ \begin{matrix} M\big((s_2) \cdot (s_1s_3)\cdot (s_2)\bullet 0) \\ \oplus \end{matrix}
   \longrightarrow
\oplus \begin{matrix}
    M\big((s_1s_3) \cdot (s_2)\bullet 0)\\
     M((s_2)\cdot (s_1s_3)\bullet 0)
\end{matrix}  \longrightarrow
\oplus \begin{matrix}
    M(s_2\bullet 0)\\
     M(s_1s_3\bullet 0)
\end{matrix}  \longrightarrow M(0)   \longrightarrow \mathbb{M}(0, \mathcal{H}) \longrightarrow 0}\\
& \qquad M((s_3s_2s_3)\bullet 0)\\   
&\qquad \qquad\quad \uparrow \hspace*{7cm} \\
& \qquad\oplus 
\begin{matrix} 
M\big(s_2s_3\bullet 0\big)\\
M\big(s_3s_2\bullet 0\big) 
\end{matrix} \longleftarrow \oplus \begin{matrix}
M\big(s_3\bullet 0\big)\\
    M\big(s_1s_2\bullet 0\big)
\end{matrix} \longleftarrow M(s_1\bullet 0) \longleftarrow 0\ \  {\Large \bf ?}
\end{align*} 
 }
\end{remark}

\begin{remark}
{\color{black} Following the progress in this paper towards characters or $W_{\mathcal{H}}$ in settings (3) and (4), particularly when there are only two (size 2) holes},
{\color{black}one might turn to $W_{\mathcal{H}}$  for $\mathcal{H}$ with} $\geqslant 3$ minimal
holes. 
{Then $W_{\mathcal{H}}$ in our sought-for character numerators, are not always ``finite Coxeter'', as we explain in below simpler noteworthy case of Setting (3). Let
\[
\mathfrak{g} = \mathfrak{sl}_5, \qquad \lambda \in P^+, \qquad
\mathcal{H} = \{ \{ 1 \}, \{ 3 \},  \{ 2, 4 \} \}.
\]
Then the subgroup $W_{\mathcal{H}}$ of $W \simeq S_5$ generated by $s_1,
s_3, t_2 := s_2 s_4$ has at least the relations
\[
(s_1 t_2)^6 = (s_1 s_3)^2 = (t_2 s_3)^4 = 1
\]
by Lemma~\ref{L2holes}, and one would like to know if these relations
give a Coxeter presentation of $W_{\mathcal{H}}$. But this is necessarily
false, since the only finite Coxeter group with connected underlying
Coxeter graph containing an edge labeled $6$, is a dihedral group
\cite{Cox,Cox2} -- with $2$ Coxeter generators.}
\end{remark}

{\color{black}
\subsection{Setting 3: Minimal holes of sizes at most 2, with all the size 2 ones in them intersecting non-trivially, and integrable Weyl group and its translated cosets} 
 We work for $\mathcal{H}^{\min}$ in setting (3) above Theorem \ref{Theorem character of second order Verma via parabolic Vermas}.
Here we assume that $1\in I$ belongs to all the minimal size 2 holes in $\mathcal{H}^{\min}$, and node $1$ is also an isolated in the Dynkin graph.
So $\mathfrak{g}$ decomposes as $\mathfrak{g}_{\{1\}}\oplus \mathfrak{g}_{I\setminus\{1\}}$.
For any $\lambda\in \mathfrak{h}^*$ and for such $\mathcal{H}$, we obtain the characters and BGG type resolutions for $\mathbb{M}(\lambda, \mathcal{H})$, in terms of those of parabolic Vermas on the ``boundaries'' of $\mathbb{M}(\lambda,\mathcal{H})$. 

\noindent
\underline{Notations}:
Over the subalgebra $\mathfrak{g}_{I\setminus\{1\}}$, let $M_{I\setminus \{1\}}(\mu)\twoheadrightarrow M_{I\setminus \{1\}}(\mu, T)\twoheadrightarrow L_{I\setminus \{1\}}(\mu)$ denote the Verma, parabolic Verma and their simple quotient, with highest weight(s) $\mu\big|_{\mathfrak{h}^*_{I\setminus \{1\}}}$ the restriction of $\mu\in \mathfrak{h}^*$, and with their integrabilities $\emptyset\subseteq T\subseteq J_{\mu}\setminus\{1\}$.

\begin{proof}[{Proof of Theorem \ref{Theorem character of second order Verma via parabolic Vermas}}]
We first show our character formula	for $M=\mathbb{M}(\lambda,\mathcal{H})$ in the theorem, for which we begin by noting some relations in $M$:  
For $0\neq v_{\lambda}\in M_{\lambda}$ a fixed highest weight vector and every $j\in J$ and $h\in H$ (recall their definitions in the theorem) -
	\begin{equation}\label{Eqn Parabolic Verma relations in second order Verma}     f_{j}^{\lambda(\alpha_{j}^{\vee})+1} v_{\lambda}= f_{h}^{\lambda(\alpha_{h}^{\vee})+1} f_1^{\lambda(\alpha_1^{\vee})+1} v_{\lambda} = f_{j}^{\lambda(\alpha_{j}^{\vee})+1} f_1^{\lambda(\alpha_1^{\vee})+1} v_{\lambda}=0.
	\end{equation}
Visibly, the only relations in $M$ along $\big(J\sqcup (H\setminus\{1\})\big)$-directions are the leftmost ones in \eqref{Eqn Parabolic Verma relations in second order Verma}.
	So the highest weight $\mathfrak{g}_{I\setminus \{1\}}$-module $U\mathfrak{g}_{I\setminus \{1\}} v_{\lambda}$ is the parabolic Verma $M_{I\setminus\{1\} }(\lambda, J)$.
So roughly speaking, $\sum\limits_{w\in W_J}(-1)^{\ell(w)} e^{w\bullet \lambda}$ appears in $\mathrm{ch}M$; other exponentials involve $e^{-\alpha_1}$.
    More precisely,
	\begin{equation}\label{Eqn character formula lower half}
	\sum\limits_{\mathclap{\xi \in \mathbb{Z}_{\geq 0}\Pi \
			\text{ s.t. }\mathrm{ht}_{\{1\}}(\xi)\leq \lambda(\alpha_1^{\vee})}} \qquad  \dim M_{\lambda-\xi}\ e^{\lambda-\xi} =  \frac{1-e^{s_1\bullet \lambda} }{1-e^{-\alpha_1}}\mathrm{ch}M_{I\setminus \{1\}}(\lambda, J),
	\end{equation}
	as:
	i) $f_1$ commutes with every element $F\in U\mathfrak{n}^-_{I\setminus \{1\}}$ (since $\{1\}$ is isolated in the Dynkin graph).\\
	ii) For such $F$ and for $0\leq b\leq\lambda(\alpha_1^{\vee}) $, observe by $\mathfrak{sl}_2$-theory in $M$ that $f_1^{b} F v_{\lambda}\neq 0 \iff F v_{\lambda}\neq 0$.\\ 
	iii) So $\dim M_{\lambda-\xi} = \dim M_{I\setminus\{1\}} (\lambda, J)_{\lambda-\xi}$ for all $\xi\in \mathbb{Z}_{\geq 0}\Pi$ with $\mathrm{ht}_{\{1\}}(\xi)\leq \lambda(\alpha_i^{\vee})$.
	
Moreover, the equality of dimensions in point iii) holds true even if $\mathrm{ht}_{\{1\}}(\xi)> \lambda(\alpha_i^{\vee})$, because:
	The $\mathfrak{g}$-submodule $U\mathfrak{g}\cdot f_1^{\lambda(\alpha_1^{\vee})+1} v_{\lambda}\subset M$ generated by $f_1^{\lambda(\alpha_1^{\vee})+1}v_{\lambda}$ is the parabolic Verma $M\big(s_1\bullet\lambda, J\sqcup (H\setminus \{1\})\big)$, since every relation in $U\mathfrak{g}\cdot f_1^{\lambda(\alpha_1^{\vee})+1} v_{\lambda}$ is either the second or the third relation in \eqref{Eqn Parabolic Verma relations in second order Verma}, and which resemble the integrability relations in $M_{I\setminus \{1\}}(s_1\bullet\lambda, J\sqcup H\setminus \{1\})$. 
	Now the desired character formula is implied by \eqref{Eqn character formula lower half} and
	\begin{equation}\label{Eqn character formula upper half}
	\sum\limits_{\mathclap{\xi \in \mathbb{Z}_{\geq 0}\Pi \
			\text{ s.t. }\mathrm{ht}_{\{1\}}(\xi)> \lambda(\alpha_1^{\vee})}} \qquad  \dim M_{\lambda-\xi}\ e^{\lambda-\xi} =  \frac{e^{s_1\bullet \lambda}}{1-e^{-\alpha_1}} \mathrm{ch}M_{ I\setminus \{1\}}(\lambda, J\sqcup H\setminus \{1\}).
	\end{equation}
    Now we assume $I=J\sqcup H$ -- note $J$ and $H$ do not intersect by minimality of holes in them -- and show the following claimed resolution in the theorem.
	\begin{align}\label{Eqn BGG resolution for setting (3)}
	\cdots  \xrightarrow{d_{i+1}}\qquad\ \ 
	\begin{matrix}
	& \bigoplus\limits_{\mathclap{u\in W_J \text{ s.t. } \ell(u)=i+1}} \quad M(u\bullet \lambda)\\
	&   \bigoplus\limits_{\mathclap{w\in W_{J\sqcup (H\setminus \{1\})}\setminus W_J \text{ s.t. }\ell(w)=i+1}}\quad  M(s_1\ w\bullet \lambda)  \end{matrix} \xrightarrow{d_i} \cdots   \xrightarrow{d_{1}}    \begin{matrix}
	& \bigoplus\limits_{j\in J} M(s_j\bullet \lambda)\\
	& \bigoplus\limits_{h\in H\setminus \{1\}} M(s_1s_h\bullet \lambda)
	\end{matrix} \xrightarrow{d_0}  M(\lambda) \twoheadrightarrow  \mathbb{M}(\lambda, \mathcal{H})
	\end{align}
	The above resolution with $s_1$ omitted in all the top weights of Verma summands, resembles the usual BGG resolution of the $\mathfrak{g}_{I\setminus \{1\}}$-simple $L_{I\setminus \{1\}}(\lambda)$:
	\begin{equation}\label{Eqn BGG resolution for inregrable simples}
	\cdots\xrightarrow{c_{i+1}} \quad \qquad \bigoplus\limits_{\mathclap{w\in W_{I\setminus\{1\}}\text{ s.t. }\ell(w)=i+1}} \quad  M_{I\setminus\{1\}}(w\bullet \lambda) \xrightarrow{c_i}\cdots \xrightarrow{c_1}
	\bigoplus\limits_{i\in I\setminus\{1\}} M_{I\setminus \{1\}}(s_i\bullet \lambda)  \xrightarrow{c_0} M_{I\setminus\{1\}}(\lambda) \twoheadrightarrow L_{I\setminus\{1\}}(\lambda)
	\end{equation}
       The resolution in the general case of $J\sqcup H\subseteq I$ can be obtained by tensoring the chain \eqref{Eqn BGG resolution for inregrable simples} with $U\bigg(\bigoplus\limits_{\alpha\in \Delta^+\setminus \Delta^+_{J\sqcup H}}\mathfrak{n}^-_{\alpha}\bigg)$; in view of the freeness of the summands in \eqref{Eqn BGG resolution for inregrable simples} w.r.t. that unipotent subalgebra. 
	The proof of \eqref{Eqn BGG resolution for setting (3)} can be completed proceeding in the following steps :\\
	0) For each $w\in W_{\mathcal{H}}$, we let $m_{w\bullet \lambda}\neq 0$ denote a highest weight vector in the summand $M(w\bullet \lambda)$ of our resolution. 
	We define $W_K(i):=\big\{w\in W_K\ \big|\ \ell(w)=i \big\}$ for every $K\subseteq I$, and $m_1:=\lambda(\alpha_1^{\vee})+1$.\\
	1) Observe $M(s_1w\bullet\lambda)\equiv M_{\{1\}}(s_1\bullet\lambda)\otimes_{\mathbb{C}} M_{I\setminus\{1\}}(w\bullet \lambda)$ as $\mathfrak{g}_{\{1\}}\oplus \mathfrak{g}_{I\setminus\{1\}}$-modules for every $w\in W_{J\sqcup (H\setminus\{1\})}\setminus W_J$; similarly $M(u\bullet \lambda) = M_{\{1\}}(\lambda)\otimes_{\mathbb{C}} M_{I\setminus\{1\}}(u\bullet\lambda)$ $\forall$ $u\in W_J$.\\
	2) We let $d_0$ to be the usual embedding.
	For each $i$, we let $\big[c_i^{w}(u)\big]_{\substack{u\in W_{I\setminus \{1\}}(i+1) \\ w \in W_{I\setminus \{1\}}(i)}}\ = \ \bigg(\begin{matrix}
	    c_i' & c_i''\\
        0 & c_i'''
	\end{matrix}\bigg)$ be the matrix of the $\mathfrak{g}$-module map $c_i$,  with: i) its entries $c_i^w(u)\in (U\mathfrak{n}^-)_{w\bullet \lambda- u\bullet \lambda}$; ii) $c_i',c_i'',c_i'''$ are the blocks w.r.t. the \Big[$W_J(.)\ \sqcup \ \big(W_{J\sqcup (H\setminus \{1\})}(.)\setminus W_J\big)$\Big]-decompositions of the domain and co-domain of $c_i$.
    Here we used the standard $\mathbb{Z}_{\leq 0}\Pi$-gradation of $U\mathfrak{n}^-$. \\
	3) 
	We define $d_i$ on  $\begin{matrix}
	\bigoplus\limits_{u\in W_J(i+1)} M(u\bullet \lambda) \\
	\bigoplus\limits_{w\in W_{J\sqcup (H\setminus \{1\})}(i+1)\setminus W_J} M(s_1w\bullet \lambda)
	\end{matrix}$ 
	by point 1) as $\bigg(\begin{matrix} \text{Id}\otimes c_i' &  f_1\otimes c_i''' \\ 0 & \text{Id}\otimes c_i'' \end{matrix}\bigg)$ via extending $(U\mathfrak{n}^-)$-linearly the following assignments : 
	$m_{u\bullet \lambda} \overset{d_i}{\mapsto} \sum\limits_{w\in W_J(i)}  c_i^w(u) m_{w\bullet \lambda}$ $\forall$ $u\in W_J(i+1)$ and  
	$m_{s_1u\bullet \lambda} \ \ \overset{d_i}{\mapsto} \sum\limits_{w\in W_J(i)}\ f_1^{m_1}\ c_i^w(u)m_{w\bullet \lambda}$ + $ \sum\limits_{w\in W_{J\sqcup (H\setminus \{1\})}(i) \setminus W_J} c_i^w(u)m_{s_1w\bullet \lambda}$ $\forall$ $u\in W_{J\sqcup (H\setminus\{1\})}(i+1)\setminus W_J$.
	So $d_i$s are visibly well-defined; note $d_i$ maps $W_J(i+1)$-parts into $W_J(i)$-parts.\\
	4) Now the vanishing of $d_{i-1}\circ d_i(m_{u\bullet\lambda})$, for every $u\in W_{J\sqcup (H\setminus \{1\})}(i+1)\setminus W_J$ (which suffices to be showed):
	By the exactness along the $c_i$-chain, the vanishing of the \big[$W_{J\sqcup (H\setminus \{1\})}(i-1)\setminus W_J$\big]-components $\qquad \sum\limits_{\mathclap{\substack{w\in W_{J\sqcup (H\setminus \{1\})}(i) \setminus W_J\\  v\in W_{J\sqcup (H\setminus\{1\})}(i-1)\setminus W_J}}}\qquad c_i^w(u) c_{i-1}^v(w)m_{s_1v\bullet \lambda}$ in $d_{i-1}\bigg(\sum\limits_{w\in W_{J\sqcup (H\setminus\{1\})}(i) \setminus W_J} c_i^w(u)m_{s_1w\bullet \lambda}\bigg)$ follows.
	In $d_{i-1}\bigg(\sum\limits_{w\in W_J(i)}$ $f_1^{m_1} c_i^w(u) m_{w\bullet \lambda} + \sum\limits_{w\in W_{J\sqcup (H\setminus\{1\})}(i)\setminus W_J}$ $c_i^w(u)m_{s_1w\bullet \lambda}\bigg)$ for $u\in W_{J\sqcup (H\setminus \{1\})}(i+1)\setminus W_J$, the $W_J(i-1)$-components $\sum\limits_{w\in W_J(i)}f_1^{m_1} c_i^w(u)$ $\sum\limits_{v\in W_J(i-1)}c_{i-1}^v(w) m_{v\bullet \lambda} + \sum\limits_{w\in W_{J\sqcup (H\setminus\{1\})}(i)\setminus W_J}$ $c_i^w(u)$  $\sum\limits_{v\in W_J(i-1)} f_1^{m_1} c_{i-1}^v(w) m_{v\bullet \lambda}$ vanish by the same reasoning  and as $f_1$ is in the center of $U\mathfrak{n}^-$.
	The exactness of $d_i$-chain is seen similarly, by that of $c_i$-chain or by using the matrix representations of $d_i$s.
This completes the proof of Theorem \ref{Theorem character of second order Verma via parabolic Vermas}.
\end{proof}
}

\subsection{Setting 4: Type $A_3$, edges incident on both the nodes of the size 2 hole, and Weyl semigroup consists of the words on alphabets within holes}{\color{black}
Here $\mathfrak{g}=\mathfrak{sl}_4(\mathbb{C})$, $I=\{1,2,3\}$ and $\mathcal{H}=\{H_1=\{2\},\ H_2= \{1,3\}\}$. 
Note both the nodes $1,3$ have edges to $2$; this generalizes the settings in the previous subsections wherein at most one node in $H_i$ can have edges to the nodes in $H_j$ for every pair $H_i\neq H_j\in \mathcal{H}^{\min}$.
We now quickly show the proof of the character formula for $\mathbb{M}(0,\mathcal{H})$; observe that the character for the other second order Verma $\mathbb{M}\big(\lambda, 
\{\{1,3\}\} \big)$ for any $\lambda$ with $\{1,3\}\subseteq J_{\lambda}$, was already computed in Subsection \ref{Subsection Setting 1} or Subsection \ref{Subsection Setting 2}.
It might be interesting explore if there is a BGG type resolution using the 16 exponential weights in the character numerator in this case, similar to \eqref{EBGGH};  which we could not obtain yet.
\begin{proof}[{Proof of Proposition \ref{Prop A3 second order holes character}}]
Let $(.,.)$ be the standard invariant form with norm $\|.\|$ on $\mathfrak{h}^*$. 
Let $R=\mathrm{ch} M(0)$ be the usual denominator in the Weyl--Kac character formula.
We begin by recalling from \cite[Proposition 9.8]{Kac book} for the highest weight module $M:=\mathbb{M}(0, \mathcal{H})$ in the statement, that
	\begin{equation}\label{Eqn character formula principle}
	\mathrm{ch}M=\sum\limits_{\substack{ \mu\preceq 0 \\ \| \mu+\rho\|=\|0+ \rho\| }}c(\mu)\mathrm{ch}M(\mu)\qquad
	\ \text{ with }\ c(\mu)\in \mathbb{Z} \ \ \text{ and }c(0)=1.
	\end{equation}
In order to compute $c(\mu)$s here, we determine explicitly: i)~all $\mu\preceq 0$ with $\| \mu+\rho\|=\| \rho\|$ ii)~multiplicities of such $\mu$ in $\mathrm{ch} M$, and iii)~the expression for $R\times \mathrm{ch}M$.
	
	Assume $\mu=-X\alpha_1-Y\alpha_2-Z\alpha_3$ for $X,Y,Z\in \mathbb{Z}_{\geq 0}$ satisfies $\big( \mu+\rho, \mu+\rho)= (\rho, \rho)$.
	So $\|X\alpha+Y\alpha_2+Z\alpha_3\|^2- 2\big(\rho, X\alpha+Y\alpha_2+Z\alpha_3\big)=0$. 
	Now by $(\alpha_2,\alpha_2)=(\alpha_1,\alpha_1)=(\alpha_3,\alpha_3)$ and $\Big(\rho,\frac{2 \alpha_i}{(\alpha_i,\alpha_i)} \Big)=1$ $\forall$ $i\in I$, we have $(X^2+Y^2+Z^2)(\alpha_2,\alpha_2)-(\alpha_2,\alpha_2)(X+Y+Z)+2XY(\alpha_1,\alpha_2)+2YZ(\alpha_2,\alpha_3)$.
	Finally using $\left(\alpha_i,\ \frac{2\alpha_2}{(\alpha_2,\alpha_2)}\right)=\left\langle \alpha_i,\alpha_2^{\vee} 
	\right\rangle=-1$ for $i\in \{1,3\}$ 
	\begin{equation}\label{Eqn for X,Y,Z}
	(X^2+Y^2+Z^2)\ -\ (X+Y+Z)\ -\ XY\ -\ 
	YZ\ \ =\ \ 0.
	\end{equation}
	We see that necessarily $X,Y\leq 3$ and $Y\leq 4$; i.e., $\mu \succeq w_{\circ}\bullet 0$, for the longest word $w_{\circ}\in W$.
	\begin{question}
		In semisimple case, for any $\lambda\in P^+$ -- particularly for $\lambda=0$ -- does $\|\mu+\rho\|=\|\lambda+\rho\|$ imply $\mu\succeq w_{\circ}\lambda$?
    An answer to this problem might help the computation of $c(\mu)$s. \end{question}
	Re-writing \eqref{Eqn for X,Y,Z}, we see $0=X^2+Z^2-(Y+1)(X+Z)+Y^2-Y= \Big(X-\frac{(Y+1)}{2}\Big)^2+\Big(Z-\frac{(Y+1)}{2}\Big)^2+ Y^2-Y-\frac{(Y+1)^2}{2}$.
	Now this expression is strictly positive when $Y\geq 5$, as $Y^2-Y-\frac{(Y+1)^2}{2}=\frac{1}{2}\big(Y(Y-4)-1\big)>0$ if $Y\geq 5$.
	Next suppose $X\geq 4$ \big(given $Y\leq 4$\big).
	For $Y\in \{0,1,2,3,4\}$, we see that $\Big(X-\frac{(Y+1)}{2}\Big)^2+ Y^2-Y-\frac{(Y+1)^2}{2}$ resp. equals: $X^2-X-\frac{1}{4}\ , \ X^2-2X-1\ , \ X^2-3X-\frac{1}{4}\ , \ X^2-4X+2\ , \ X^2-5X+\frac{23}{4}$.
	All these expressions are strictly positive if $X\geq 4$; so we must have $X\leq 3$, similarly $Z\leq 3$, finishing the proof of part i).
	
	For goal ii) of computing $\mu$-multiplicities, we will prove for $\mu=-\alpha_1-b\alpha_2-c\alpha_3$, more strongly 
	\begin{equation}\label{dim of M(1,3: 2)}
	\dim M_{\mu}\ = \ \begin{cases}
	2\ &\  \text{if }a=b=c\geq 1,\\
	1\ &\  \text{if }b \text{ is between }a \text{ and }c\neq a,\\
	0 \ & \ \text{else (i.e., if }b>\max \{a,c\} \text{ or } b< \min\{a,c\}).
	\end{cases}
	\end{equation}
	Assuming this claim to be true, for goal iii) on numerator of $\mathrm{ch} M$ in the theorem, we compute\\ $\Big(1-e^{-\alpha_1}\Big)\Big(1-e^{-\alpha_2}\Big)\Big(1-e^{-\alpha_3}\Big)\Big(1-e^{-\alpha_1-\alpha_2}\Big)\Big(1-e^{-\alpha_2-\alpha_3}\Big)\Big(1-e^{-\alpha_1-\alpha_2-\alpha_3}\Big) \ \times\  \bigg( \sum\limits_{0\leq a\leq b\leq c\leq 4}e^{-a\alpha_1-b\alpha_2-b\alpha_3}\ +\ \sum\limits_{4\geq a\geq b\geq c\geq 0}e^{-a\alpha_1-b\alpha_2-b\alpha_3} \bigg)
	$ using SageMath program; its code is written in Appendix \ref{Section Sagemath code}.
    Here we took the truncated character as all the weights $\mu$ in \eqref{Eqn character formula principle} lie above $w_{\circ}\bullet 0$.
	
	For \eqref{dim of M(1,3: 2)}, we fix a highest weight vector $0\neq v\in M_0$, and observe that
	$M=\mathbb{C}v\oplus\big( U\mathfrak{n}^-\cdot f_1v\ +\  U\mathfrak{n}^-\cdot f_3v\big)$; since $f_2v=0$.
	Next for each $i\in \{1,3\}$, the highest weight $\mathfrak{g}$-module $U\mathfrak{n}^-f_iv$ is a quotient of the parabolic Verma $M\big(-\alpha_i\ , \ I \setminus \{i\}\big)$; by the integrability relations $f_2^2 (f_iv)= f_j(f_iv)=0$ for $i\neq j\in \{1,3\}$.
	Thus $\wt M\subseteq  \{0\}\bigsqcup \bigg(\bigcup\limits_{i\in \{1,3\}} \Big[\wt M\big(-\alpha_i,\ I\setminus\{i\}\big)\ = \ \big\{-\alpha_i,\ -\alpha_i-\alpha_2,\ - \alpha_1-\alpha_2-\alpha_3\big\} \ - \ \mathbb{Z}_{\geq 0}\big\{\alpha_i,\ \alpha_i+\alpha_2,\ \alpha_1+\alpha_2+\alpha_3\big\}\Big] \bigg)$.
By this, we arbitrarily fix a weight $\mu=-a\alpha_1-b\alpha_2-c\alpha_3 \in \wt M$ with either $a\geq b\geq c$ or $a\leq b\leq c$, for the whole proof below.
	%Now we note that $(-\alpha_i,\ \beta)\in \{-1,-2\}$ for all $\alpha\in \Delta^+\setminus \Delta_{\mathcal{I}\setminus \{i\}}^+=\big\{\alpha_i,\ \alpha_i+\alpha_2,\ \alpha_1+\alpha_2+\alpha_3\big\}$.
	%Hence by the simplicity of parabolic Verma modules $M\big(0,\ \mathcal{I}\setminus\{i\}\big)$, we must have that $U(\mathfrak{n}^-)f_iv\ \equiv\ M(-\alpha_i,\ \mathcal{I}\setminus \{i\} \big)\ \equiv\ L(-\alpha_i)$ for all $i\in \{1,3\}$, and moreover $\dim V_{\mu\precneqq 0}= \sum\limits_{i\in \{1,3\}}\dim M(-\alpha_i,\ \mathcal{I}\setminus \{i\} \big)_{\mu}$.
	This formula for $\wt M$ already proves the last case of 0 multiplicities in \eqref{dim of M(1,3: 2)}.
	
	We first study the cases of $a>b$ or $b>c$ (so $a>c$), which comprise of the 3 subcases $a>b=c$ or $a>b>c$ or $a=b>c$; the proof when $a<b$ or $b<c$ is similar.
	We fix the ordered basis for $\mathfrak{n}^-$
	\begin{equation}\label{Eqn PBW basis in Verma}
	\big\{\big[f_3, [f_2,f_1]\big]\  , \  [f_3,f_2]\ , \ [f_2,f_1]\  , \ f_2\ , \ f_3\ , \ f_1\big\}.
	\end{equation}
	We work with the images in $M$ of the PBW-monomials on the root vector basis in \eqref{Eqn PBW basis in Verma}:
	\begin{equation}\label{PBW monomial basis vect in Verma}
	\big[f_3, [f_2,f_1]\big]^p \ \cdot \   [f_3,f_2]^q\ \cdot \ [f_2,f_1]^r\  \cdot \ f_2^s\ \cdot \ f_3^t\ \cdot \ f_1^m \cdot v\qquad \text{for }p,q,r,s,t,m\in \mathbb{Z}_{\geq 0}. 
	\end{equation}
	Consider the images of these vectors in $M_{\mu}$; so $p+r+m=a\ \geq \  p+q+r+s=b\ \geq \ p+q+t=c$.
	We \underline{prove in the below steps}, with their plans quoted at the beginning of each step,
	\begin{equation}\label{Eqn desired form for unique PBW vector}
	M_{\mu} \ = \ \mathbb{C}\bigg\{ v_{\mu}\ \ :=\ \ \big[f_3,[f_2,f_1]\big]^{c}\cdot [f_2,f_1]^{\overset{\geq 0}{b-c}}\cdot f_1^{\overset{\geq 0}{a-b}}v \bigg\}\quad (\text{for }a>c); \quad \text{so that }\ \dim M_{\mu} =1.
	\end{equation}
	1) {\it $v_{\mu}$ survives in $M$} :
	Note $e_2$ commutes with $\big[f_3, [f_2,f_1]\big]$, and moreover $e_2^{b-c}[f_2,f_1]^{b-c} - \big((b-c)!\big)f_1^{b-c}\in \sum_{i\geq 0}U\mathfrak{n}^-\cdot e_2^{1+i}$.
	Similarly $e_3^c \big[f_3,[f_2,f_1]\big]^c - (c!)[f_2,f_1]^c \in \sum_{i\geq 0}U\mathfrak{n}^-\cdot e_3^{1+i}$.
	Now $e_2^{c}\cdot e_3^{c}\cdot e_2^{b-c}\ v_{\mu} \ \ =\ \ (c!)^2 \big((b-c)!\big)f_1^{a}v\ \ \neq 0 \in M$ is easily seen.
	In the below steps, we arbitrarily fix a non-zero PBW weight vector from \eqref{PBW monomial basis vect in Verma}, and prove it to be a scalar times $v_{\mu}$. 
	\smallskip\\
	2) {\it Crucial equality/identity in $M$ to turn $[f_3,f_2]\times f_1$s into $\big[ f_3, [f_2,f_1]\big]$} : Fix $k,l\in \mathbb{N}$.
	$[f_3,f_2]^k f_1^l v= [f_3,f_2]^{k-1} f_3f_2f_1^l v - [f_3,f_2]^{k-1} f_2\cancel{f_3f_1^{l}v}^{=0} = \underset{(\text{as }[f_2,f_1] \text{ and }f_1 \text{ commute})}{l[f_3,f_2]^{k-1} f_3[f_2,f_1]f_1^{l-1}v} + [f_3,f_2]^{k-1} f_3f_1^l \cancel{f_2 v}^{=0}
	=l[f_3,f_2]^{k-1}\\ \big[f_3,[f_2,f_1]\big] f_1^{l-1}v+l\delta_{l,1}[f_3,f_2]^{k-1}[f_2,f_1]f_3f_1^{l-1}v$.
	In this manner, we observe at each $j^{\text{th}}$-iteration for $1\leq j\leq \min\{k,l\}$ : 
	\begin{align}\label{Imp identity in A3}
	\begin{aligned}[]
	[f_3,f_2]^k f_1^l v =  (j!)\binom{l}{j}   \big[f_3, [f_2,f_1]\big]^{j-1} [f_3,f_2]^{k-j}\Big(\big[f_3, [f_2,f_1]\big] f_1^{l-j}v+\ \delta_{j, l} [f_2,f_1]f_3f_1^{l-j} v\Big). 
	\end{aligned}\end{align}
	3) {\it Reducing $f_3$s at the beginning} : Note $p+r+m\geq p+q+r+s\implies m\geq q+s$.
    When $m=t=0$, necessarily $q=s=0$, and the corresponding weight vector in \eqref{PBW monomial basis vect in Verma} is already in the desired form in \eqref{Eqn desired form for unique PBW vector}.
	We know that only one of $m$ and $t$ can be non-zero.
	Suppose $m=0$ and $t>0$.
	Then $q=s=0$, and $b>c$ implies $r>t$.
    Indeed we will see below that the vector (for parameters $p; q=0;r;s=0;t;m=0$) we started with in \eqref{PBW monomial basis vect in Verma} vanishes in $M$.
	We apply \eqref{Imp identity in A3} at $j=t-1$ to the powers of $[f_2,f_1]$ and $f_3$ (in place of $[f_3,f_2]$ and $f_1$ therein), so that the vector in the previous line is $(t!)(-1)^{t-1}\Big( \big[f_3,[f_2,f_1]\big]^{p+t-1} [f_2,f_1]^{r-t-1}\big) \big( [f_2,f_1]^2f_3v\Big)$.
	The goal of this step is to prove that the vector $[f_2,f_1]^2 f_3 v=  -[f_2,f_1]f_1 [f_2, f_3]v =  -[f_2,f_1]\big[f_3, [f_2,f_1] \big] v +[f_3,f_2][f_2,f_1]f_1v $ vanishes in $M$ \big(for this reason we applied \eqref{Imp identity in A3} for $j=t-1$ and not for $j=t$\big), using the following computations.
	Namely, in $M(\lambda)$, we  find constants $d_1, \ldots, d_6\in \mathbb{C}$ for which :\\
	$ -[f_2,f_1]\big[f_3, [f_2,f_1] \big] m_0 +[f_3,f_2][f_2,f_1]f_1 m_0 = \begin{bmatrix}
	d_1 \big[f_3, [f_2, f_1]\big]f_1 \\
	d_2 [f_3, f_2]f_1^2\\
	d_3 [f_2, f_1]f_3f_1\\
	d_4 f_2 f_3 f_1^2
	\end{bmatrix} f_2m_0 \ +\ \begin{bmatrix}
	d_5 f_2^2 f_1\\
	d_6 [f_2, f_1]f_2
	\end{bmatrix}f_3 f_1 m_0$.
	The terms on the right hand side of this equation lie in the submodule $U\mathfrak{n}^-\cdot\{f_1f_3m_0, f_2m_0\}$, and those on the left are some PBW basis vectors.
	So we write each of the six summands $ d_1 \big[f_3, [f_2, f_1]\big]f_1f_2 m_0, \ldots, d_6[f_2, f_1]f_2 f_3 f_1 m_0$ as a linear combination of PBW basis  weight vectors :
	\begin{itemize}
		\item $d_1 \Big( - \big[f_3, [f_2, f_1]\big] [f_2, f_1] m_0 +  \big[f_3, [f_2, f_1]\big] f_2 f_1 m_0\Big)$
		\item $d_2 \Big( -2  [f_3, f_2][f_2, f_1]f_1 m_0 + [f_3, f_2]f_2 f_1^2 m_0  \Big)$
		\item \ \ $d_3 \Big(- [f_2, f_1]f_3 [f_2, f_1]m_0 + [f_2, f_1][f_3,f_2]f_1 m_0 + [f_2, f_1]f_2 f_3 f_1 m_0 \Big) \ = $\\
		$ d_3 \Big(- \big[f_3 ,[f_2, f_1]\big] [f_2, f_1]m_0 - [f_2, f_1]^2 f_3 m_0 + [f_2, f_1][f_3,f_2]f_1 m_0 + [f_2, f_1]f_2 f_3 f_1 m_0 \Big) $
		\item \ \ $d_4\Big( -2 f_2 f_3 [f_2, f_1]f_1 m_0 + [f_3, f_2]f_2 f_1^2 m_0 + f_2^2 f_3 f_1^2 m_0   \Big) = $\\
		$d_4\Big(  -2 \big[f_3, [f_2, f_1]\big]f_2f_1m_0 -2[f_2,f_1]f_2 f_3 f_1 m_0 + [f_3,f_2]f_2 f_1^2m_0 + f_2^2 f_3 f_1^2 m_0 \Big)$
		\item $d_5\Big(f_2^2 f_3 f_1^2m_0 \Big)$\\
		\item $d_6 \Big([f_2, f_1]f_2 f_3f_1m_0 \Big)$
	\end{itemize}
	Comparing the coefficients of the seven PBW basis vectors in $M(0)_{-2\alpha_1 -2\alpha_2 -\alpha_3}$ above, namely $\big[f_3, [f_2, f_1]\big] [f_2, f_1] m_0$ , $\big[f_3, [f_2, f_1]\big] f_2 f_1 m_0\ , \  [f_3, f_2][f_2, f_1]f_1 m_0\ , \ [f_3, f_2]f_2 f_1^2 m_0 \ , \ [f_2, f_1]^2 f_3 m_0$,
	$[f_2, f_1]f_2 f_3 f_1 m_0\ , \ f_2^2 f_3 f_1^2m_0 $, we see respectively the following equations \\
	$-d_1-d_3 =-1 \quad     -2d_2+d_3= 1\qquad  d_1-2d_4\ = \ d_2+d_4 \ = \ -d_3\ = \   d_3-2d_4+d_6 \ = \  d_4+d_5 = 0$.\\
	We are done by observing $(d_1,\ldots, d_6)\ = \ \Big( 1, \frac{-1}{2}, 0, \frac{1}{2}, \frac{-1}{2}, 1 \Big)$ to satisfy the above system of 7 equations.
	In the rest of the proof we assume that $m>0$; so $t=0$ and $m\geq q+s$.\smallskip\\
	4) {\it Converting all the $f_2$s at the front to $[f_2,f_1]$s, and then $[f_3,f_2]$s into $\big[f_3, [f_2,f_1]\big]$s}  :  As above  $f_2^sf_1^mv = m [f_2, f_1]f_2^{s-1}f_1^{m-1}v + f_2^{s-1}f_1^m \cancel{f_2 v}^{=0}$, and so $f_2^sf_1^m v = \Big(m\times \cdots \times (m-s+1)\big) [f_2,f_1]^{s}f_1^{m-s}v$. 
	So we could have begun with $s=0$ in \eqref{PBW monomial basis vect in Verma} (and $m>0$).
	For the second claim in this step, suppose $q>0$.
	As $[f_3,f_2]$ commutes with $[f_2,f_1]$, we push $[f_3,f_2]^q$ past $[f_2,f_1]^r$, and work with $[f_3,f_2]^qf_1^{m}v$ in the front.
	Recall $m\geq q$, and moreover either $m>q$ or ($m=q$ and) $r>0$.
	When $m>q$, we turn each $[f_3,f_2]$ into $\big[f_3, [f_2,f_1]\big] $ as explained in Step 2.
	When $m=q$, we apply  \eqref{Imp identity in A3} for $j=m-1$ (as in Step 3), resulting in $\star\star\star[f_2,f_1][f_3,f_2]f_1v =\star\star\star\big(\big[f_3, [f_2,f_1]\big] [f_2,f_1]v+ [f_2,f_1]^2 f_3 v\big)$.
	We are done by noting from Step 3 that the rightmost term in this equation vanishes in $M$.
	
The leftover case is for (non-zero) vectors in \eqref{PBW monomial basis vect in Verma} with $a=b=c$; fix one such vector.
Note as in Step 1, $e_2^a e_3^a \big[ f_3, [f_2,f_1] \big]^{a}v =(a!)^2 f_1^a v\neq 0$ survives in $M$.
	Observe $m=t=0$ forces $s=q=r=0$, as $a=b=c$; yielding the vector in the previous line.
	Also since both $t$ and $m$ cannot be simultaneously positive, let us assume without loss of generality $m>t=0$.
	Now the total count of $f_2$s not exceeding that of $f_3$s forces $r=s=0$.
	% And then $r=0$ for the number of $f_1$s to not exceed the number of $f_3$s.
	So we are left with the vectors $\big[f_3, [f_2,f_1]\big]^{a-m}[f_3,f_2]^m f_1^m v$ for $m\leq a$.
    By \eqref{Imp identity in A3},
	$[f_3,f_2]^m f_1^mv$ $= 
(m!) \big[f_3,[f_2,f_1]\big]^{m-1} [f_3,f_2]f_1v$; which is again non-vanishing in $M$ as $e_2^me_3^m\big[f_3,[f_2,f_1]\big]^{m-1}\cdot [f_3,f_2]f_1v$ $=  \big((m-1)!\big)^2 f_1^m v\ \neq 0$ in $M$.
	The linear independence (or non proportionality) of $\big[f_3,[f_2,f_1]\big]^{m-1} [f_3,f_2]f_1v$ with $\big[f_3, [f_2,f_1]\big]^m v$ follows by:  $e_1^m\cdot  \big[f_3,[f_2,f_1]\big]^{m-1}\cdot [f_3,f_2]f_1v \ = \ (-1)^{m-1}(m-1)![f_3,f_2]^{m-1}\cdot \underbrace{e_1f_1v}_{=0}$, but $e_1^{m}\cdot \big[f_3,[f_2,f_1]\big]^m v\ = \ (-1)^m (m!) [f_3,f_2]^m v \neq 0$ in $M$. 
	Thus $\dim M_{-m(\alpha_1+\alpha_2+\alpha_2)}=2$ $\forall$ $m\in \mathbb{N}$.
	Finally the proof of the proposition is complete.
\end{proof}

}

\subsection{Concluding remarks and questions}\label{Subsection conclusion}

Looking back -- in the course of this paper, we have uncovered a host of
objects and properties:
(a)~The higher order Verma modules $\mathbb{M}(\lambda, \mathcal{H})$ --
whose weight-sets have closed-form expressions and Minkowski sum
decompositions, and comprise the weights of all highest weight
$\mathfrak{g}$-modules -- and moreover, these modules themselves comprise
all highest weight $\mathfrak{sl}_2^{\oplus n}$-modules.
(b)~BGG resolutions and $W_\mathcal{H}$-invariant characters of these
modules.
(c)~$k$th order upper- and lower-approximations of all highest weight
modules.
(d)~The $k$th order integrability of highest weight modules, leading to
iteratively stratifying the quotients of each $M(\lambda)$, each time
into intervals.
(e)~Category $\mathcal{O}^\mathcal{H}$, with enough projectives and
BGG-type reciprocity.

All of these are natural extensions of their zeroth and first order
versions in the literature -- $M(\lambda), M(\lambda, J)$; their
characters and resolutions; $\mathcal{O}, \mathcal{O}^{\mathfrak{p}_J}$
-- i.e., they ``occur in nature'', and are not ``artificial'' constructs.
Yet, surprisingly, to our knowledge they were unexplored to date. The
present paper provides information on the above objects/properties to
varying degrees of completeness -- from Theorem~\ref{T1} to Proposition \ref{Prop A3 second order holes character}. 
We end this paper {\color{black}now with listing five of our weight formulas at one place below} and a few questions that are natural
to explore going ahead, given the above results.
{\color{black}
\begin{theorem}\label{Theorem wt form Km setting}
	Fix any Kac--Moody $\mathfrak{g}$, $\lambda\in \mathfrak{h}^*$, and highest weight $\mathfrak{g}$-module $M(\lambda)\twoheadrightarrow V\neq 0$ with hole-set $\mathcal{H}_V$.	
	\begin{align}
	\text{By the highest order Verma }\mathbb{M}(\lambda,\ \mathcal{H}_V) \twoheadrightarrow\ V\ & :\quad 
	\ \wt V \ = \ \wt \mathbb{M}\big(\lambda, \mathcal{H}_V\big).
	\\ 
	\text{By higher order Vermas, stratification}  \  :\ \ \  
	\ & \wt V \ = \ \bigcup\limits_{\mathcal{H}_V \ \subseteq 
		\mathcal{H}\subseteq  \text{ Indep}(J_{\lambda})} \wt \mathbb{M}\big(\lambda, \mathcal{H}\big).\\
	\text{By parabolic }\big(1^{\text{st}}\text{-}\text{order}\big) \text{Vermas} \  :\quad \ \wt V \ & =\  \bigcup\limits_{\substack{J\subseteq J_{\lambda} \text{ s.t.} \\  J\cap H \neq \emptyset \ \forall \ H\in \mathcal{H}_V }}\wt M(\lambda, J).
	\end{align}
	\begin{equation}
	\{\wt V \} \underset{\text{bijec.}}{\leftrightarrow}\ \left\{(\lambda,\ \mathcal{H})\right\} \text{ by isolating top weight}\   \ \&   \ { holes } :\    \wt V \ = \ \wt L_{{J_{\lambda}}}(\lambda)\ +\ \mathbb{M}\big(0, \mathcal{H}_V\big)
	\end{equation}
	\[
	\hspace*{1cm}= \ \wt \ L_{{J_{\lambda}}}(\lambda)\ +  \   \wt \Bigg( \frac{M(0)}{\sum\limits_{H\in \mathcal{H}_V}\Big(U\mathfrak{g}\prod\limits_{h\in H}f_h\Big) M(0)_0}\Bigg)\ = \ \wt \ L_{{J_{\lambda}}}(\lambda)\ + \ \ \  \  \bigcup\limits_{\mathclap{\substack{J\subseteq J_{\lambda} \text{s.t.} \\  J\cap H \neq \emptyset\ \forall \ H\in \mathcal{H}_V }}}\ \ \ \wt M(0, J).
	\]
	\begin{equation}    \label{Eqn wt form via Jordan--Holder series}           \text{By ``Jordan--H\"{o}lder series'' / via non-holes}  \ : \ \ \   \wt V  \ = \ \quad\ \  \bigcup\limits_{\mathclap{ H\in \mathrm{Indep}(J_{\lambda})\setminus \mathcal{H}_V }} \qquad \wt L\big(s_H \bullet \lambda\big).                 \end{equation}
 \end{theorem}
}

\begin{enumerate}
\item Much is understood about the first order approximations of highest
weight modules, but less is known about the higher order analogues.
Thus, a natural question is to better understand the higher order Verma
modules $\mathbb{M}(\lambda, \mathcal{H})$ for all $\mathfrak{g},
\lambda, \mathcal{H}$. Can one ``upgrade'' the weight-sets of the modules
$\mathbb{M}(\lambda, \mathcal{H})$ to their characters, at once extending
the classical formulas \eqref{EWeyl}, \eqref{Epvm} for parabolic Verma
modules as well as \eqref{EWCFH} and \eqref{EWCF2holes2}?

\item Second, do the modules $\mathbb{M}(\lambda, \mathcal{H})$ occur at
the end of a BGG-type resolution that uses Verma modules? More
concretely, in what generality in $(\mathfrak{g}, \lambda, \mathcal{H})$
can one show a resolution of $\mathbb{M}(\lambda, \mathcal{H})$, which at
once extends the usual BGG resolution for $M(\lambda, J)$ as well
as~\eqref{EBGGH}, \eqref{EBGGH2} and \eqref{Eqn BGG resolution for setting (3)}? (Informally, can one extend these
Koszul--Taylor resolutions from $(\Z / 2 \Z)^{\oplus n}$ to other Weyl
groups.) 
{\color{black}Can we write such a resolution for the second order Verma in Proposition \ref{Prop A3 second order holes character}?}

\item {\color{black}Can the general version of $(W_{\mathcal{H}}, \ell_{\mathcal{H}})$ be speculated from the one in Proposition \ref{Prop A3 second order holes character}?
In which settings, does one see a (semi)group fashioned $W_{\mathcal{H}}$ with} Coxeter-type enumeration as words in the
generators $\{ {\bf s}_H : H \in \mathcal{H}^{\min} \}$, which act on
$\lambda$? And can Proposition~\ref{PWchar} -- or a refinement thereof --
be proved for more general $\mathfrak{g}, \lambda, \mathcal{H}$? (Note,
it also holds in all parabolic categories
$\mathcal{O}^{\mathfrak{p}_J}$.)

\item Another follow-up is to interpret the modules $\mathbb{M}(\lambda,
\mathcal{H})$ and their resolutions, geometrically on the flag variety --
then extend this to more general $V$, and to the categories
$\mathcal{O}^{\mathcal{H}}$.

\item Extend the treatment of $\mathcal{O}^{\mathfrak{p}_J}$ in the
literature (see e.g.\ \cite{H2,Jantzen,Rocha}) to the higher order
parabolic categories $\mathcal{O}^{\mathcal{H}}$, for all $\mathcal{H}
\subset {\rm Indep}(I)$. Natural questions include:
\begin{itemize}
\item[(a)] \textit{BGG-type reciprocity:}
Extending Theorem~\ref{TBGGrec} to arbitrary $\mathcal{O}^{\mathcal{H}}$
over all $\mathfrak{g}$.

\item[(b)] \textit{Parabolic Verma modules:}
Finding ``higher order Jantzen filtrations'' for $\mathbb{M}(\lambda,
\mathcal{H})$ (see \cite{X2,X3}). (Also note, if $\mathbb{M}(\lambda,
\mathcal{H})$ is simple then it is a parabolic Verma module.)

\item[(c)] \textit{Blocks:}
Endomorphism rings of block-progenerators and other projectives in
$\mathcal{O}^{\mathcal{H}}$, and their properties, including Koszulity
and Koszul duals \cite{Backelin,BGS,Soergel,Stroppel}.
Homomorphisms between higher order Vermas.
Determining the block structure of $\mathcal{O}^{\mathcal{H}, [\lambda]}$
(see \cite{X1}).
\end{itemize}

\item This question is more speculative. It involves
Theorem~\ref{T1storder}, which characterizes the first order
integrability of a highest weight module $V$, i.e.\ the set $J_V$, in
terms of the convex hull of the set of weights -- or in other words, the
extremal rays and edges of the $W_{J_V}$-invariant polyhedral shape
$\conv (\wt V)$. Similarly, the zeroth order integrability of $V$ is the
vertex $\lambda$. It could be interesting to examine what ``higher order
geometric combinatorics'' emerges out of the higher order integrabilities
$(\mathcal{H}_V^{\min})_{\leq k}$ (see Section~\ref{Sspeculate},
including Definition~\ref{Dhigher}).

In parallel, one can ask if Proposition~\ref{PKL} can be extended to all
blocks $\mathcal{O}^{\mathcal{H}, [\lambda]}$ over arbitrary
$\mathfrak{g}$. Namely, if the categories $\mathcal{O}^{\mathcal{H}}$
lead to ``refined'' Kazhdan--Lusztig combinatorics and (quotient) spaces
associated to Hecke modules -- working internally in
$\mathcal{O}^{\mathcal{H}}$ rather than first converting to Verma
characters via BGG resolutions -- which specialize to relative/parabolic
versions (e.g.\ \cite{CC,Deodhar}) at first order.
\end{enumerate}
%}}}

%{{{1 Appendix A: Proof of enough projectives -- details
\appendix
\section{Proof of enough projectives -- details}\label{Sappendix}

Here we include details of the proof of a part of Theorem~\ref{T4}:
$\mathcal{O}^{\mathcal{H}}$ has enough projectives. As said above, we do
so because there are multiple standard objects for the simple objects
(which do get used in BGG reciprocity, see Section~\ref{SBGG}).
We thus list and verify the five conditions in \cite[Theorem 3.2.1]{BGS}.
The setting in \cite{BGS} does hold: by Lemma~\ref{Lbasic}, $\mathcal{A}$
is an abelian subcategory of $\mathbb{C}$-vector spaces.
\begin{enumerate}
\item Every object of $\mathcal{A} \subset \mathcal{O}$ has finite
length.

\item There are only finitely many simple isoclasses in $\mathcal{A}
\subset \mathcal{O}^{W \bullet \lambda}$. Index these by $S$.

\item The endomorphisms of all simple objects in $\mathcal{A} \subset
\mathcal{O}$ are scalars.
\end{enumerate}

 For every lower-closed
subset $T \subset S$ (if $s \in T, s' \leq s$ then $s' \in T$), let
$\mathcal{A}_T$ be the full subcategory of objects in $\mathcal{A}$, all
of whose simple subquotients are $L(\mu)$ with $\mu \in T$. E.g.,
$(\leq_\mu) := \{ s \in S : s \leq \mu \}$ for $\mu \in S$, and similarly
$(<_\mu)$.
Finally, define the (co)standard objects in $\mathcal{A}_{(\leq_\mu)}$
for $\mu \in S$ via Proposition \ref{Psimple}:
\begin{equation}
\mathcal{H}'_\mu := \{ J_\mu \cap H : H \in \mathcal{H}^{\min} \}, \qquad
\Delta(\mu) := \mathbb{M}(\mu, \mathcal{H}'_\mu), \qquad
\nabla(\mu) := \mathbb{M}(\mu, \mathcal{H}'_\mu)^\vee.
\end{equation}
Moreover, since $L(\mu)^\vee \cong L(\mu)$, one has canonical maps
(up to scaling)
\[
\pi : \Delta(\mu) \twoheadrightarrow L(\mu), \qquad
\pi^\vee : L(\mu) \hookrightarrow \nabla(\mu).
\]
\begin{enumerate}
\setcounter{enumi}{4}
\item The objects $K := \ker(\pi)$, ${\rm coker}(\pi^\vee) = K^\vee$ lie
in $\mathcal{A}_{(<_\mu)}$, in that their simple subquotients are all of
the form $L(\nu)$ with $\nu < \mu$. (The assertion for $K^\vee$ follows
from that for $K$, since restricted duality is an exact contravariant
functor on $\mathcal{O}$, hence on $\mathcal{A}$ via Lemma~\ref{Lbasic}.)
\end{enumerate}

It remains to verify the fourth condition:
\begin{enumerate}
\setcounter{enumi}{3}
\item \textit{Fix a lower-closed subset $T \subset S$ and a maximal
element $\mu \in (T, \leq)$. Then in $\mathcal{A}_T$, $\Delta(\mu)
\twoheadrightarrow L(\mu)$ is a projective cover and $L(\mu)
\hookrightarrow \nabla(\mu)$ is an injective hull.}
\end{enumerate}

The exactness and contravariance of restricted duality on $\mathcal{A}_T$
shows that the assertion for $\nabla(\mu)$ follows from that for
$\Delta(\mu)$. Thus, it suffices to show that $\Delta(\mu)$ is a
projective cover of $L(\mu)$.

Begin by noting that $\Delta(\mu) = \mathbb{M}(\mu, \mathcal{H}'_\mu)$ is
indecomposable. Moreover, $N := \ker (\pi : \Delta(\mu)
\twoheadrightarrow L(\mu))$ is indeed a superfluous submodule of
$\Delta(\mu)$, since it does not intersect the highest weight line
$\Delta(\mu)_\mu$ which generates $\Delta(\mu)$.
Thus, it remains to show that $\Delta(\mu)$ is a projective object in
$\mathcal{A}_T$ if $\mu$ is maximal in $(T, \leq) \subset (S, \leq)$. We
do so by showing that the functor
${\rm Hom}_{\mathcal{A}_T}(\Delta(\mu), -)$
is exact. More strongly, denote by $M_\mu$ the $\mu$-weight space of $M$,
and assert the functorial isomorphism
\[
{\rm Hom}_{\mathcal{A}_T}(\Delta(\mu), M) \cong M_\mu,
\qquad M \in \mathcal{A}_T.
\]

Indeed, fix a highest weight vector $v_\mu \in \Delta(\mu)_\mu$, and send
$\varphi \in {\rm Hom}_{\mathcal{A}_T}(\Delta(\mu), M)$ to
$\varphi(v_\mu) \in M_\mu$, for $M \in \mathcal{A}_T$. Clearly this
assignment is linear and injective, and we now prove the surjectivity.
Given $M \in \mathcal{A}_T \subset \mathcal{O}^{\mathcal{H}} =
\mathcal{O}^{\mathcal{H}^{\min}}$, choose $0 \neq m_\mu \in M_\mu$. We
claim that the corresponding map
\[
\varphi : \Delta(\mu) = \mathbb{M}(\mu, \mathcal{H}'_\mu) \to M, \qquad F
\cdot v_\mu \mapsto F \cdot m_\mu, \ F \in U(\mathfrak{g})
\]
is indeed a $\mathfrak{g}$-module map. Since $\mu$ is maximal in $T$, it
suffices to verify that for all $H \in \mathcal{H}^{\min}$,
\begin{equation}\label{Eproj}
{\bf f}^{(\mu)}_{J_\mu \cap H} \cdot M_\mu = 0, \qquad \text{where we
recall,} \qquad {\bf f}^{(\mu)}_{J_\mu \cap H} := \prod_{h \in J_\mu \cap
H} f_h^{\langle \mu, \alpha_h^\vee \rangle + 1}.
\end{equation}

By the definition of $\mathcal{O}^{\mathcal{H}}$, ${\bf f}_H = {\bf
f}^{(0)}_H$ acts nilpotently on $m_\mu \in M_\mu$, so there exists $n$
such that
\[
0 = {\bf f}_H^n v_\mu = {\bf f}^n_{J_\mu \cap H} {\bf f}^n_{H \setminus
J_\mu} v_\mu.
\]
Acting by $e_h^n$ for all $h \in H \setminus J_\mu$, it follows via
\eqref{Esl2} as above that
${\bf f}^n_{J_\mu \cap H} m_\mu = 0$.
Now define
\[
m_h = \langle \mu, \alpha_h^\vee \rangle + 1, \ \forall h \in J_\mu,
\qquad H_1 := \{ h \in J_\mu \cap H : n \geq m_h \}.
\]
Repeating the proof of Lemma~\ref{L3}, one obtains:
(i)~$H_1$ is nonempty, and
(ii)~${\bf f}^{(\mu)}_{H_1} m_\mu = 0$.
This implies~\eqref{Eproj}; hence, $\Delta(\mu)$ is projective in
$\mathcal{A}_T$. Thus by \cite[Theorem 3.2.1]{BGS}, $\mathcal{A}$ and
hence $\mathcal{O}^{\mathcal{H}}$ has enough projectives. Using the
properties of restricted duality, $\mathcal{O}^{\mathcal{H}}$ also has
enough injectives. \qed
%}}}

%{{{1 Appendix B: Proof of Theorem (no-holes)
\section{Proof of Theorem~\ref{Tnoholes}}\label{Anoholes}

We end with a self-contained proof of Theorem~\ref{Tnoholes} (shown
in \cite{MDWF}) for completeness. (Note, the converse to that theorem
follows from Lemma~\ref{Lkostant}.) Our proof uses a preliminary lemma:

\begin{lemma}[{\cite[Lemma 3.2]{MDWF}}]\label{Lidea}
Fix arbitrary $\mathfrak{g}, \lambda$, and $V \neq 0$. If $\mu \in \wt V
\cap (\lambda - \mathbb{Z}_{\geq 0} \Pi_{J_\lambda})$, then
\[
\mu - \sum_{i \in I \setminus J_\lambda} c_i \alpha_i \in \wt V, \qquad
\text{for all } c_i \in \mathbb{Z}_{\geq 0}, \, i \not\in J_\lambda.
\]
\end{lemma}

In other words, $\wt V = \wt_{J_\lambda} V - \mathbb{Z}_{\geq 0}
\Pi_{J_\lambda^c}$, where $\wt_J V := \wt V \cap (\lambda - \mathbb{Z}_{\geq 0}
\Pi_J)$ for $J \subset I$.

\begin{proof}
This follows the approach in the proofs of \cite[Theorem D]{Kh1} and
\cite[Theorem 5.8]{DK}. Fix a highest weight vector $0 \neq v_\lambda \in
V_\lambda$; since $V_\mu \neq 0$, there exists a PBW monomial $F \in
U(\mathfrak{n}^-)_{\mu - \lambda}$ such that $F v_\lambda \neq 0$. Fix an
ordering of $I \setminus J_\lambda$; we now claim that
\[
v_{\bf c} := F \prod_{i \in I \setminus J_\lambda} f_i^{c_i} \cdot
v_\lambda \in V_{\mu - \sum_{i \not\in J_\lambda} c_i \alpha_i}
\]
is nonzero, which proves the lemma. Indeed, via the relations $[e_i, f_j]
= 0$ for $i \neq j$ (which also imply that $F$ commutes with $e_i$ for $i
\not\in J_\lambda$),
\[
\prod_{i \in I \setminus J_\lambda} e_i^{c_i} \cdot v_{\bf c}
= F \prod_{i \in I \setminus J_\lambda} e_i^{c_i} 
\prod_{i \in I \setminus J_\lambda} f_i^{c_i}  \cdot v_\lambda = F
\prod_{i \in I \setminus J_\lambda} (e_i^{c_i} f_i^{c_i}) \cdot
v_\lambda,
\]
and this is nonzero by $\mathfrak{sl}_2$-theory, since $\mathbb{C}[f_i]
v_\lambda$ is a simple Verma $U(\mathfrak{sl}_{\alpha_i})$-module for $i
\not\in J_\lambda$.
\end{proof}

With Lemma~\ref{Lidea} in hand, we have:

\begin{proof}[Proof of Theorem~\ref{Tnoholes}]
Suppose $\mathcal{H}_V = \emptyset$ and $\mu \leq \lambda$ is a weight of
$M(\lambda)$. We show that $\mu \in \wt V$ by induction on ${\rm ht}
(\lambda - \mu) \geq 0$. The base case is obvious, since $V \neq 0$.

For the induction step, write $\mu = \lambda - \sum_{l \in \supp(\lambda
- \mu)} c_l \alpha_l$ with $\supp \left( \sum_{i \in I} c_i \alpha_i
\right) := \{ i : c_i \neq 0 \}$ as usual. Thus $c_l \in \mathbb{Z}_{>0}$
here. Suppose the Dynkin subdiagram induced on the subset $I' :=
\supp(\lambda - \mu) \cap J_\lambda$ has no edge (e.g.\ if $I'$ is empty
or a singleton). Then $\lambda - \sum_{l \in I'} c_l \alpha_l \in \wt V$,
since $\mathcal{H}_V$ is empty. But then $\mu \in \wt V$ by
Lemma~\ref{Lidea}.

Otherwise, $I'$ has some edges, say between nodes $i, i' \in I'$ (among
others). We may assume $c_{i'} \geq c_i > 0$; thus ${\rm ht}(\lambda -
\mu - c_i \alpha_i) < {\rm ht}(\lambda - \mu)$, and so $\mu + c_i
\alpha_i \in \wt V$ by the induction hypothesis. Now write
\[
\mu' := \mu + c_i \alpha_i,
\]
and fix a nonzero weight vector $v_{\mu'} \in V_{\mu'}$. We will employ
$\mathfrak{sl}_{\alpha_i}$-theory on the module
$V_i := U(\mathfrak{sl}_{\alpha_i}) \cdot v_{\mu'}$.
Begin by observing that the ``$\alpha_i$-weight'' of $v_{\mu'} \in V_i$
is a nonnegative integer:
\[
\langle \lambda, \alpha_i^\vee \rangle \geq 0 \geq \langle \alpha_l,
\alpha_i^\vee \rangle\ \forall l \neq i,i', \ \ \langle \alpha_{i'},
\alpha_i^\vee \rangle \leq -1 \quad \implies \quad
\langle \mu + c_i \alpha_i, \alpha_i^\vee \rangle \geq c_{i'} \geq c_i > 0.
\]

Thus $s_i(\mu + c_i \alpha_i) = s_i(\mu')$ equals $\mu + c_i \alpha_i - n
\alpha_i$, where $n \geq c_i$. This shows that $s_i(\mu') \leq \mu'$, and
also that the root string $[s_i(\mu'), \mu']$ between them contains
$\mu$.

Now since $V_i := U(\mathfrak{sl}_{\alpha_i}) \cdot v_{\mu'} \in
\mathcal{O}_{\mathfrak{sl}_{\alpha_i}}$ and $\mu' \in \wt V_i$, it
follows that $s_i(\mu') \leq \mu'$ also lies in $\wt V_i$, hence so does
the entire $\alpha_i$-root string between them. (This holds for all
objects in $\mathcal{O}_{\mathfrak{sl}_{\alpha_i}}$ because it holds for
every highest weight $\mathfrak{sl}_{\alpha_i}$-module -- which is either
a Verma or integrable by $\mathfrak{sl}_2$-theory.) In particular,
\[
\mu \in [ s_i(\mu'), \mu' ] \subset \wt V_i \subset \wt V,
\]
and the proof is complete.
\end{proof}
%}}}
{\color{black}
\section{SageMath code for character numerator in the proof of Proposition \ref{Prop A3 second order holes character}}\label{Section Sagemath code}

We are in type $A_3$ here.
We replace exponentials $e^{-\alpha_1}, e^{-\alpha_2}, e^{-\alpha_3}  \text{ by symbols } x,y,z$, and work in the ring of polynomials $\mathbb{Q}[x,y,z]$.  \\
{\bf Code for the program}\\
\begin{lstlisting}
PR = QQ['x,y,z']
x, y, z = PR.gens()
F = -1
for a in range(0, 5):
  for b in range(a, 5):
    for c in range(b, 5):
      F += x^a*y^b*z^c
      F += z^a*y^b*x^c
G = (1-x)*(1-y)*(1-z)*(1-x*y)*(1-y*z)*(1-x*y*z)
print(F*G)
\end{lstlisting}
\smallskip
\textbf{Output}
\[
2x^7 y^8 z^7 - x^7 y^8 z^6 - 2x^7 y^7 z^7 - x^6 y^8 z^7 + x^7 y^7 z^5 + x^7 y^6 z^6 + x^6 y^6 z^7 + x^5 y^7 z^7 - x^7 y^6 z^5
\]
\[
 - x^5 y^6 z^7 - x^7 y^8 z^2 - x^2 y^8 z^7 + x^6 y^8 z^2 + x^2 y^8 z^6 + x^7 y^7 z + xy^7 z^7 - x^5 y^7 z - xy^7 z^5
\]
\[
 - x^6 y^6 - y^6 z^6 + x^5 y^6 + x^7 y^2 z^2 + y^6 z^5 + x^2 y^2 z^7 - x^7 y^2 z  \boldsymbol{-x^3 y^4 z^3}  - xy^2 z^7
\]
\[
 - x^6 y z^2 \boldsymbol{+ x^3 y^4 z^2} \boldsymbol{+ x^3 y^3 z^3} \boldsymbol{+ x^2 y^4 z^3} - x^2 y z^6  \boldsymbol{- x^2y^4z^2} + x^6y  \boldsymbol{- x^3y^3z} \boldsymbol{- x^3y^2z^2}
\]
\[
 \boldsymbol{- x^2y^2z^3} \boldsymbol{- xy^3z^3} + yz^6 + x^5z \boldsymbol{+ x^3y^2z} \boldsymbol{+ xy^2z^3} + xz^5 - x^5 \boldsymbol{+ xy^3z} \boldsymbol{+ x^2yz^2} - z^5  \boldsymbol{-xz}  \boldsymbol{-y}  \boldsymbol{+1}
\]
The highlighted 16 monomials $x^ay^bz^c \leftarrow\!\rightarrow e^{-a\alpha_1 -b\alpha_2 -c\alpha_3}$ are exactly all the weights that satisfy \eqref{Eqn for X,Y,Z} at $(X,Y,Z)=(a,b,c)$. 
If we increase the number of iterations, i.e., the ranges for $a,b,c$ from $(0,5)$ to various $(0, k)$ for $k\geq 6$, we see the 18 monomials at the end of the output-sum for each $k$, to be these 16 monomials plus $x^k, z^k$.
Moreover, in the above output for $k=5$, the monomials other than the 16 character-terms, can be seen to  successively vanish with increasing values of $k$.
}

%}}}

\end{document}